\documentclass[11pt,reqno]{amsart}

\usepackage{dsfont}
\usepackage{amssymb, amsmath, amsthm}
\usepackage{enumerate,color}
\newtheorem{thm}{Theorem}[section]
\newtheorem{lem}[thm]{Lemma}
\newtheorem{cor}[thm]{Corollary}
\newtheorem{prop}[thm]{Proposition}

\newtheorem{rmk}{Remark}

\numberwithin{equation}{section}

\newcommand{\bel}{\begin{equation} \label}
\newcommand{\ee}{\end{equation}}
\def\beq{\begin{equation}}
\def\eeq{\end{equation}}
\newcommand{\bea}{\begin{eqnarray}}
\newcommand{\eea}{\end{eqnarray}}
\newcommand{\beas}{\begin{eqnarray*}}
\newcommand{\eeas}{\end{eqnarray*}}
\newcommand{\pd}{\partial}

\newcommand{\dd}{\mbox{d}}

\newcommand{\re}{\mathfrak R}

\newcommand{\R}{\mathbb{R}}
\newcommand{\C}{\mathbb{C}} 
\newcommand{\N}{\mathbb{N}}

\newcommand{\cA}{\mathcal{A}}
\newcommand{\cB}{\mathcal{B}}
\newcommand{\cC}{\mathcal{C}}

\newcommand{\cS}{\mathcal{S}}

\newcommand{\supp}{\mathrm{supp}\,}  %supp

\allowdisplaybreaks
\def\epsilon{\varepsilon}
\def\phi {\varphi}

\providecommand{\abs}[1]{\left\lvert#1\right\rvert}
\providecommand{\norm}[1]{\left\lVert#1\right\rVert}

\renewcommand{\leq}{\leqslant}
\renewcommand{\geq}{\geqslant}

\providecommand{\abs}[1]{\left\lvert#1\right\rvert}
\providecommand{\norm}[1]{\left\lVert#1\right\rVert}
\def\thefootnote{{}}

\newcommand{\overbar}[1]{\mkern 1.5mu\overline{\mkern-1.5mu#1\mkern-1.5mu}\mkern 1.5mu}
\usepackage[font=small]{subcaption}
\usepackage[font=small]{caption}
\usepackage{graphicx}
\newtheorem{example}{Example}[section]
\usepackage{booktabs}
\title{\bf Identification of time-varying source term in time-fractional evolution equations}

\author{
Yavar Kian$^1$,
\'Eric Soccorsi$^1$,
Qi Xue$^2$
and Masahiro Yamamoto$^{3,4,5}$
}

\begin{document}

\begin{abstract} This paper is concerned with the inverse problem of determining the time and space dependent source term of diffusion equations with constant-order time-fractional derivative in $(0,2)$. We examine two different cases. In the first one, the source is the product of 
a spatial term and a temporal term, and we prove that the term depending on the space variable  can be retrieved  by  observation over the time interval of the solution on an arbitrary sub-boundary. Under some suitable assumptions we can also show the simultaneous recovery of the spatial term and the temporal term. In the second case, we assume that the first term of the product varies with one fixed space variable, while the second one is a function of all the remaining space and time variables,  and we show that they are uniquely determined by one arbitrary lateral measurement of the solution. These source identification results boil down to a weak unique continuation principle in the first case and a unique continuation principle for Cauchy data in the second one, that are preliminarily established. Finally, numerical reconstruction of the spatial term in the first case is carried out through an iterative algorithm based on the Tikhonov regularization method.\\

\medskip
\noindent
{\bf  Keywords:} Inverse source problems, diffusion equation, fractional diffusion equation, uniqueness result, numerical reconstruction, Tikhonov regularization method. \\

\medskip
\noindent
{\bf Mathematics subject classification 2010 :} 35R30, 	35R11.

\end{abstract}

\maketitle

\renewcommand{\thefootnote}{\fnsymbol{footnote}}
\footnotetext{\hspace*{-5mm} 
\begin{tabular}{@{}r@{}p{13cm}@{}} 
%& Manuscript last updated: \today.
\\
$^1$& Aix Marseille Universit\'e, Universit\'e de Toulon, CNRS, CPT, Marseille, France;
Email: yavar.kian@univ-amu.fr; Email: eric.soccorsi@univ-amu.fr\\
$^2$& ISTerre, CNRS \& Univ. Grenoble Alpes, CS 40700 38058 Grenoble cedex 9 France; Email: qi.xue@univ-grenoble-alpes.fr\\
$^3$& Graduate School of Mathematical Sciences, 
the University of Tokyo, 
3-8-1 Komaba, Meguro-ku, Tokyo 153-8914, Japan; \\
$^4$&Honorary Member of Academy of Romanian Scientists, Splaiul Independentei
Street, no 54, 050094 Bucharest Romania;\\
$^5$&Peoples' Friendship University of Russia (RUDN University) 6 Miklukho-Maklaya
St, Moscow, 117198, Russia;
E-mail:
myama@next.odn.ne.jp
\\

 \end{tabular}}
 
 \tableofcontents

%%%%%%%%%%%%%%%%%%%%%%%%%%%%%%%%%%%%%%%%%%%%%%%%%%%%%%%%%%%%%%%%%
%%%%%%%%%%%%%%%%%%%%%%%%%%%%%%%%%%%%%%%%%%%%%%%%%%%%%%%%%%%%%%%%%
%%%%%%%%%%%%%%%%%                                    Introduction                             %%%%%%%%%%%%%%%%%%%%%%
%%%%%%%%%%%%%%%%%%%%%%%%%%%%%%%%%%%%%%%%%%%%%%%%%%%%%%%%%%%%%%%%%
%%%%%%%%%%%%%%%%%%%%%%%%%%%%%%%%%%%%%%%%%%%%%%%%%%%%%%%%%%%%%%%%%

\section{Introduction}
\label{sec-intro}

\subsection{Settings}
\label{sec-settings}

Let $\Omega$ be a bounded and connected open subset of $\R^d$, $d \geq 2$, with $C^{2}$ boundary $\partial \Omega$. Given $a:=(a_{i,j})_{1 \leq i,j \leq d} \in \cC^1(\overline{\Omega};\R^{d^2})$, symmetric, i.e.,
$$ a_{i,j}(x)=a_{j,i}(x),\ x \in \Omega,\ i,j = 1,\ldots,d, $$
and fulfilling the ellipticity condition
\bel{ell}
\exists c>0,\ \sum_{i,j=1}^d a_{i,j}(x) \xi_i \xi_j \geq c |\xi|^2,\ x \in \Omega,\ \xi=(\xi_1,\ldots,\xi_d) \in \R^d,
\ee
we introduce the formal differential operator  
$$ \cA_0 u(x) :=-\sum_{i,j=1}^d \partial_{x_i} \left( a_{i,j}(x) \partial_{x_j} u(x) \right),\  x:=(x_1,\ldots,x_d) \in\Omega, $$
where $\partial_{x_i}$ denotes the partial derivative with respect to $x_i$, $i=1,\ldots,d$. 
We perturb $\cA_0$ by a potential function $q \in L^\kappa(\Omega)$, $\kappa \in (d,+\infty]$, that is lower bounded by some positive constant, 
\bel{a9}
\exists r \in (0,+\infty),\ q(x)\geq r ,\ x \in \Omega,
\ee
and define the operator $\cA_q:=\cA_0 + q$, where the notation $q$ is understood as the multiplication operator by the corresponding function. 

Next, for $T\in(0,+\infty)$, $\alpha\in(0,2)$ and
$\rho \in L^\infty(\Omega)$ obeying
\bel{eq-rho}
 0<\rho_0 \leq\rho(x) \leq\rho_M <+\infty,\ x \in \Omega, 
\ee
we consider the following initial boundary value problem (IBVP) with source term $f \in L^1(0,T;L^2(\Omega))$,
\bel{eq1}
\left\{ \begin{array}{rcll} 
(\rho(x) \partial_t^{\alpha}+\cA_q) u(t,x) & = & f(t,x), & (t,x)\in  Q:=(0,T) \times \Omega,\\
B_\star u(t,x) & = & 0, & (t,x) \in \Sigma := (0,T) \times \pd \Omega, \\  
\pd_t^k u(0,x) & = & 0, & x \in \Omega,\ k=0,\ldots,N_\alpha,
\end{array}
\right.
\ee
where 
$$ N_\alpha:= \left\{ \begin{array}{cl} 0 & \mbox{if}\ \alpha \in (0,1], \\  1 & \mbox{if}\ \alpha \in (1,2), \end{array} \right. $$
and $\partial_t^\alpha$ denotes the fractional Caputo derivative of order $\alpha$ with respect to $t$, defined by
\bel{cap} 
\pd_t^\alpha u(t,x):=\frac{1}{\Gamma(N_{\alpha}+1-\alpha)}\int_0^t(t-s)^{N_{\alpha}-\alpha}\pd_s^{N_{\alpha}+1} u(s,x) \dd s,\ (t,x) \in Q,
\ee
when $\alpha \in (0,1) \cup (1,2)$, while $\partial_t^\alpha$ is the usual first order derivative $\pd_t$ when $\alpha=1$. 
In the second line of \eqref{eq1}, $B_\star$ is either of the following two boundary operators:
\begin{enumerate}[(a)] 
\item $B_\star u:=u$,
\item $B_\star u:= \partial_{\nu_a} u$, where $\partial_{\nu_a}$ stands for the normal derivative with respect to $a=(a_{i,j})_{1 \leq i,j \leq d}$, expressed by
$$\partial_{\nu_a }w(x) := \sum_{i,j=1}^d a_{i,j}(x) \partial_{x_j} w(x) \nu_i(x),\  x \in \partial\Omega, $$
and $\nu=(\nu_1,\ldots,\nu_d)$ is the outward unit normal vector to $\partial\Omega$. 
\end{enumerate}
Otherwise stated, the IBVP \eqref{eq1} is endowed with  the homogeneous Dirichlet (resp., Neumann) boundary condition when $B_\star$ is given by (a) (resp., (b)). 

In the present paper we use the definition of  a weak solution to Problem \eqref{eq1} which is given by \cite{KY,KY3,SY} in a more general framework  by mean of the Laplace transform  with respect to the time variable (see \cite[Definition 1.1]{KY}). 
Here and in the remaining part of this paper, we use the notation $v(t,\cdot)$ as shorthand for the function $x \mapsto v(t,x)$.

According to \cite[Theorem 2.3]{KY3}, the weak solution to \eqref{eq1} exists and is unique within the class $L^1(0,T;L^2(\Omega))$, and that it enjoys the Duhamel representation formula given in Section \ref{sec-analytic}. We refer the reader to \cite{KY,KY3,KSY,SY} for the existence and the uniqueness issue of such a solution to \eqref{eq1}, as well as for its classical properties. We point out that for $\alpha=1$, the weak solution to \eqref{eq1} coincides with the classical variational $\cC^1([0,T];H^{-1}(\Omega))\cap \cC([0,T];H^1(\Omega))$-solution to the corresponding parabolic equation.

In this paper we examine the inverse problem of determining the source term $f$ appearing in the first line of \eqref{eq1}, from either internal or lateral measurement of the weak solution $u$ to \eqref{eq1}. It turns out that this  inverse problem is ill-posed in the sense that the above data do not uniquely determine $f$,  and we refer the interested reader to the Appendix where this issue is discussed in greater detail. As a consequence, the inverse source problem under investigation has to be reformulated. Different lines of research can be pursued. One possible direction is the one of assuming that the unknown function $f : Q \to \R$ depends on a restricted number of parameters of $(t,x) \in Q$. Another direction is the one of considering source terms with separated variables. In this paper, we follow the second direction,  that is, we suppose that the unknown source term belong to either of the two following sets.

 The first class of source terms under consideration takes the form 
\bel{source1} 
f(t,x)= \sigma(t) g(x),\ (t,x) \in Q,
\ee
and the corresponding inverse problem is as follows.

\textbf{Inverse Problem 1} (IP1): Determine the spatially dependent function $g$ and/or the temporal dependent function $\sigma$ from knowledge of $u_{| (0,T) \times \Omega'}$, where $u$ denotes the solution to \eqref{eq1} associated with source term $f$ given by \eqref{source1}, and $\Omega'$ is an  arbitrarily chosen non-empty open subset of $ \Omega$.

 In the same spirit, we also consider sources obtained by superposition of two source terms in the form of \eqref{source1}, that is
\bel{source1p} 
f(t,x):=\sigma(t) g(x)+\beta(t)h(x),\ (t,x) \in Q,
\ee
which suggests investigating the following inverse problem:

\textbf{Inverse Problem 1'} (IP1'): Assuming that  $\sigma$ and $\beta$ are known, retrieve the two spatially dependent functions $g$ and $h$ from knowledge of $u_{| (0,T) \times \Omega'}$, where, $u$ and $\Omega'$ are the same as in (IP1).

The second set of source terms is made of functions which do not dependent upon the last component $x_d$ of the space variable $x=(x_1,\ldots,x_{d-1},x_d) \in \Omega$. More precisely, given $L \in (0,+\infty)$ and 
 an open bounded subset $\omega \subset \R^{d-1}$,  we set 
$\Omega_0:=\omega \times (-L,L)$  and we assume that
\bel{cy}
\overline{\Omega_0} \subset \Omega.
\ee
We consider a source term $f$ expressed by
\bel{source2} 
f(t,x',x_d):=\left\{ \begin{array}{cl}g(t,x') h(x_d) & \mbox{if}\ (t,x',x_d) \in Q_0 := (0,T) \times \Omega_0, \\
0 & \mbox{if}\ (t,x',x_d) \in Q \setminus Q_0, \end{array} \right.
\ee
where $x'=(x_1,\ldots,x_{d-1})$ denotes the $d-1$ first components of $x$. We investigate the following inverse problem.

\textbf{Inverse Problem 2} (IP2):  Given  an arbitrary non-empty open subset $\gamma$ of $ \partial\Omega$, determine $h$ and/or $g$ by $( u_{| (0,T) \times \gamma},\partial_\nu u_{| (0,T) \times \gamma})$. Here $u$ is the solution to \eqref{eq1} associated with $f$ given by \eqref{source2} and $\partial_\nu u := \nabla u \cdot \nu$.

\subsection{Motivations}
Depending on  whether $\alpha=1$ or $\alpha \in (0,1) \cup (1,2)$, the system \eqref{eq1} models typical or anomalous diffusion phenomena appearing in several areas of applied sciences, such as geophysics, environmental science and biology, see, e.g., \cite{JR,NSY}. In this context, sub-diffusive (resp., super-diffusive) processes are described by \eqref{eq1} with $\alpha \in (0,1)$ (resp., $\alpha \in (1,2)$), and the kinetic equation \eqref{eq1} may be seen as a corresponding macroscopic model to microscopic diffusion phenomena driven by continuous time random walk, see, e.g., \cite{MK}.  Thus, inverse problem (IP1) (resp., (IP2)) is to know whether time and space varying source terms can be retrieved by  internal (resp, lateral) data, in  the presence of typical or anomalous diffusion.  For instance, such problems occur in the context of underground diffusion of contaminants, see \cite{JR,NSY}, and we point out that they can be adapted to the recovery of moving sources  as in \cite{Ki}.

\subsection{A short review of inverse source problems}

Inverse problems are generally nonlinear in the sense that the unknown parameter of the problem depends in a nonlinear way on the data,  although the direct model is linear in $u$. For instance, this is the case for inverse coefficients problems or inverse spectral problems, see, e.g., \cite{I,LLY2}. However, this is no longer true for inverse source problems as the dependence of the unknown source term is linear with respect to the (internal or lateral) data. When this remarkable feature of inverse source problems does not guarantee that they are easy to solve, it certainly does explain why they have become increasingly popular among the mathematical community. 

This is particularly true when typical diffusion is considered, i.e. when $\alpha=1$ in \eqref{eq1}, where the inverse problem of determining a time independent source term has been extensively studied by several authors in
\cite{choY,cho,KSS,yam,yam1}, the list being non exhaustive, and also in \cite{IY},  which relies on the Bukhgeim-Klibanov method introduced 
in \cite{BK}. As for inverse time independent source problems with $\alpha \in (0,1) \cup (1,2)$, we refer the reader to \cite{JLLY}, and to \cite{JR0,JR,KOSY, KSY,LKS, RZ} for inverse coefficient problems  for anomalous diffusion equations.

In all the above mentioned inverse source results, the source term was  time-independent. The stability issue in determining the temporal source term of time-fractional diffusion equations was examined in \cite{FK,SY}, and in the same context, the time and space dependent factor of suitable source terms is reconstructed in \cite{KY2}. As for the determination of time dependent sources in parabolic equations, we refer the reader to \cite{AE,EH,Ik,KW}, and to \cite{BHKY,HK,HKLZ} for the same problem with hyperbolic equations. 

Let us now collect the main achievements of this article in the coming section.

\subsection{Main results}
\label{sec-mainres}
We start by stating the following weak uniqueness result for the IBVP \eqref{eq1} associated with (IP1).
\begin{thm}
\label{t1} 
Let $\sigma \in L^1(0,T)$ be supported in $[0,T)$ and assume that $f$ is given by \eqref{source1}, where $g \in L^2(\Omega)$. 
Denote by $u$ the weak solution to \eqref{eq1}. Then, for all $\alpha\in(0,2)$ and  any  non-empty open subset $\Omega' \subset \Omega$, we have:
$$  u_{| (0,T) \times \Omega'}=0  \Longrightarrow  f=0\ \mbox{in}\ Q. $$
\end{thm}

As a corollary, we have the following unique identification result which provides a positive answer to (IP1).

\begin{cor}
\label{c1} 
For $j=1,2$, let $\sigma_j \in L^1(0,T)$  satisfy
$\supp \sigma_j \subset [0,T)$ and let $g_j \in L^2(\Omega)$.
Assume that either of the two following conditions is fulfilled
\begin{enumerate}[(i)]
\item $\sigma_1=\sigma_2$  and is not identically zero in $(0,T)$,
\item $g_1=g_2$  and is not identically zero in $\Omega$,
\end{enumerate}
and let $u_j$ denote the solution to \eqref{eq1} with $f=f_j$, where 
$$f_j(t,x):=\sigma_j(t) g_j(x),\ (t,x) \in Q. $$
Then, the following implication holds for any non-empty open subset $\Omega' \subset \Omega$:
$$
 u_1=u_2\ \mbox{in}\ (0,T) \times \Omega'  \Longrightarrow  \sigma_1= \sigma_2\ \mbox{in}\ (0,T)\ \mbox{and}\  g_1= g_2\ \mbox{in}\  \Omega .
$$
\end{cor}

Actually, Corollary \ref{c1} remains valid upon removing the hypothesis on the support of $\sigma_j$, $j=1,2$, which was inherited from Theorem \ref{t1}, but this is at the expense of a greater regularity assumption on these two functions.  We refer the reader to Theorem \ref{t1bis} in Section \ref{sec-or} for the statement of the corresponding result.

We turn now to investigating (IP2). Prior to doing so we state the following uniqueness result for  extra Neumann or Dirichlet data.

\begin{thm}
\label{t2} 
 Let $\rho=1$ a.e. in $\Omega$ and set $\mathcal A_q=-\Delta$. 
Assume \eqref{cy} and  suppose that $\Omega \setminus \overline{\Omega_0}$ is connected.
For $\zeta \in\left(\frac{3}{4},1\right)$ and $r\in \left(\frac{1}{\alpha(1-\zeta)},+\infty\right)$, let $g \in L^r(0,T;L^2(\omega))$ be supported in $[0,T)\times\overline{\omega}$ and let $h \in L^2(-L,L)$. Then, for all $\alpha\in(0,2)$ there exists a unique solution $u \in \cC([0,T]; H^{2\zeta}(\Omega))$ to the IBVP \eqref{eq1} with source term $f$ defined by \eqref{source2}. Moreover, for any non-empty subset $\gamma$ of $\partial\Omega$ we have
\bel{t2a} 
 B_\star^*u_{|(0,T)\times\gamma}=0  \Longrightarrow  f=0\ \mbox{in}\ Q,
\ee

where
$$B_\star^* u:= \left\{ \begin{array}{cl} \partial_{\nu} u & \mbox{if}\ B_\star u=u,\\ u & \mbox{if}\ B_\star u= \partial_{\nu} u.
\end{array} \right.$$
\end{thm}

The coming result, which is a byproduct of Theorem \ref{t2}, likewise Corollary \ref{c1} follows from Theorem \ref{t1}, answers the question raised by (IP2).

\begin{cor}
\label{c2} 
 Let  $\rho$, $\mathcal A_q$, $\Omega_0$ and $r$ be the same as in Theorem \ref{t2}.
%Let \eqref{cy} be fulfilled and assume that $\Omega \setminus \overline{\Omega_0}$ is connected. 
For $j=1,2$, let $g_j \in L^r(0,T; L^2(\omega))$  satisfy
$\supp g_j \subset [0,T) \times \overline{\omega}$, let $h_j \in L^2(-L,L)$.  By $u_j$ we denote the solution  to \eqref{eq1} with  source term $f_j$, defined by \eqref{source2} where $(g,h)$ is replaced by $(g_j,h_j)$.

Assume either of the two following conditions
\begin{enumerate}[(i)]
\item $h_1=h_2$,
%\item $\sigma_1(t,\cdot)=\sigma_2(t,\cdot)=\sigma(t,\cdot)$ for all $t\in(0,t_0)$, where $t_0 \in (0,T)$ is fixed and $
%\tilde{\sigma}$ is a $L^2(\omega)$-valued holomorphic function in the half-strip $\cS_\delta$ introduced in 
%Theorem \ref{t1}, which is not-identically zero and fulfills
%\bel{c2a} 
%\norm{\sigma(t,\cdot)}_{L^2(\omega)}\leq (1+ t)^N,\ t\in (0,+\infty),
%\ee
%for some positive constant $C$ and some natural number $N$ that are independent of $t$;
\item $g_1=g_2$.
\end{enumerate}
Then,  for any non-empty relatively open subset $\gamma$ of $\pd \Omega$, 
$$
\partial_\nu^ku_1(t,x)=\partial_\nu^ku_2(t,x),\ (t,x)\in (0,T)\times\gamma,\ k=0,1, 
$$
yields $g_1=g_2$ in $(0,T) \times \omega$ and $h_1=h_2$ in $(-L,L)$.
\end{cor}

 In the same way as Corollary \ref{c1} is changed into Theorem \ref{t1bis}, the uniqueness result of Corollary \ref{c2} translates into the one of Theorem \ref{t2bis} for non-compactly supported sources in the form of \eqref{source2}.

We also point out that the statement of Theorem \ref{t1} (resp., Theorem \ref{t2}) can be adapted to the framework of distributed order fractional diffusion equations, and we refer the reader to Theorem \ref{t5} (resp., Theorem \ref{t6}) in Section \ref{sec-DO} for the corresponding result. \\

The derivation of a source identification result from a uniqueness result such as Theorem \ref{t1}, is rather standard in the analysis of inverse source problems, see, e.g., \cite{choY,cho,HK,JLLY}. The strategy used in these four articles to determine the spatial part of the source \eqref{source1} under the assumption that its temporal part $\sigma$ is known, is to turn the non-homogeneous diffusion equation under study into a homogeneous one by  reducing the source information into the initial data. This requires that $\sigma(0)\neq0$ and $\sigma \in \mathcal C^1([0,T])$. The condition $\sigma(0) \neq 0$ suggests that the source should be switched on before the data are collected,  which is quite unexpected considering that only the spatial part $g$ of the source is retrieved here. 
%Indeed, from a practical point of view, one may wonder why starting the observation early might be a problem for determining $g$. 
It turns out that this  extra condition $\sigma(0) \neq 0$ was removed in \cite[Theorem 2]{HK} for an evolutionary equation over the infinite time range $(0,+\infty)$. As far as we know, Theorem \ref{t1} is one of the few source identification results available in the mathematical literature for a system  evolving on a finite time interval $(0,T)$. Furthermore, we emphasize that the second condition $\sigma\in\mathcal C^1([0,T])$ requested by
\cite{choY,cho,HK,JLLY} is weakened to $\sigma\in L^1(0,T)$ in Theorem \ref{t1} and Corollary \ref{c1}.  Here we need a new approach since the $L^1$-regularity is too weak for  arguing as in \cite{JLLY}. 

In Theorem \ref{t1}, either of the two terms $\sigma$ or $g$ appearing on the right hand side of \eqref{source1}, is retrieved when the other one is known. This result is very similar to the ones of \cite{choY,cho,HK,JLLY}. On the other hand, Theorem \ref{t1bis} claims simultaneous identification of $\sigma$ and $g$, but this is at the expense of greater regularity on $\alpha$ and upon assuming partial knowledge of $\sigma$.  More precisely, it is requested that up to some fixed time $t_0 \in (0,T)$, the function $t \mapsto \sigma(t)$ be known and depend analytically on $t$. In this respect, Theorem \ref{t2} and Corollary \ref{c2} may be seen as an alternative approach to Theorem \ref{t1bis} for recovering a source term depending on the time variable and all the space variables excepting one.

A key ingredient in the derivation of Theorems \ref{t1} and \ref{t2} is the time analyticity of the solution to \eqref{eq1}, exhibited in Proposition \ref{pr1}. While this  analyticity property is a classical feature of the solution  when $\alpha=1$, its  derivation requires a careful treatment for $\alpha \in (0,1) \cup (1,2)$ and is based on the representation formula \cite[Theorem 1.1 and Remark 1]{KSY} of the solution to \eqref{eq1}.

\subsection{Outline}

The two theorems and the two corollaries stated above are  proved in 
Section \ref{sec-p1}. They rely on suitable analytic properties of the solution to \eqref{eq1}, which are established in Section \ref{sec-analytic}. Section \ref{sec-or} contains two identification results  for non-compactly supported time-dependent source terms in the form of \eqref{source1} or \eqref{source2}, which are similar to the ones stated in Corollaries \ref{c1} and \ref{c2}.
In Section \ref{sec-DO}, we adapt the uniqueness results of Theorems \ref{t1} and \ref{t2} to the framework of distributed order diffusion equations. 
In Section \ref{sec-NR} we build an iterative method providing numerical reconstructions of the two-dimensional space-varying part of unknown sources \eqref{source1} when $\alpha \in (1,2)$.
Section \ref{sec-IP1prime} is devoted to the study of (IP1) with $\alpha=1$. 
%Even if for such source terms it is impossible to get the full recovery from our data, we can prove the recovery of %information about such source. These results are stated in Theorem \ref{t7} and Corollary \ref{c5}. 
Finally, in the Appendix we discuss the obstruction to unique determination of the source term in \eqref{eq1}, from either internal or boundary data.

%%%%%%%%%%%%%%%%%%%%%%%%%%%%%%%%%%%%%%%%%%%%%%%%%%%%%%%%%%%%%%%%%
%%%%%%%%%%%%%%%%%%%%%%%%%%%%%%%%%%%%%%%%%%%%%%%%%%%%%%%%%%%%%%%%%
%%%%%%%%%%%%%%%%%                                Direct problem                           %%%%%%%%%%%%%%%%%%%%%%
%%%%%%%%%%%%%%%%%%%%%%%%%%%%%%%%%%%%%%%%%%%%%%%%%%%%%%%%%%%%%%%%%
%%%%%%%%%%%%%%%%%%%%%%%%%%%%%%%%%%%%%%%%%%%%%%%%%%%%%%%%%%%%%%%%%

\section{Direct problem: representation and time-analyticity of the solution}
\label{sec-analytic}

In this section we establish time-analytic properties of the weak solution $u$ to \eqref{eq1}. Their derivation is based on an appropriate representation formula of $u$, that is borrowed from \cite{KSY}.

\subsection{A representation formula}
\label{sec-Duhamel}
Assume \eqref{ell}: Let $\alpha \in (0,2)$, let $\rho \in L^\infty(\Omega)$ fulfill \eqref{eq-rho}, and let $q \in L^\kappa(\Omega)$, with $\kappa \in (d,+\infty]$, satisfy \eqref{a9}. Let $A_q$ be the self-adjoint operator in $L^2(\Omega)$, generated by the closed sesquilinear form  
$$ (u,v) \mapsto  \sum_{i,j=1}^d \left( a_{i,j}(x) \pd_{x_i} u(x) \pd_{x_j} v(x) + q(x) u(x) v(x) \right)dx,\ u, v \in V, $$
where $V:=H_0^{1}(\Omega)$ if $\mathcal{A}_q$ is endowed with a homogeneous Dirichlet boundary condition, while $V:=H^1(\Omega)$ if the boundary condition attached to $\mathcal{A}_q$ is of Neumann type. Note that here we restrict the space $L^2(\Omega)$ to real valued functions. Otherwise stated, $A_q$ is the (positive) self-adjoint operator in $L^2(\Omega)$, acting as $\mathcal{A}_q$ on its domain $D(A_q)$, dense in $L^2(\Omega)$. We denote by $A_0$ the operator $A_q$ when $q=0$ a.e. in $\Omega$. In light of \eqref{ell}-\eqref{a9}, $D(A_q)$ is independent of $q$ (see, e.g., \cite[Section 2.1]{KSY}) and it is embedded in $H^2(\Omega)$:
\bel{dom-A}
D(A_q) = D(A_0) \subset H^2(\Omega).
\ee
Next we introduce the operator $A_{q,\rho}:=\rho^{-1} A_q$, with domain
\bel{dom-Ar}
D(A_{q,\rho}) = D(A_q),
\ee
which is positive and self-adjoint in the weighted-space $L_\rho^2(\Omega) := L^2(\Omega;\rho dx)$. Evidently, $A_{q,\rho}$ is self-adjoint in $L_\rho^2(\Omega)$. For all $p \in \C \setminus \R_-$, the operator $A_q + \rho p^\alpha$, where $A_q$ is defined in Section \ref{sec-mainres}, is boundedly invertible in $L^2(\Omega)$, by virtue of \cite[Proposition 2.1]{KSY}. Moreover, in view of \cite[Eq. (2.4)-(2.5)]{KSY}, the following resolvent estimate
\bel{res-est} 
\norm{(A_q+\rho p^\alpha)^{-1}}_{\cB(L^2(\Omega))} \leq C \abs{p}^{-\alpha},\ p \in \C \setminus \R_-, 
\ee
holds with $C=\rho_0^{-1} \max \left\{ 2 , \sin (\alpha \arctan( (3 \rho_M)^{-1} \rho_0 ) )^{-1} \right\}$.
Here, $(A_q + \rho p^\alpha)^{-1}$ denotes the resolvent operator of $A_q + \rho p^\alpha$ and $\cB(L^2(\Omega))$ is the space of linear bounded operators in $L^2(\Omega)$. 
These two results were established for $\alpha \in (0,1)$ and for the form domain $V=H_0^1(\Omega)$ in \cite{KSY}, but they extend to $\alpha \in [1,2)$ and $V=H^1(\Omega)$ in a straightforward way.

Next, for all $f \in L^1(0,T;L^2(\Omega))$, the weak solution $u$ to \eqref{eq1} reads
\bel{tt1d}
u(t,\cdot)=\int_0^t S(t-s) f(s,\cdot)ds,\ t \in (0,T),
\ee
where we have set
\bel{tt1e}
S(t) h :=\int_{\gamma(\epsilon,\theta)} e^{tp}(A_q+\rho p^\alpha)^{-1}h dp,\ t \in (0,+\infty),\ h \in L^2(\Omega).
\ee
Here, $\epsilon$ is arbitrary in $(0,+\infty)$, $\theta$ can be any angle in $\left( \frac{\pi }{2}, \min \left( \pi, \frac{\pi}{\alpha} \right) \right)$ and $\gamma(\epsilon,\theta)$ is the following modified Haenkel contour in $\C$, 
\bel{g1} 
\gamma(\epsilon,\theta) := \gamma_-(\epsilon,\theta) \cup \gamma_0(\epsilon,
\theta) \cup \gamma_+(\epsilon,\theta),
\ee
where
\bel{g2}
\gamma_\pm(\epsilon,\theta) 
:= \{s e^{\pm i \theta}:\ s \in [\epsilon, +\infty)  \}\ \mbox{and}\ \gamma_0(\epsilon,\theta) := \{ \epsilon e^{i \beta}:\ \beta \in
[-\theta,\theta] \}
\ee 
are traversed in the positive sense. The Duhamel representation formula \eqref{tt1d}-\eqref{tt1e} is a direct consequence of \cite[Theorem 1.1 and Remark 1]{KSY} and the density of $L^\infty(0,T;L^2(\Omega))\cap \cC((0,T],L^2(\Omega))$ in  $L^1(0,T;L^2(\Omega))$. It is our main tool in the derivation of the time-analytic properties of the weak solution $u$ to the IBVP \eqref{eq1}.

\subsection{Time-analyticity}

 Let us now establish that the solution to \eqref{eq1} associated with suitable source term $f$, depends analytically on the time variable. 

\begin{prop}
\label{pr1} 
Assume that $f \in L^1(0,T;L^2(\Omega))$ is supported in $[0,T-3\epsilon_\star]\times\overline{\Omega}$, for some fixed $\epsilon_\star \in \left( 0,\frac{T}{4} \right)$. Then, there exists $\theta_\star \in \left( 0, \min \left( \frac{\pi}{4} , \frac{\pi}{2\alpha} - \frac{\pi}{4} \right) \right)$, such that the weak solution $u$ to \eqref{eq1}, given by \eqref{tt1d}-\eqref{tt1e}, extends to an $L^2(\Omega)$-valued map still denoted by $u$, which is analytic in $ \cC_{\theta_\star}$, where 
$\cC_{\theta_\star}:= \{T-\epsilon_\star+\tau e^{i\psi}:\ \tau \in(0,+\infty),\ \psi \in (-\theta_\star,\theta_\star) \}$.
Moreover, $t \mapsto u(t,\cdot)$ is holomorphic in $\cC_{\theta_\star}$, its Laplace transform $p \mapsto U(p)=\mathcal L[u] (p) :  x\mapsto \int_0^{+\infty}e^{-pt} u(t,x)dt$ is well defined for all $p \in (0,+\infty)$, and each $U(p)$  is a solution to the following BVP
\bel{ell-BVP}
\left\{ \begin{array}{rcll} 
(\cA_q+\rho p^\alpha) U(p) & = & \int_0^Te^{-pt}f(t,\cdot)dt, & \mbox{in}\ \Omega,\\
B_\star U(p) & = & 0, & \mbox{on}\ \pd\Omega.
\end{array}
\right.
\ee
\end{prop}

We point out that the derivation of Proposition \ref{pr1} is similar to the one of \cite[Proposition 3.1]{KLLY}, where fractional diffusion equations with non-homogeneous boundary condition are considered. But, since the boundary condition is homogeneous here, we can simplify the proof presented in \cite{KLLY} as follows.

\paragraph{\bf Proof of Proposition \ref{pr1}.} 
%Let us first consider the case $\alpha \in(0,1)$. 
 Remembering that the weak solution $u$ to \eqref{eq1} is expressed by \eqref{tt1d}-\eqref{tt1e} for some fixed $(\epsilon,\theta) \in (0,1) \times \left( \frac{\pi }{2}, \min \left( \pi , \frac{\pi}{\alpha} \right) \right)$, we pick $\theta_\star\in \left(0,\frac{\theta-\pi/2}{2}\right)$,  
 $\tau \in (0,+\infty)$ and $\psi \in (-\theta_\star,\theta_\star)$, and we notice that $z :=T-\epsilon_\star+\tau e^{i\psi}\in \cC_{\theta_\star}$.
Then, for all $p=re^{\pm i\theta}\in \gamma_\pm(\epsilon,\theta)$ and all $s\in (0,T-3\epsilon_\star)$, it holds true that
$$
\re ((z-s)p) = r \left( \tau \cos(\psi \pm \theta)+(T-\epsilon_\star-s)\cos \theta \right) \leq  2r\epsilon_\star \cos(\theta),
$$
where the symbol $\re$ denotes the real part, which entails
\bel{tt1c}
\re ((z-s)p) \leq  2\epsilon_\star \abs{p} \cos \theta,\ p \in \gamma_\pm(\epsilon,\theta),\ z \in \cC_{\theta_\star},\ s \in (0,T-3\epsilon_\star).
\ee
Hence, in light of \eqref{res-est} and \eqref{g1}-\eqref{g2}, we see for every $z \in \cC_{\theta_\star}$ that the function
\bel{def-v} 
v(z,\cdot):=\int_0^{T-3\epsilon_\star}\int_{\gamma(\epsilon,\theta)}e^{(z-s)p}(A_q+\rho p^\alpha)^{-1} f(s,\cdot) dp\ ds,
\ee
is well-defined in $\Omega$. 
Further, since $f(t,\cdot)=0$ for all $t \in (T-3\epsilon_\star,T)$, we infer from this and \eqref{tt1d}-\eqref{tt1e} that
$$v(t,\cdot)=\int_0^{T-3\epsilon_\star} S(t-s) f(s,\cdot)ds=\int_0^t S(t-s) f(s,\cdot)ds=u(t,\cdot),\ t \in (T-3\epsilon_\star,T).$$
Therefore, putting $v:=u$ on $(0,T-\epsilon_\star) \times \Omega$, we obtain that $u=v_{| Q}$. Moreover, by \cite[Theorem 1.1 and Remark 1]{KSY}, the Laplace transform $V(p)=\int_0^{+\infty} e^{-p t} v(t,\cdot) dt$ of $v$, is solution to the BVP 
\bel{ell-v}
\left\{ \begin{array}{rcll} 
(\cA_q+\rho p^\alpha) V(p) & = & \int_0^Te^{-pt}f(t,\cdot)dt, & \mbox{in}\ \Omega,\\
B_\star V(p) & = & 0, & \mbox{on}\ \pd\Omega.
\end{array}
\right.
\ee

It remains to show that $z \mapsto v(z,\cdot)$ is a holomorphic $L^2(\Omega)$-valued function in $\cC_{\theta_\star}$. To do that, we refer to \eqref{g1} and \eqref{def-v}, and we decompose $v$ into the sum $v_0 + v_+ + v_-$, where
$$ v_j(z,\cdot) := \int_0^{T-3\epsilon_\star}\int_{\gamma_j(\epsilon,\theta)}e^{(z-s)p}(A_q+\rho p^\alpha)^{-1} f(s,\cdot) dp\ ds,\ j=0,+,-.$$
Since $z \mapsto v_0(z,\cdot)$ is obviously holomorphic in $\cC_{\theta_\star}$, we are thus left with the task of proving that this is also the case for $z \mapsto v_{\pm}(z,\cdot)$. This can be done upon noticing that the $L^2(\Omega)$-valued function
$$ z\mapsto e^{(z-s)p}(A_q+\rho p^\alpha)^{-1} f(s,\cdot),\ p \in \gamma_\pm(\epsilon,\theta),\ s \in(0,T-3\epsilon_\star), $$
is holomorphic in $\cC_{\theta_\star}$, that the two following estimates,
$$\norm{\partial_z^ke^{(z-s)p}(A_q+\rho p^\alpha)^{-1} f(s,\cdot)}_{L^2(\Omega)} \leq C(1+ \abs{p})\abs{p}^{-\alpha} e^{2\epsilon_\star \abs{p} \cos\theta}\norm{f(s,\cdot)}_{L^2(\Omega)},\ z \in \cC_{\theta_\star} $$
hold for $k=0,1$ and some constant $C$ that is independent of $p$ and $s$, by virtue of \eqref{res-est} and \eqref{tt1c}, and that the function
$(r,s)\mapsto (1+r) r^{-\alpha} e^{2\epsilon_\star r \cos\theta}\norm{f(s,\cdot)}_{L^2(\Omega)}$ belongs to $L^1((\epsilon,+\infty)\times (0,T-3\epsilon_\star))$.\qed

\begin{rmk}
\label{rmk-res}
Since $F(p) \in L^2(\Omega)$ for all $p \in (0,+\infty)$, then, in accordance with Section \ref{sec-Duhamel}, we may reformulate the claim of Proposition \ref{pr1} that $U(p)$ solves \eqref{ell-BVP}, as
\bel{def-Up} 
U(p) = (A_q+ \rho p^\alpha)^{-1} F(p). 
\ee
Since the multiplication operator by $\rho$ is invertible in $\cB(L^2(\Omega))$, according to \eqref{eq-rho}, then
$A_{q,\rho} + p^{\alpha}$ is boundedly invertible in $L_\rho^2(\Omega)$ for each $p \in (0,+\infty)$, and \eqref{def-Up} may thus be equivalently rewritten as
$$
U(p) = (A_{q,\rho}+ p^\alpha)^{-1} \rho^{-1} F(p). 
$$
\end{rmk}

Armed with Proposition \ref{pr1}, we turn now to proving the main results of this article.

%%%%%%%%%%%%%%%%%%%%%%%%%%%%%%%%%%%%%%%%%%%%%%%%%%%%%%%%%%%%%%%%%
%%%%%%%%%%%%%%%%%%%%%%%%%%%%%%%%%%%%%%%%%%%%%%%%%%%%%%%%%%%%%%%%%
%%%%%%%%%%%%%%%%%                               Proof of Thm 1                             %%%%%%%%%%%%%%%%%%%%%%
%%%%%%%%%%%%%%%%%%%%%%%%%%%%%%%%%%%%%%%%%%%%%%%%%%%%%%%%%%%%%%%%%
%%%%%%%%%%%%%%%%%%%%%%%%%%%%%%%%%%%%%%%%%%%%%%%%%%%%%%%%%%%%%%%%%

\section{Proofs of the main results}
\label{sec-p1}

 In this section we prove Theorems \ref{t1} and \ref{t2}, and Corollary \ref{c1}, but we omit the derivation of Corollary \ref{c2} from Theorem \ref{t2}, which follows the same path as the one of Corollary \ref{c1} from Theorem \ref{t1}.

\subsection{Proof of Theorem \ref{t1}}
\label{sec-t1pr}

We split the proof into 4 steps. In the first one, we establish a family of resolvent identities for the Laplace transform of the solution to \eqref{eq1}, indexed by the Laplace variable $p \in (0,+\infty)$. The second step is to express these identities in terms of the spectral decomposition of the operator $A_{q,\rho}$, introduced in Section \ref{sec-mainres}. The third step, based on a weak unique continuation principle for second order elliptic equations, provides the desired result, while Step 4 contains the proof of a technical claim, used in Step 3.

\paragraph{\it Step 1: A $p$-indexed family of resolvent identities.} As $\supp \sigma\subset [0,T)$ by assumption, we pick $\epsilon_* \in (0,T \slash 4)$ such that $\supp \sigma \subset [0,T-3\epsilon_*]$. Then, with reference to Proposition \ref{pr1}, we extend the weak solution to \eqref{eq1} into a $L^2(\Omega)$-valued function 
$z \mapsto u(z,\cdot)$, defined in $(0,T-\epsilon_*]\cup \cC_{\theta_\star}$ for some $\theta_\star \in \left( 0 , \min \left( \frac{\pi}{4} , \frac{\pi}{2} - \frac{\pi}{2 \alpha} \right) \right)$, which is holomorphic in $\cC_{\theta_\star}$. Evidently, the $L^2(\Omega')$-valued function $z \mapsto u(z,\cdot)_{|\Omega'}$ is holomorphic in $\cC_{\theta_\star}$ as well. Bearing in mind that $u_{| Q'}=0$, where we have set $Q':=(0,T) \times \Omega'$, and that $(0,T)\cap \cC_{\theta_\star}=[T-\epsilon_*,T)$, we get that
$$ u(z,x)=0,\ (z,x) \in((0,T)\cup \cC_{\theta_\star})\times \Omega',$$
from the unique continuation principle for holomorphic functions. In particular, this entails that
$u(t,x)=0$ for all $(t,x)\in(0+\infty)\times \Omega'$ 
and consequently that the Laplace transform $U(p)$ of $u$ with respect to $t$, vanishes a.e. in $\Omega'$ for every $p \in (0,+\infty)$. Putting this together with the second statement of Proposition \ref{pr1}, we obtain that each $U(p)$, $p \in (0,+\infty)$, is solution to
\bel{t1d}
\left\{ \begin{array}{rcll} 
(\cA_q + \rho p^\alpha) U(p) & = & \widehat{\sigma}(p)g, & \mbox{in}\ \Omega,\\
B_\star U(p) & = & 0, & \mbox{on}\ \partial\Omega,\\
U(p) & = & 0, & \mbox{in}\ \Omega',
\end{array}
\right.
\ee
where we have set $\widehat{\sigma}(p):=\int_0^Te^{-pt}\sigma(t)dt$. Since $f \in L^2(\Omega)$, then,  in accordance with Remark \ref{rmk-res}, \eqref{t1d} may be equivalently reformulated , as
\bel{t1dbis}
\left\{ \begin{array}{rcll} 
U(p) & = &  \widehat{\sigma}(p) (A_{q,\rho}+ p^\alpha)^{-1} \rho^{-1} g  & \mbox{in}\ L_\rho^2(\Omega),\\
U(p) & = & 0 & \mbox{in}\ L_\rho^2(\Omega').
\end{array}
\right.
\ee

\paragraph{\it Step 2: Spectral representation.} Since the injection $V \hookrightarrow L^2(\Omega)$ is compact, the resolvent of the operator $A_{q,\rho}$, defined in Section \ref{sec-mainres}, is compact in $L_\rho^2(\Omega)$. Let $\{ \lambda_n: n \in \N \}$ be the increasing sequence of the eigenvalues of $A_{q,\rho}$. For each $n \in \N$, we denote by $m_n \in \N$ the algebraic multiplicity of the eigenvalue $\lambda_n$ and we introduce a family $\{ \varphi_{n,k}:\ k=1,\ldots, m_n \}$ of eigenfunctions of $A_{q,\rho}$, which satisfy 
$$ A_{q,\rho} \varphi_{n,k} = \lambda_n \varphi_{n,k}, $$
and form an orthonormal basis in $L_\rho^2(\Omega)$ of the eigenspace of $A_{q,\rho}$ associated with $\lambda_{n}$ (i.e. the kernel of $A_{q,\rho}-\lambda_n I$, where the notation $I$ stands for the identity operator of $L_\rho^2(\Omega)$). 
The first line in \eqref{t1dbis} then yields for all $p \in (0,+\infty)$, that the following equality
$$U(p)=\widehat{\sigma}(p)\left( \sum_{n=1}^{+\infty}\frac{\sum_{k=1}^{m_n}  g_{n,k}\varphi_{n,k}}{\lambda_n+p^\alpha}\right), $$
holds in $L_\rho^2(\Omega)$ with $g_{n,k}:=\langle \rho^{-1}g,\varphi_{n,k}\rangle_{L_\rho^2(\Omega)}$. 
From this, the second line of \eqref{t1dbis} and the continuity of the projection from $L_\rho^2(\Omega)$ into $L_\rho^2(\Omega')$, it then follows that
\bel{t1e} 
\widehat{\sigma}(p)\left( \sum_{n=1}^{+\infty}\frac{\sum_{k=1}^{m_n}  g_{n,k} \varphi_{n,k}(x)}{\lambda_n+p^\alpha} \right)=0,\ x \in \Omega',\ p \in (0,+\infty).
\ee

\paragraph{\it Step 3: End of the proof.} Since $p \mapsto \widehat{\sigma}(p)$ is holomorphic in $\C_+:= \{ z \in \C:\ \re{z} >0 \}$, then either of the two following conditions is true:
\begin{enumerate}[(a)]
\item For all $p \in \C_+$ we have $\widehat{\sigma}(p)=0$;
\item There exists an open interval $I \subset (0,+\infty)$, such that $\widehat{\sigma}(p) \neq 0$ for each $p \in I$.
\end{enumerate}

The first case is easily treated as we get that $\sigma=0$ a.e. in $(0,T)$ from (a) and the injectivity of the Laplace transform, which entails the desired result. In the second case, we combine (b) with \eqref{t1e} and obtain that
\bel{t1f}
\sum_{n=1}^{+\infty}\frac{\sum_{k=1}^{m_n}  g_{n,k} \varphi_{n,k}(x)}{\lambda_n+p^\alpha}=0,\ x \in \Omega',\ p\in I.
\ee
Let us introduce the following $L_\rho^2(\Omega')$-valued function,
\bel{def-R}
R(z):= \sum_{n=1}^{+\infty}\frac{\sum_{k=1}^{m_n}  g_{n,k} \varphi_{n,k}}{\lambda_n+z},\ z \in \C \setminus \{ -\lambda_n:\ n\in \N\},
\ee
meromorphic in $\C \setminus \{ -\lambda_n:\ n\in \N\}$ with simple poles $\{ -\lambda_n:\ n\in \N\}$. Evidently, \eqref{t1f} can be equivalently rewritten as
$$ R(p^\alpha)=\sum_{n=1}^{+\infty}\frac{\sum_{k=1}^{m_n}  g_{n,k} \varphi_{n,k}}{\lambda_n+p^\alpha}=0,\ p\in I, $$
the above identity being understood in $L_\rho^2(\Omega')$. Therefore, we necessarily have 
$R(z)=0$ for all $z\in \mathbb C\setminus\{-\lambda_n:\ n\in \N\}$, and consequently it holds true for all $n \in \N$ that
\bel{t1g}
\sum_{k=1}^{m_n} g_{n,k} \varphi_{n,k}(x)=0,\ x \in \Omega'.
\ee
Assume for a while that for each $n \in \N$, the eigenfunctions $\varphi_{n,k}$, $k=1,\ldots, m_n$, are linearly independent in $L_\rho^2(\Omega')$, the proof of this claim being postponed to Step 4, below.
Then, we infer from \eqref{t1g} that $g_{n,k}=0$ for all $n\in \N$ and all $k=1,\ldots,m_n$. Therefore, we find that
$$ g= \sum_{n=1}^{+\infty} \sum_{k=1}^{m_n} g_{n,k} \varphi_{n,k} = 0 $$
in $L_\rho^2(\Omega)$, which proves the desired result. 

\paragraph{\it Step 4: The $\varphi_{n,k}$, $k=1,\ldots, m_n$, are linearly independent in $L_\rho^2(\Omega')$.} For $n \in \N$ fixed, we consider $m_n$ complex numbers $\alpha_k$, for $k=1,\ldots,m_n$, such that 
\bel{t1h}
\sum_{k=1}^{m_n}\alpha_k \varphi_{n,k}(x) =0,\ x \in \Omega',
\ee
and we put $\varphi:= \sum_{k=1}^{m_n} \alpha_k \varphi_{n,k}$. Since each $\varphi_{n,k}$ lies in $D(A_{q,\rho})$, the domain of the operator $A_{q,\rho}$, then the same is true for $\varphi$, i.e. 
\bel{indom} 
\varphi \in D(A_{q,\rho}) = D(A_q), 
\ee
according to \eqref{dom-A}, and we have $A_{q,\rho} \varphi = \lambda_n \varphi$ in $L_\rho^2(\Omega)$. This and \eqref{t1h} translate into the fact that
$$\left\{ \begin{array}{rcll} 
( \cA_q - \lambda_n \rho) \varphi & = & 0, & \mbox{in}\ \Omega,\\
%\B \varphi & = & 0, & \mbox{on}\ \pd \Omega,\\
\phi & = & 0, & \mbox{in}\ \Omega'. 
\end{array}
\right.$$
Moreover, as we have $\varphi \in H^2(\Omega)$ from \eqref{dom-A}-\eqref{dom-Ar} and \eqref{indom}, the weak unique continuation principle for second order elliptic partial differential equations
(see, e.g., \cite[Theorem 1]{SS}) then yields that $\varphi=0$ a.e. in $\Omega$, i.e. 
$$
\varphi(x) = \sum_{k=1}^{m_n} \alpha_k \varphi_{n,k}(x)=0,\ x \in \Omega.
$$
Bearing in mind that $\{ \varphi_{n,k}:\ k=1,\ldots, m_n \}$ is orthonormal in $L_\rho^2(\Omega)$, we deduce from the above line that $\alpha_k=0$ for all $k=1,\ldots,m_n$, which establishes that the $\varphi_{n,k}$, $k=1,\ldots, m_n$, are linearly independent in $L_\rho^2(\Omega')$.

Having completed the proof of Theorem \ref{t1}, we turn now to  showing Corollary \ref{c1}.

\subsection{Proof of Corollary \ref{c1}}
In light of \eqref{eq1}, $u:=u_1-u_2$ is a weak solution is solution to
\bel{eq}
\left\{ \begin{array}{rcll} 
(\rho(x) \partial_t^{\alpha} +\cA_q) u(t,x) & = & f(t,x), & (t,x) \in  Q,\\
B_\star u(t,x) & = & 0, & (t,x) \in \Sigma, \\  
\pd_t^k u(0,\cdot) & = & 0, & \mbox{in}\ \Omega,\ k=0,\ldots,N_\alpha,
\end{array}
\right.
\ee
with $f(t,x)=\sigma_1(t) g_1(x) - \sigma_2(t) g_2(x)$ for a.e. $(t,x) \in Q$. 

%\subsection{Proof of Corollary \ref{c1}}
In the first (resp., second) case (i) (resp., (ii)), we have $f(t,x)=\sigma_1(t) (g_1-g_2)(x)$ where $\sigma_1 \in L^1(0,T)$ is supported in $[0,T)$ and $g_1-g_2 \in L^2(\Omega)$ (resp., $f(t,x)=(\sigma_1-\sigma_2)(t) g_1(x)$ where $\sigma_1-\sigma_2 \in L^1(0,T)$ is supported in $[0,T)$ and $g_1\in L^2(\Omega)$). Since $u=0$ in $Q'$, then, under Condition (i), an application of Theorem \ref{t1} yields $\sigma_1(t) (g_1-g_2)(x)=0$ for a.e. $(t,x) \in Q$ and hence $g_1=g_2$ in $\Omega$. Similarly, under Condition (ii), we obtain that $(\sigma_1-\sigma_2)(t) g_1(x)=0$ for a.e. $(t,x) \in Q$ and consequently that $\sigma_1=\sigma_2$ in $(0,T)$. The proof of Corollary \ref{c2} is thus complete.
% and we turn now to proving Theorem \ref{t1bis}. 

%%%%%%%%%%%%%%%%%%%%%%%%%%%%%%%%%%%%%%%%%%%%%%%%%%%%%%%%%%%%%%%%%
%%%%%%%%%%%%%%%%%%%%%%%%%%%%%%%%%%%%%%%%%%%%%%%%%%%%%%%%%%%%%%%%%
%%%%%%%%%%%%%%%%%                               Proof of Thm 2                             %%%%%%%%%%%%%%%%%%%%%%
%%%%%%%%%%%%%%%%%%%%%%%%%%%%%%%%%%%%%%%%%%%%%%%%%%%%%%%%%%%%%%%%%
%%%%%%%%%%%%%%%%%%%%%%%%%%%%%%%%%%%%%%%%%%%%%%%%%%%%%%%%%%%%%%%%%

\subsection{Proof of Theorem \ref{t2}} 

We split the proof into three steps. The first one is to prove existence of a $\cC([0,T]; H^{2\zeta}(\Omega))$-solution to the IBVP \eqref{eq1} with $\rho=1$ a.e. in $\Omega$ and $\mathcal{A}_q=-\Delta$. Then show the transformation of our inverse problem into inverse problems for a family of elliptic equations. Using this transformation we complete the proof of the theorem.

\paragraph{\it Step 1: Improved space-regularity result}
We start by establishing that the weak-solution to \eqref{eq1} associated with $\rho=1$, $q=0$ and source term $f\in L^r(0,T;L^2(\Omega))$, lies in $\cC([0,T];H^{2\zeta}(\Omega))$. 

As a preamble, we set $A:=A_0$, where we recall that $A_0$ is the self-adjoint realization of the (opposite of) the Laplace operator in $L^2(\Omega)$, endowed with either Dirichlet or Neumann boundary condition. Otherwise stated, $A$ is the self-adjoint operator in $L^2(\Omega)$, acting as $-\Delta$ on its domain $D(A)=H^2(\Omega) \cap H_0^1(\Omega)$ when the boundary operator $B_\star$ appearing in \eqref{eq1} reads $B_\star u = u$, while it is $D(A)=H^2(\Omega)$ when $B_\star u = \pd_{\nu_a} u$. We denote by $(\lambda_n)_{n \in \N}$ the sequence of eigenvalues of $A$, arranged in non-decreasing order and repeated with the multiplicity, and we introduce  an orthonormal basis $(\phi_n)_{n\ \in \N}$ in $L^2(\Omega)$ of eigenfunctions of $A$, obeying $A \varphi_n=\lambda_n \varphi_n$ for all $n \in \N$. 

Since the operator $A$ is nonnegative, we recall from the functional calculus, that 
$$
(A+1)^s h=\sum_{n=1}^{+\infty} (1+\lambda_n)^s \langle h,\phi_n\rangle_{L^2(\Omega)} \phi_n,\  h \in D((1+A)^s),\ s \in [0,+\infty), $$
where $D((1+A)^s)=\left\{ h \in L^2(\Omega):\ \sum_{n=1}^{+\infty}\abs{\langle h,\phi_n\rangle_{L^2(\Omega)}}^2 (1+\lambda_n)^{2s}<\infty \right\}$. For further reference, we set
$$ 
\norm{h}_{D((A+1)^s)} :=\left(\sum_{n=1}^{+\infty} (1+\lambda_n)^{2s} \abs{\langle h , \phi_n \rangle_{L^2(\Omega)}}^2 \right)^{\frac{1}{2}},\ h \in D((A+1)^s).
$$
As $f \in L^r(0,T;L^2(\Omega))$ with $r > \frac{1}{\alpha}$, then 
the weak solution $u$ to \eqref{eq1} reads
\bel{eae1}
u(t,\cdot) = \sum_{n=1}^\infty u_n(t) \varphi_n,\ t\in(0,T),
\ee
where $u_n(t):=\int_0^t(t-s)^{\alpha-1} E_{\alpha,\alpha}(-\lambda_n (t-s)^\alpha) \langle f(s,\cdot),\phi_n\rangle_{L^2(\Omega)} ds$ and $E_{\alpha,\beta}$ is the Mittag-Leffler function:
$$
E_{\alpha,\beta}(z)=\sum_{k=0}^\infty \frac{z^k}{\Gamma(\alpha k+\beta)},\ z\in \C,\ \alpha \in (0,+\infty), \beta \in \R.
$$
We refer the reader to \cite[Theorem 2.4]{SY}, \cite[Theorem 1.1]{KY} or \cite[Lemma 3.3]{FK} for the derivation of the representation formula \eqref{eae1} of $u$. Next, we recall from \cite[Theorem 1.6]{P} that
$$\abs{E_{\alpha,\alpha}(-\lambda_n t^\alpha)}\leq C \frac{t^{-\alpha\zeta}+1}{(1+\lambda_n)^\zeta},\ t \in (0,T),\ n \in \N,\  \zeta\in[0,1),$$
for some positive constant $C$ which is independent of $n$ and $t$.
Thus, for all $n \in \N$ we have
$$ \abs{ t^{\alpha-1} (1+\lambda_n)^{\zeta} E_{\alpha,\alpha}(-\lambda_n t^{\alpha})} \leq C t^{\alpha(1-\zeta)-1},\ t \in (0,T),\  \zeta\in[0,1), $$
and consequently $t \mapsto t^{\alpha-1} (1+\lambda_n)^{\zeta} E_{\alpha,\alpha}(-\lambda_n t^{\alpha}) \in L^{r'}(0,T),\  \zeta\in[0,1)$, 
% $$t\mapsto t^{\alpha-1}\lambda_n^\zeta E_{\alpha,\alpha}(-\lambda_n t^\alpha)\in L^{r'}(0,T),$$
where $r'$ is the real number conjugated to $r$, i.e. $r'$ is such that
$\frac{1}{r'}=1-\frac{1}{r} >1-\alpha(1-\zeta)$. Therefore, for all $\zeta\in[0,1)$, using that $s \mapsto \langle f(s,\cdot),\phi_n\rangle_{L^2(\Omega)} \in L^r(0,T)$, we obtain that $t \mapsto (1+\lambda_n)^\zeta u_n(t) \in \cC([0,T])$, and the following estimate
\bea
\norm{\sum_{k=n}^m u_k(t) \phi_k}_{D((1+A)^\zeta)} &\leq & \int_0^t (t-s)^{\alpha-1} \left(\sum_{k=n}^m (1+\lambda_n)^{2 \zeta} E_{\alpha,\alpha}(-\lambda_n (t-s)^\alpha)^2 \abs{\langle f(s,\cdot),\phi_n\rangle_{L^2(\Omega)}} ^2 \right)^{\frac{1}{2}} ds \nonumber \\
&\leq  & C\int_0^t(t-s)^{\alpha(1-\zeta)-1}\ \left( \sum_{k=n}^m \abs{\langle f(s,\cdot),\phi_n \rangle_{L^2(\Omega)}}^2 \right)^{\frac{1}{2}}  ds \nonumber \\
&\leq & C  \left( \int_0^T s^{(\alpha(1-\zeta)-1)r'} ds \right)^{\frac{1}{r'}} \left( \int_0^T  \left( \sum_{k=n}^m \abs{\langle f(s,\cdot),\phi_n \rangle_{L^2(\Omega)}}^2 \right)^{\frac{r}{2}} ds \right)^{\frac{1}{r}}, \label{eae1b}
\eea
which is true for all $t \in [0,T]$ and for all natural numbers $m$ and $n$ with $n\leq m$.

On the other hand, since $\lim_{n,m \to+\infty} \left( \int_0^T \left( \sum_{k=n}^m \abs{\langle f(s,\cdot),\phi_n \rangle_{L^2(\Omega)}}^2 \right)^{\frac{r}{2}} ds \right)^{\frac{1}{r}}=0$ as we have $f\in L^r(0,T;L^2(\Omega))$ by assumption, we derive from \eqref{eae1b}
that $\left( \sum_{k=1}^n u_k \phi_k \right)_{n \in \N}$ is a Cauchy sequence in $\cC([0,T];D((A+1)^\zeta))$,  $\zeta\in[0,1)$. 
Therefore, for all $\zeta\in[0,1)$, we have $u \in  \cC([0,T],D((A+1)^\zeta)))$ by \eqref{eae1} and consequently $u\in \cC([0,T],H^{2\zeta}(\Omega))$ from the embedding $D((A+1)^\zeta))  \subset H^{2\zeta}(\Omega)$.

Having established the first claim of Theorem \ref{t2}, we turn now to proving \eqref{t2a}.

Put $\Omega_0:=\omega\times (-L,L)$ and pick an open subset $\Omega_\star \subset \R^d$ with $\cC^2$ boundary, fulfilling all the following conditions simultaneously:
\bel{cond} 
\mbox{(a)}\ \Omega\subset \Omega_\star,\ \mbox{(b)}\ \pd \Omega \setminus \pd \Omega_\star \subset \gamma,\ \mbox{(c)}\ \Omega':=\Omega_\star \setminus \overline{\Omega}\ \mbox{is not empty},\ \mbox{(d)}\ \Omega_\star \setminus \overline{\Omega_0}\ \mbox{is connected}. 
\ee
Notice that such a subset $\Omega_\star$ exists in $\R^d$ as $\Omega \setminus \overline{\Omega_0}$ is connected and $\pd \Omega$ is $\cC^2$. We split the proof into two steps.
\paragraph{\it Step 2: Elliptic BVPs indexed by $p$}
Setting $f(t,x):=0$ and $u(t,x):=0$ for a.e. $(t,x) \in Q'$, we infer from \eqref{cond}(b) and the assumption
$u_{|(0,T)\times\gamma}=\partial_{\nu}u_{|(0,T)\times\gamma}=0$, that
\bel{eq3}
\left\{ \begin{array}{rcll} 
(\partial_t^{\alpha} -\Delta) u(t,x) & = & f(t,x), & (t,x)\in  (0,T) \times \Omega_\star,\\
u& = & 0, & (t,x) \in (0,T) \times \pd \Omega_\star, \\  
\pd_t^k u(0,\cdot) & = & 0, & \mbox{in}\ \Omega_\star,\ k=0,\ldots,N_\alpha.
\end{array}
\right.
\ee
We have $r>2$ as $\alpha (1-\zeta)<\frac{1}{2}$, whence $f \in L^1(0,T;L^2(\Omega_\star))$. Moreover, $f$ being supported in $[0,T) \times \overline{\Omega_\star}$, hence in $[0,T-3 \epsilon_\star] \times \overline{\Omega_\star}$ for some fixed $\epsilon_\star \in \left(0, \frac{T}{3} \right)$, we extend $t \mapsto u(t,\cdot)$ to a $L^2(\Omega_\star)$-valued function in $(0,+\infty)$ which is analytic in $(T-\epsilon_\star,+\infty)$, by invoking Proposition \ref{pr1} where $\Omega$ is replaced by $\Omega_\star$. 
Bearing in mind that $u$ vanishes in $Q'$, by assumption, we find that
\bel{ae1} 
u(t,x)=0,\ (t,x) \in (0,+\infty) \times\Omega'.
\ee
Moreover, in light of Proposition \ref{pr1}, we get for all $p \in (0,+\infty)$ that the Laplace transform $U(p)=\int_0^{+\infty} e^{-p t} u(t) dt$ of $u$, is solution to the following BVP
\bel{ae2}
\left\{ \begin{array}{rcll} 
(-\Delta + p^\alpha ) U(p) & = & \int_0^Te^{-pt}f(t,\cdot)dt, & \mbox{in}\ \Omega_\star,\\
 U(p) & = & 0, & \mbox{on}\ \pd \Omega_\star,
\end{array}
\right.
\ee
where $\nu_\star$ is the outward unit normal vector to $\pd \Omega_\star$.
Since $F(p) \in L^2(\Omega_\star)$ for each $p \in (0,+\infty)$ and since $\pd \Omega_\star$ is $\cC^2$, then $U(p) \in H^2(\Omega_\star)$ by elliptic regularity.

Next, as $f$ is supported in $[0,T] \times \overline{\Omega_0}$, we have
$F(p)=0$ in $\Omega_\star \setminus \overline{\Omega_0}$ for all $p \in (0,+\infty)$, and consequently
\bel{t2c}
\left\{ \begin{array}{rcll} 
(-\Delta+p^\alpha ) U(p) & = & 0, & \mbox{in}\ \Omega_\star \setminus \overline{\Omega_0},\\
U(p) & = & 0, & \mbox{on}\ \Omega',
\end{array}
\right.
\ee
by \eqref{ae1}-\eqref{ae2}. Since $\Omega_\star \setminus \overline{\Omega_0}$ is connected and $\Omega' \subset  \Omega_\star \setminus \overline{\Omega_0}$, and since $U(p) \in H^2(\Omega_\star \setminus \overline{\Omega_0})$, then the weak unique continuation principle for elliptic equations to \eqref{t2c} yields that
$U(p)=0$ in $\Omega_\star \setminus \overline{\Omega_0}$. Thus, taking into account that $\overline{\Omega_0} \subset \Omega_\star$ and that $U(p) \in H^2(\Omega_\star)$, we have
$$ U(p)=\pd_{\nu_0}U(p)=0\ \mbox{in}\ \pd\Omega_0,$$
where $\nu_0$ denotes the outward unit normal vector to $\Omega_0$. From this and the first line of \eqref{ae2}, it then follows that 
\bel{t2d}
\left\{ \begin{array}{rcll} 
(-\Delta+p^\alpha ) U(p,x) & = &G(p,x') h(x_d), & x=(x',x_d) \in \Omega_0,\\
U(p)=\partial_{\nu_0}U(p) & = & 0, & \mbox{on}\ \pd\Omega_0,
\end{array}\right.
\ee
where $G(p,\cdot):=\int_0^{+\infty} e^{-pt} g(t,\cdot) dt$.

\paragraph{\it Step 3: Fourier transform.}
For all $(k,\theta) \in\R \times \mathbb S^{d-2}$, where $ \mathbb S^{d-2}$ is the unit sphere of $\R^{d-1}$, we notice that
$$(-\Delta+p^\alpha)e^{-ik\theta\cdot x'}e^{\omega(p,k)x_d}=(k^2-\omega(p,k)^2+p^\alpha)e^{-ik\theta\cdot x'}e^{\omega(p,k)x_d}=0,\ p \in (0,+\infty),$$
where $\omega(p,k):= \left( p^\alpha + k^2  \right)^{\frac{1}{2}}$.
This and \eqref{t2d} yield
$$\int_{\Omega_0} G(p,x')h(x_d) e^{-ik\theta\cdot x'} e^{\omega(p,k)x_d}dx'dx_d=\int_{\Omega_0} (-\Delta+p^\alpha ) U(p,x) e^{-ik\theta\cdot x'}e^{\omega(p,k)x_d} dx' dx_d=0, $$
upon integrating by parts, and hence we get that
$$\left(\int_{\omega} G(p,x')e^{-ik\theta\cdot x'}dx'\right)\left(\int_{-L}^L h(x_d)e^{\omega(p,k)x_d}dx_d \right)=0$$
from Fubini's theorem.
Putting $G(p,\cdot)=0$ in $\R^{d-1}\setminus \omega$ and $h=0$ in $\R\setminus (-L,L)$, we thus find that
\bel{t2e}
\left(\int_{\R^{n-1}}G(p,x')e^{-ik\theta\cdot x'}dx'\right) \left(\int_\R h(x_d)e^{\omega(p,k)x_d}dx_d\right)=0,\ \theta \in \mathbb S^{d-2},\ k \in \R.
\ee
Next, $h \in L^1(\R)$ being compactly supported and not identically zero in $\R$, its Fourier transform
$z \mapsto \int_\R h(x_d)e^{z x_d}dx_d$
is holomorphic and not identically zero in $\C$. Therefore, there exists a non empty interval $(a,b) \subset (0,+\infty)$, with $a<b$, such that we have
$$ \int_\R h(x_d) e^{\omega(p,k) x_d} dx_d \neq 0,\ k\in(a,b).$$
This and \eqref{t2e} yield $\int_{\R^{d-1}} G(p,x') e^{-ik \theta \cdot x'}dx'=0$ for all $\theta \in \mathbb S^{d-2}$ and $k \in (a,b)$. Otherwise stated, the partial Fourier transform of $x' \mapsto G(p,x')$ vanishes in the concentric ring $C_{a,b}:=\{y \in \R^{d-1}:\ a< |y| <b\}$, where $\abs{y}$ denotes the Euclidian norm of $y \in \R^{d-1}$, i.e. 
\bel{ae3}  \int_{\R^{d-1}} G(p,x')e^{-i \xi \cdot x'}dx'=0,\ \xi \in C_{a,b}. 
\ee
Next, since $x' \mapsto G(p,x')$ is supported in the compact subset $\overline{\omega}$, then the function 
$\xi \mapsto \int_{\R^{d-1}} G(p,x')e^{-i \xi \cdot x'}dx'$ is real-analytic in $\R^{d-1}$, so we infer from \eqref{ae3} that
$\int_{\R^{d-1}} G(p,x')e^{-i\xi\cdot x'}dx'=0$ for all $\xi \in \R^{d-1}$. 
Therefore, we have $G(p,\cdot)=0$ in $\R^{d-1}$, by the injectivity of the partial Fourier transform with respect to $x'$, and since this equality holds for all $p \in (0,+\infty)$, we obtain that $g=0$ in $(0,T) \times \omega$,
from the injectivity of the Laplace transform with respect to $t$. This completes the proof of Theorem \ref{t2}.

%%%%%%%%%%%%%%%%%%%%%%%%%%%%%%%%%%%%%%%%%%%%%%%%%%%%%%%%%%%%%%%%%
%%%%%%%%%%%%%%%%%%%%%%%%%%%%%%%%%%%%%%%%%%%%%%%%%%%%%%%%%%%%%%%%%
%%%%%%%%%%%%%%%%%                             Other results                     %%%%%%%%%%%%%%%%%%%%%%
%%%%%%%%%%%%%%%%%%%%%%%%%%%%%%%%%%%%%%%%%%%%%%%%%%%%%%%%%%%%%%%%%
%%%%%%%%%%%%%%%%%%%%%%%%%%%%%%%%%%%%%%%%%%%%%%%%%%%%%%%%%%%%%%%%%

\section{Other results}
\label{sec-or}

 This section contains an alternative approach to the analysis of the inverse problems (IP1) and (IP2), presented in Section \ref{sec-p1}. Here, unlike Corollaries \ref{c1} and \ref{c2}, we no longer assume that the time-varying part of the unknown source term is compactly supported in $(0,T]$.

 The result that we have in mind for (IP1) can be stated as follows.

\begin{thm}
\label{t1bis} 
For $j=1,2$, let $g_j \in L^2(\Omega)$ and let $\sigma_j \in L^1(0,T)$ fulfill
$$\sigma_1(t)=\sigma_2(t)=\sigma(t),\ t \in (0,t_0), $$
where $t_0 \in (0,T)$ and $\sigma$ is a non-zero holomorphic function in the complex half-strip $\cS_\delta:=\{x+iy:\ x\in(-\delta,+\infty),\ y\in(-\delta,\delta)\}$ of fixed width $\delta \in (0,+\infty)$.  Assume moreover that $\sigma$ grows no faster than polynomials, i.e. that
\bel{hyp-sig}
\abs{\sigma(t)} \leq C (1+ t)^N,\ t\in (0,+\infty),
\ee
for some positive constant $C$ and some natural number $N$, which are both independent of $t$,  and that $g_1$ is not identically zero in $\Omega$. Then,  for all non-empty open subset $\Omega'$ of $\Omega$, we have:
\bel{impl-c1}
u_1=u_2\ \mbox{in}\ (0,T) \times \Omega'  \Longrightarrow  \sigma_1= \sigma_2\ \mbox{in}\ (0,T)\ \mbox{and}\  g_1= g_2\ \mbox{in}\  \Omega.
\ee
\end{thm}

\paragraph{\bf Proof of Theorem \ref{t1bis}.}
With reference to \eqref{eq}, we consider the following IBVP
\bel{eq2}
\left\{ \begin{array}{rcll} 
(\rho(x) \partial_t^{\alpha} +\cA_q) w(t,x) & = & \sigma(t)(g_1(x)-g_2(x)), & (t,x) \in (0,+\infty) \times \Omega,\\
B_\star w(t,x) & = & 0, & (t,x) \in (0,+\infty) \times \pd \Omega, \\  
\pd_t^k w(0,x) & = & 0, & x \in \Omega,\ k=0,\ldots, N_\alpha,
\end{array}
\right.
\ee
for $\alpha \in (0,2)$. 
With reference to Section \ref{sec-Duhamel}, \eqref{eq2} admits a unique solution $w \in \cC([0,+\infty),L^2(\Omega))$, which is expressed by \eqref{tt1d}-\eqref{tt1e}. Moreover, due to \eqref{hyp-sig}, we get upon arguing as in the derivation of
\cite[Theorem 1.4]{LKS} that the $L^2(\Omega)$-valued function $t \mapsto w(t,\cdot)$ is analytic in $(0,+\infty)$.

On the other hand, from the uniqueness of the solution to \eqref{eq} with $T=t_0$ and 
$f(t,x)=\sigma(t)(g_1(x)-g_2(x))$, we get that 
$w(t,x)=u(t,x)$ for a.e. $(t,x) \in (0,t_0)\times \Omega$.
Since $u=0$ in $Q'$, by assumption, the analyticity of $t \mapsto w(t,\cdot)$ in $(0,+\infty)$ then yields
\bel{mes-c2}
w(t,x)=0,\ (t,x) \in (0,+\infty) \times \Omega'.
\ee
Thus, taking the Laplace transform with respect to $t\in(0,+\infty)$ in \eqref{eq2} and in \eqref{mes-c2}, we obtain in the same way as in the derivation of \eqref{t1dbis} in Section 3.1, that for every $p \in (0,+\infty)$, 
$$
\left\{ \begin{array}{rcll} 
W(p) & = &  \widehat{\sigma}(p) (A_{q,\rho}+ p^\alpha)^{-1} \rho^{-1} (g_1-g_2)  & \mbox{in}\ L_\rho^2(\Omega),\\
W(p) & = & 0 & \mbox{in}\ L_\rho^2(\Omega'),
\end{array}
\right.
$$
where $W(p):=\int_0^{+\infty} e^{-tp} w(t,\cdot) dt$ and $\widehat{\sigma}(p):= \int_0^{+\infty} e^{-tp} \sigma(t) dt$ are the Laplace transforms of $w$ and $\sigma$, respectively. Notice from \eqref{hyp-sig} that $\widehat{\sigma}(p)$ is well-defined for each $p \in (0,+\infty)$.

Now,  by mimicking the last three steps of the derivation of Theorem \ref{t1}, we obtain that 
$g_1=g_2$ in $\Omega$. Therefore, condition (ii) of Corollary \ref{c1} is fulfilled and we deduce from Corollary \ref{c1} that \eqref{impl-c1} holds true. This terminates the proof of Theorem \ref{t1bis}.

 Analogously, we obtain the following result for (IP2) upon arguing is in the proof of Theorem \ref{t1bis}.

\begin{thm}
\label{t2bis} 
For $j=1,2$, let $h_j \in L^2(-L,L)$ and let $g_j \in L^1(0,T;L^2(\omega))$ fulfill
\bel{cs-a}
g_1(t,x')=g_2(t,x')=g(t,x'),\ (t,x') \in (0,t_0) \times \omega,
\ee
for some $t_0 \in (0,T)$, where $t \mapsto h(t,\cdot)$ is a non-zero holomorphic $L^2(\omega)$-valued function in the complex half-strip $\cS_\delta$ introduced in Theorem \ref{t1bis}, which grows no faster than polynomials: 
$$
\exists C>0,\ \exists N \in \N,\ \norm{g(t,\cdot)}_{L^2(\omega)} \leq C (1+ t)^N,\ t\in (0,+\infty).
$$
Assume moreover that $h_1$ is not identically zero in $\Omega$. Then,  we have $g_1=g_2$ in $(0,T) \times \omega$ and $h_1=h_2$ in $(-L,L)$, whenever the two following conditions
$$
\partial_\nu^ku_1(t,x)=\partial_\nu^ku_2(t,x),\ (t,x)\in (0,T)\times\gamma,\ k=0,1, 
$$
are satisfied for some non-empty open subset $\gamma$ of $\partial \Omega$.
\end{thm}

%%%%%%%%%%%%%%%%%%%%%%%%%%%%%%%%%%%%%%%%%%%%%%%%%%%%%%%%%%%%%%%%%
%%%%%%%%%%%%%%%%%%%%%%%%%%%%%%%%%%%%%%%%%%%%%%%%%%%%%%%%%%%%%%%%%
%%%%%%%%%%%%%%%%%                               Generalization                             %%%%%%%%%%%%%%%%%%%%%%
%%%%%%%%%%%%%%%%%%%%%%%%%%%%%%%%%%%%%%%%%%%%%%%%%%%%%%%%%%%%%%%%%
%%%%%%%%%%%%%%%%%%%%%%%%%%%%%%%%%%%%%%%%%%%%%%%%%%%%%%%%%%%%%%%%%

\section{Distributed order diffusion equations}
\label{sec-DO}
 In this section we adapt the analysis carried out in the preceding sections to the framework of distributed order fractional diffusion equations. More precisely, we consider the IBVP 
\begin{equation}
\label{equ-u-distri}
\left\{ \begin{array}{rcll} 
(\rho(x) \mathbb D^{(\mu)}_t +\cA_q) u(t,x) & = & f(t,x), & (t,x) \in  Q,\\
B_\star u(t,x) & = & 0, & (t,x) \in \Sigma, \\  
u(0,x) & = & 0, & x \in \Omega.
\end{array}
\right.
\end{equation}
where $\mathbb D^{(\mu)}_t$ denotes the distributed order fractional derivative
$$
\mathbb D^{(\mu)}_t h(t) := \int_0^1 \mu(\alpha) \pd^\alpha_t h(t) d\alpha,
$$
induced by a non-negative weight function $\mu \in L^\infty(0,1)$, obeying the following condition:
\bel{mu}
\exists \alpha_0 \in(0,1),\ \exists \delta \in (0,\alpha_0),\ \forall \alpha \in (\alpha_0-\delta,\alpha_0),\ \mu(\alpha) \ge \frac{\mu(\alpha_0)}{2}>0.
\ee

We aim to recover the source term $f$ expressed by one of the two prescribed forms \eqref{source1} or \eqref{source2}, from either internal or lateral data. But, prior to doing this, we examine the direct problem associated with the IBVP \eqref{equ-u-distri} and we show that it admits a unique solution whenever $f \in L^1(0,T;L^2(\Omega))$.

\subsection{The direct problem}
Here and in the remaining part of this section, $\rho$ and $\cA_q$ are the same as in Section \ref{sec-settings}, and $\partial^\alpha_t$ is the Caputo derivative of order $\alpha$ defined by \eqref{cap}. Let $f \in L^1(0,T;L^2(\Omega))$. We stick with the definition \cite[Definition 1.1]{LKS} of a weak solution to \eqref{equ-u-distri}, that is to say that $u$ is a weak solution to \eqref{equ-u-distri} if we have $u=v_{|Q}$ for some $v \in \cS'(\R_+,L^2(\Omega))$ whose Laplace transform $V$ verifies the following BVP for all $p \in (0,+\infty)$, 
\bel{eeeq1}
\left\{ \begin{array}{rcl} 
(\cA_q +\rho p \vartheta(p)) V(p) & =  & \int_0^Te^{-pt}f(t,\cdot)\ \mbox{in}\ \Omega, \\
B_\star V(p)& = & 0\ \mbox{on}\ \pd\Omega,
\end{array}
\right.
\ee
where $\vartheta(p):=\int_0^1 p^{\alpha-1} \mu(\alpha) d \alpha$.

We recall from \cite[Theorems 1.1 and 1.2]{LKS} that under the more restrictive assumption $f\in L^\infty(0,T;L^2(\Omega))$, the IBVP \eqref{equ-u-distri} admits a unique solution $u \in \cC([0,T], L^2(\Omega)) \cap L^1(0,T;H^{2 \zeta}(\Omega))$ for every $\zeta \in (0,1)$. %such that
%\bel{est1}
%\norm{u}_{L^1(0,T;H^{2 \zeta}(\Omega))} \leq C \norm{f}_{L^1(0,T;L^2(\Omega))},
%\ee
%for some positive constant $C$ which may depend on $\zeta$, but is independent of $f$. 
Moreover, by \cite[Proposition 2.1]{LKS}, $u$ enjoys the following representation formula
\bel{di1}
u(t,\cdot)=\int_0^t S_\mu(t-s) f(s,\cdot) ds,\ t \in (0,T),
\ee
where
\bel{di2} 
S_\mu(t) \psi := \frac{1}{2i\pi} \sum_{n=1}^{+\infty}  \sum_{k=1}^{m_n} \left( \int_{\gamma(\epsilon,\theta)}  \frac{e^{pt}}{\vartheta(p)+\lambda_n} dp \right)
 \langle \rho^{-1}\psi,\varphi_{n,k}\rangle_{L_\rho^2(\Omega)} \varphi_{n,k},\ \psi \in L^2(\Omega).
\ee
In \eqref{di2}, the pair $(\epsilon,\theta)$ is arbitrary in $(0,+\infty) \times \left( \frac{\pi}{2} , \pi \right)$, the contour $\gamma(\epsilon,\theta)$ is given by \eqref{g1}-\eqref{g2}, and the $\lambda_n$, $m_n$ and $\varphi_{n,k}$ are the same as in the proof of Theorem \ref{t1}.

Let us now extend \eqref{di1}-\eqref{di2} to the case of source terms $f \in L^1(0,T;L^2(\Omega))$.

\begin{prop} 
\label{p2} 
Assume \eqref{mu} and let $f \in L^1(0,T;L^2(\Omega))$. Then, for every $\zeta \in (0,1)$, there exists a unique weak solution $u \in  L^1(0,T;H^{2 \zeta}(\Omega))$ to \eqref{equ-u-distri}, which is expressed by \eqref{di1}-\eqref{di2}.
\end{prop}

\begin{proof}
Let $(f_n)_{n \in\N} \in \cC^\infty_0(0,T;L^2(\Omega))^{\N}$ be an approximating sequence of $f$ in $L^1(0,T;L^2(\Omega))$, i.e. such that
\bel{p2a}
\lim_{n\to\infty} \norm{f_n-f}_{L^1(0,T;L^2(\Omega))}=0.
\ee
Next, with reference to \eqref{di1}, we introduce for all $n \in \N$
$$ v_n(t,\cdot) :=\int_0^t S_\mu(t-s) \mathds{1}_{(0,T)}(s) f_n(s,\cdot) ds,\ t \in [0,+\infty), $$
in $\cS'(\R_+,L^2(\Omega))$, where $S_\mu$ is given by \eqref{di2} and $\mathds{1}_{(0,T)}$ denotes the characteristic function of the interval $(0,T)$.

%order to prove the existence of solutions for \eqref{equ-u-distri}, we will show that with this definition of $v_1$ we can define $\hat{v}_1(p)$ for all $p \in (0,+\infty)$ and condition
%\eqref{eeeq1} is fulfilled. For this purpose, we start by noticing that according 

As $f_n \in L^\infty(0,T;L^2(\Omega))$ for all $n \in \N$, the Laplace transform $V_n$ of $v_n$, verifies
\bel{p2b}
(A_{q,\rho}+p\vartheta(p)) V_n(p) =  \int_0^Te^{-pt}f_n(t,\cdot)dt,\ p \in (0,+\infty),
\ee
according to \cite[Proposition 2.1]{LKS}. Moreover, we have
\bel{di3}
\limsup_{n \to \infty} \norm{\int_0^Te^{-pt}f_n(t,\cdot)dt-\int_0^Te^{-pt}f(t,\cdot)dt}_{L^2(\Omega)} \leq \limsup_{n \to \infty} \norm{f_n-f}_{L^1(0,T;L^2(\Omega))}=0,\ p \in (0,+\infty), 
\ee
from \eqref{p2a}.

The next step of the proof is to establish for all $p \in (0,+\infty)$ that the Laplace transform $V(p)$ of the $L^2(\Omega)$-valued tempered distribution in $[0,+\infty)$,
\bel{di6}
t \mapsto v(t,\cdot) :=\int_0^t S_\mu(t-s)\mathds{1}_{(0,T)}(s)f(s,\cdot)ds
\ee
is well-defined in $L^2(\Omega)$ and verifies
\bel{p2c}
\limsup_{n \to \infty} \norm{V_n(p)-V(p)}_{L^2(\Omega)}=0.
\ee
To this purpose, we recall the following estimate from \cite[Lemma 2.2]{LKS},
\bel{p2d}
\frac{1}{\abs{\vartheta(p)+\lambda_n}}\leq C\max(\abs{p}^{-\alpha_0+\delta},\abs{p}^{-\alpha_0}),\ p \in \C \setminus (-\infty,0],\ n \in \N,
\ee
where the positive constant $C$ is independent of $n$ and $p$. Indeed, for all $t \in (0,+\infty)$ and all $\psi \in L^2(\Omega)$,  we infer from \eqref{p2d} upon taking $\epsilon=t^{-1}$ in \eqref{g2} that
\beas
& & \norm{\sum_{n=1}^{+\infty} \sum_{k=1}^{m_n} \left( \int_{\gamma_0(\epsilon,\theta)} \frac{e^{pt}}{\vartheta(p)+\lambda_n} dp \right) \langle \rho^{-1}\psi, \varphi_{n,k}\rangle_{L_\rho^2(\Omega)} \varphi_{n,k}}_{L^2(\Omega)}\\
&\leq & C \max(t^{\alpha_0-\delta-1},t^{\alpha_0-1}) \abs{\int_{-\theta}^{\theta}e^{\cos \beta}d\beta}^2 \norm{\psi}_{L^2(\Omega)} \\
%& \leq & C  \max(t^{\alpha_0-1},t^{\alpha_0-\delta-1}) \norm{\psi}_{L^2(\Omega)}.
\eeas
and
\beas
& & \norm{\sum_{n=1}^{+\infty}\sum_{k=1}^{m_n}
\left( \int_{\gamma_\pm(\epsilon,\theta)}\frac{e^{pt}}{\vartheta(p)+\lambda_n} dp \right)
\langle \rho^{-1}\psi,\varphi_{n,k}\rangle_{L_\rho^2(\Omega)} \varphi_{n,k}}_{L^2(\Omega)}\\
&\leq & C \abs{\int_{t^{-1}}^{+\infty} \max( r^{-\alpha_0+\delta},r^{-\alpha_0} ) e^{t r \cos \theta} dr} \norm{\psi}_{L^2(\Omega)}\\
&\leq & C t^{-1} \abs{\int_{1}^{+\infty} \max ( (t^{-1} r)^{-\alpha_0+\delta}, (t^{-1} r)^{-\alpha_0}  ) e ^{r \cos \theta} dr} \norm{\psi}_{L^2(\Omega)}\\
&\leq&  C \max(t^{\alpha_0-\delta-1},t^{\alpha_0-1}) \abs{\int_{1}^{+\infty} e^{r\cos\theta}dr} \norm{\psi}_{L^2(\Omega)}.
%&\leq & C \max(t^{\alpha_0-\delta-1},t^{\alpha_0-1}) \norm{\psi}_{L^2(\Omega)}.
 \eeas
Putting these two estimates together with \eqref{g1} and \eqref{di2}, we obtain that
\bel{p2e}
\norm{S_\mu(t)}_{\cB(L^2(\Omega))} \leq C \max(t^{\alpha_0-\delta-1},t^{\alpha_0-1}),\ t \in (0,+\infty),
\ee
for some constant $C$ that is independent of $t$. Thus, it holds true for all $p \in (0,+\infty)$ that
$$ t \mapsto y_p(t):= e^{- p t} \norm{S_\mu(t)}_{\cB(L^2(\Omega))} \in L^1(0,+\infty). $$
Moreover, setting $\tilde{f}_p(t) := \mathds{1}_{(0,T)}(t) e^{-p t} \norm{f(t,\cdot)}_{L^2(\Omega)}$ for a.e. $t \in (0,+\infty)$, we  obtain for each $p \in (0,+\infty)$ that
\beas
%\norm{V(p)}_{L^2(\Omega)}& \leq & 
& &e^{-pt} \norm{\int_0^t S_\mu(t-s) \mathds{1}_{(0,T)}(s) f(s,\cdot) ds}_{L^2(\Omega)}\\
&\leq & \int_0^t e^{-p(t-s)} \norm{S_\mu(t-s)}_{\mathcal B(L^2(\Omega))}\mathds{1}_{(0,T)}(s)e^{-ps}\norm{f(s,\cdot)}_{L^2(\Omega)}ds \\
&\leq & (y_p \ast \tilde{f}_p)(t),
\eeas
where the symbol $\ast$ denotes the convolution in $(0,+\infty)$. Therefore, we find for every fixed $p \in (0,+\infty)$ that
$$\int_0^{+\infty} e^{-pt} \norm{\int_0^t S_\mu(t-s) \mathds{1}_{(0,T)}(s) f(s,\cdot) ds}_{L^2(\Omega)} dt$$ is upper bounded by
$\norm{y_p \ast \tilde{f}_p}_{L^1(0,+\infty)}$, and hence by $\norm{y_p}_{L^1(0,+\infty)} \norm{\tilde{f}_p}_{L^1(0,+\infty)}$, 
which  we combine with \eqref{p2e} to yield
\bel{di4}
\int_0^{+\infty} e^{-pt} \norm{\int_0^t S_\mu(t-s) \mathds{1}_{(0,T)}(s) f(s,\cdot) ds}_{L^2(\Omega)} dt \leq C \max( p^{\delta-\alpha_0}, p^{-\alpha_0}) \norm{f}_{L^1(0,T;L^2(\Omega))},
\ee
for some positive constant $C$ that is independent of $p$. As a consequence, $V(p)$ is well-defined in $L^2(\Omega)$ and satisfies
$$ \norm{V(p)}_{L^2(\Omega)} \leq C \max( p^{\delta-\alpha_0}, p^{-\alpha_0}) \norm{f}_{L^1(0,T;L^2(\Omega))},\ p \in (0,+\infty). $$
Arguing as before with $f-f_n$ instead of $f$, we  obtain
$$ \norm{V(p)-V_n(p)}_{L^2(\Omega)} \leq C \max( p^{\delta-\alpha_0}, p^{-\alpha_0})\norm{f-f_n}_{L^1(0,T;L^2(\Omega))},\ p \in (0,+\infty),\ n \in \N. $$
 Using this along with \eqref{p2a} we get \eqref{p2c}.

With reference to \eqref{eeeq1} we are left with the task of proving that $V(p)$ lies in $D(A_{q,\rho})$, the domain of the operator $A_{q,\rho}$, and verifies
\bel{di5}
(A_{q,\rho}+p\vartheta(p)) V(p) =  F(p),\ p \in (0,+\infty).
\ee
To do that, we recall from the very definition of the function $\vartheta$ that $p \vartheta(p) >0$ for all $p \in (0,+\infty)$, and hence that the operator $A_{q,\rho}+ p \vartheta(p)$ is lower bounded by $\lambda_1  >0$ in $L_\rho^2(\Omega)$, according to \eqref{ell}-\eqref{eq-rho}. Thus, $A_{q,\rho}+p \vartheta(p)$ is boundedly invertible in $L_\rho^2(\Omega)$ and we have
$\norm{(A_{q,\rho}+p \vartheta(p))^{-1}}_{\cB(L_\rho^2(\Omega))} \leq \lambda_1^{-1}$ for all $p \in (0,+\infty)$. Therefore, since $V_n(p)=(A_{q,\rho}+p \vartheta(p))^{-1} F_n(p)$ for all $n \in \N$ and all $p \in (0+\infty)$, from \eqref{p2b}, we  have
\beas
\norm{(A_{q,\rho}+p \vartheta(p))^{-1} F(p) - V_n(p)}_{L_\rho^2(\Omega)} & \leq & \norm{(A_{q,\rho}+p \vartheta(p))^{-1}}_{\cB(L^2(\Omega)} \norm{F(p)-F_n(p)}_{L_\rho^2(\Omega)} \\
& \leq & \rho_M^{\frac{1}{2}} \lambda_1 \norm{F(p)-F_n(p)}_{L^2(\Omega)}.
\eeas
In light of \eqref{di3}, this  implies that
$$ \lim_{n \to +\infty} \norm{(A_{q,\rho}+p \vartheta(p))^{-1} F(p) - V_n(p)}_{L^2(\Omega)}=0,\ p \in (0,+\infty). $$
From this, \eqref{p2c} and the uniqueness of the limit in $L^2(\Omega)$, it then follows that $V(p)=(A_{q,\rho}+p \vartheta(p))^{-1} F(p)$, which is \eqref{di5}.

Therefore, $u=v_{| Q}$, where $v$ is defined by \eqref{di6}, is a weak solution to \eqref{equ-u-distri}. Now, in accordance with \cite[Section 1.4]{LKS},  we can finish the proof in the same way as in the one of \cite[Theorem 1.2]{LKS}.
\end{proof}

\begin{rmk}
\label{rmk-andi}
The representation formula \eqref{di1}-\eqref{di2} of the solution to \eqref{equ-u-distri} was obtained by replacing $S$ by $S_\mu$ in \eqref{tt1d}-\eqref{tt1e}. Therefore, if $f \in L^1(0,T;L^2(\Omega))$ is supported in $[0,T-3\epsilon_\star]\times\overline{\Omega}$ for some $\epsilon_\star \in \left( 0,\frac{T}{4} \right)$, then by substituting $S_\mu$ for $S$ in the derivation of Proposition \ref{pr1}, we see that the weak solution $u$ to \eqref{equ-u-distri} extends to an $L^2(\Omega)$-valued map which is analytic in $ \cC_{\theta_\star}$. Here, $\theta_\star$ can be any angle in $\left( 0, \min \left( \frac{\pi}{4} , \frac{\pi}{2\alpha} - \frac{\pi}{4} \right) \right)$ and $\cC_{\theta_\star}$ is defined in Proposition \ref{pr1}.
Moreover, the extended function $t \mapsto u(t,\cdot)$ is holomorphic in $\cC_{\theta_\star}$ and its Laplace transform $U(p)$ is a solution to \eqref{eeeq1} for all $p \in (0,+\infty)$. 
\end{rmk}

\subsection{Uniqueness result}

 First we extend the result of Theorem \ref{t1} to the case of distributed orders.

\begin{thm}
\label{t5} 
 Let $f$ be defined by \eqref{source1}, where $\sigma \in L^1(0,T)$ is supported in $[0,T)$ and 
$g \in L^2(\Omega)$.
Assume \eqref{mu} and denote by $u$ the weak-solution to \eqref{equ-u-distri}. Then, for any non-empty open subset
$\Omega' \subset \Omega$, we have the implication:
$$
 u=0\ \mbox{in}\ (0,T) \times \Omega'  \Longrightarrow  f=0\ \mbox{in}\ \Omega .
$$
\end{thm}
\begin{proof} 
With reference to Remark \ref{rmk-andi},  we follow the same lines as in the derivation of \eqref{t1e} and get that
\bel{di7} 
\widehat{\sigma}(p) \left( \sum_{n=1}^{+\infty} \frac{\sum_{k=1}^{m_n}  g_{n,k} \varphi_{n,k} (x')}{\lambda_n+\vartheta(p)} \right)=0,\ x \in \Omega',\ p\in (0,+\infty),
\ee
where we used the notations introduced in Section \ref{sec-Duhamel}.  If $\widehat{\sigma}$ is identically zero, then we have $\sigma=0$ in $(0,T)$ by injectivity of the Laplace transform. Otherwise, we may assume that $\widehat{\sigma}(p) \neq 0$ for all $p \in I$, where $I$ is a non-empty subinterval of $(0,+\infty)$. In light of \eqref{di7}, this  yields
$$ 
\sum_{n=1}^{+\infty} \frac{\sum_{k=1}^{m_n}  g_{n,k} \varphi_{n,k} (x')}{\lambda_n+\vartheta(p)} =0,\ x \in \Omega',\ p \in I,
$$
and consequently
\bel{di8}
R(\vartheta(p))=0,\ p \in I,
\ee
in the $L_\rho^2(\Omega')$-sense, where the function $R$ is defined by \eqref{def-R}.
Next, bearing in mind that $\vartheta'(p)=\int_0^1\alpha p^{\alpha-1}\mu(\alpha)d\alpha$, we infer from \eqref{mu} that
$$
\vartheta'(p) \geq \int_{\alpha_0-\delta}^{\alpha_0}\alpha p^{\alpha-1}\mu(\alpha)d\alpha \geq \frac{\delta(\alpha_0-\delta)  \mu(\alpha_0)}{2} \min_{p \in I} (p^{\alpha_0-1},p^{\alpha_0-\delta-1})>0,\  p \in I. $$
From this, \eqref{di8} and the fact that $R$ is a meromorphic function in $\C \setminus \{ \lambda_n,\ n \in \N \}$, it then follows that
$$ \sum_{k=1}^{m_n}  g_{n,k} \varphi_{n,k}(x')=0,\ x' \in \Omega',\ n \in \N. $$
Therefore, we have $g_{n,k}=0$ for all $k=0,\ldots ,m_n$ and all $n \in \N$ and consequently $g=\sum_{n=1}^{+\infty} \sum_{k=1}^{m_n} g_{n,k} \varphi_{n,k}=0$ in $L_\rho^2(\Omega)$.
\end{proof}

The second  uniqueness result that we establish, corresponds to Theorem \ref{t2}.

\begin{thm}
\label{t6} 
Let $T$, $\Omega$, $\omega$, $L$, $\rho$ and $\cA_q$ be the same as in Theorem \ref{t2}. Denote by $u$ the solution to \eqref{equ-u-distri} where $f$ is defined by \eqref{source2}  with $\sigma \in L^1(0,T;L^2(\omega))$ and $g \in L^2(-L,L)$. 
Assume moreover that $f$ is supported in $[0,T) \times \omega$. Then, for any non-empty open subset $\gamma \subset \partial\Omega$, we have
$$
 u=\partial_{\nu} u=0\ \mbox{on}\ (0,T) \times \gamma  \Longrightarrow  f=0\ \mbox{in}\ (0,T) \times \Omega . 
$$ 
\end{thm}
\begin{proof} 
Let the function $u$ be extended as in Remark \ref{rmk-andi}. Since its Laplace transform $U(p)$, $p \in (0,\infty)$, is solution to the BVP
$$
\left\{ \begin{array}{rcll} 
(-\Delta +\vartheta(p) ) U(p,\cdot) & = & F(p,\cdot), & \mbox{in}\ \Omega,\\
B_* U(p) & = & 0, & \mbox{on}\ \partial\Omega,\\
\end{array}
\right.
$$
where $F(p) \in L^2(\Omega)$, then we have $U(p) \in H^2(\Omega)$ by the elliptic regularity theorem. Therefore, taking into account that
$$ (-\Delta+\vartheta(p))e^{-ik \theta \cdot x'} e^{(k^2+\vartheta(p))^{1 \slash 2} x_d}=\left( k^2-(k^2+\vartheta(p)) +\vartheta(p) \right)e^{-ik\theta\cdot x'} e^{(k^2+\vartheta(p))^{1 \slash 2} x_d}=0, $$
for any $k \in \R$ and any $\theta \in \mathbb S^{d-2}$, and that
\bel{t6c}
\left\{ \begin{array}{rcll} 
(-\Delta +\vartheta(p) ) U(p,x) & = & \hat{\sigma}(p,x') g(x_d), & x=(x',x_d)\in \Omega_0,\\
U(p,x)=\partial_{\nu_0} U(p,x) & = & 0, & x \in \partial\Omega_0,\\
\end{array}
\right.
\ee
for all $p \in (0,+\infty)$, we find upon multiplying the first line in \eqref{t6c} by $e^{(k^2+\vartheta(p))^{1 \slash 2} x_d}$ and integrating by parts in $\Omega_0$, that:
$$ \int_{\Omega_0} \hat{\sigma}(p,x')h(x_d) g(x_d)e^{-ik\theta\cdot x'} e^{(k^2+\vartheta(p))^{1 \slash 2} x_d} dx'dx_d=0.$$
By the Fubini theorem, the above equality immediately leads to
$$\left(\int_{\omega} \hat{\sigma}(p,x') e^{-ik\theta\cdot x'}dx'\right)\left(\int_{-L}^L g(x_d) e^{(k^2+\vartheta(p))^{1 \slash 2} x_d}dx_d\right)=0, $$
for all $p \in (0,+\infty)$, all $k \in \R$ and all $\theta \in \mathbb S^{d-2}$, so the result follows from this upon arguing in the same way as in the proof of Theorem \ref{t2}.
\end{proof}

% % % % % % % % % % % % % % % % % % % % % % % % % % % % %%%%%%%%%%%%%%%%%%%%%%
% % % % % % % % % % % % % % % % % % % % % % % % % % % % %%%%%%%%%%%%%%%%%%%%%%
% % % % %                                                                                                                                                   % % % % % 
% % % % %                                            Numerical reconstruction  method                                                 % % % % %
% % % % %                                                                                                                                                   % % % % % 
% % % % % % % % % % % % % % % % % % % % % % % % % % % % %%%%%%% % % % % % % %%%% % % 
% % % % % % % % % % % % % % % % % % % % % % % % % % % % %%%%%%%%%%%%%%%%%%%%%%

\section{Numerical reconstruction}
\label{sec-NR}

In this section we numerically solve the inverse problem (IP1) for $\alpha \in (1,2)$ by mean of
an iterative scheme based on the Tikhonov regularization method. In view of Theorem \ref{t1}, we seek numerical reconstruction of the two-dimensional spatial part $g$ of the source term \eqref{source1}
appearing in \eqref{eq1}.

\subsection{Iterative method}
\label{sec-IM}
We aim at building an efficient iterative scheme for numerical reconstruction of the spatial term $g$
of the source, from knowledge of the temporal term $\sigma$ 
and internal measurements for all times in a subregion $\Omega'$. 
A reconstruction algorithm { was proposed in \cite{JLLY} and \cite{JLY}, for $\alpha \in (0,1)$ and $\alpha=2$, respectively,  that we shall adapt to the case $\alpha \in (1,2)$. In accordance with Theorem \ref{t1}, we choose a time-varying part of the unknown source with compact support, instead of the non-compactly supported one which is considered in \cite{JLLY}.

We recall that the fractional Caputo derivative of order $\alpha \in (1,2)$ is defined by
$$\partial_t^\alpha u(t,x) := \frac{1}{\Gamma(2-\alpha)}\int_0^t (t-s)^{1-\alpha}\partial_s^2 u(s,x)ds, \quad (s,x)\in Q,$$
and that the backward fractional Caputo derivative reads
$$\partial_{t\_}^\alpha u(t,x) := \frac{1}{\Gamma(2-\alpha)}\int_t^T (s-t)^{1-\alpha}\partial_s^2 u(s,x)ds, \quad (s,x)\in Q.$$
For the homogeneous Neumann boundary condition, the forward problem under examination reads
\begin{equation}\label{eqn:forward}
\left\{
\begin{array}{ll}
(\partial_t^\alpha+\mathcal{A}_q)u(t,x)  =  g(x)\sigma(t), & (t,x) \in Q,\\
         \partial_{\nu_a}u(t,x)  =  0,          & (t,x) \in \Sigma,\\
   u(0,x) = \partial_tu(0,x)           =  0,          & x \in \Omega.
\end{array}
\right.
\end{equation}
Let us introduce the backward Riemann-Liouville fractional derivative of order $\alpha$  as
$$\partial_t^{\alpha*} u(t,x) := \frac{1}{\Gamma(2-\alpha)}\partial_t^2\int_t^T (s-t)^{1-\alpha} u(s,x)ds, \quad (s,x)\in Q.$$
Then, the weak  formulation of  the system (\ref{eqn:forward}) is given by
$$\int_Q\left( \sum_{i,j=1}^d a_{i,j}\partial_{x_i}u\partial_{x_j}w+quw+u\partial_t^{\alpha*}w \right)dxdt 
= \int_Q g\sigma w dxdt$$ 
for any test function $w\in H^\alpha(0,T;L^2(\Omega))\bigcap L^2(0,T;H^1(\Omega))$  satisfying $J_0w=J_1w=0$ in $\Omega$,
where we have set
\begin{eqnarray*}
J_0 u(t,x) &=& \lim_{t\rightarrow T} \frac{1}{\Gamma(2-\alpha)}\int_t^T (s-t)^{1-\alpha} u(s,x)ds,\\
J_1 u(t,x) &=& \lim_{t\rightarrow T} \frac{1}{\Gamma(2-\alpha)}\partial_t\int_t^T (s-t)^{1-\alpha} u(s,x)ds.
\end{eqnarray*}

Assume that we have noise contaminated measurement $u^\delta \in W^{1,1}(0,T;L^2(\Omega))$ in a subregion $\Omega'\subset\Omega$ and over the interval $(0,T)$, such that
$\|u^\delta-u(g_{\text{true}})\|_{L^2(\Omega')}\leq\delta$, where   $g_{\text{true}}$ is the true data, $u(g_{\text{true}})$ is the solution to \eqref{eqn:forward} associated with $f(t,x)=\sigma(t) g_{\text{true}}(x)$, and $\delta$ is the noise level.
Then, the numerical reconstruction of the source term  $g_{\text{true}}$ can be formulated as a least squares problem with Tikhonov regularization:
\begin{equation}\label{eqn:ls}
\min_{g\in L^2(\Omega)}\Phi(g), \quad \Phi(g) := \|u(g)-u^\delta\|^2_{L^2((0,T)\times\Omega')}+\rho\|g\|^2_{L^2(\Omega)}.
\end{equation}
Traditional iterative methods solving the least squares problem (\ref{eqn:ls}) 
require the computation of the Fr\'{e}chet derivative $\Phi'(g)$ of the object function $\Phi(g)$.
For an arbitrary direction $h\in L^2(\Omega)$, $\Phi'(g)h$ is given by
\begin{eqnarray}
\Phi'(g)h &=& 2\int_0^T\int_{\Omega'}(u(g)-u^\delta)(u'(g)h)dxdt+2\rho\int_\Omega gh dx \nonumber \\
          &=& 2\int_0^T\int_{\Omega'}(u(g)-u^\delta)u(h)dxdt+2\rho\int_\Omega gh dx \label{eqn:frechet}
\end{eqnarray}
because of the linear dependence of $u(g)$ on $g$.

Next, we consider the adjoint system
\begin{equation}\label{eqn:adjoint}
\left\{
\begin{array}{ll}
(\partial_t^{\alpha*}+\mathcal{A}_q)z  =  \chi_{\Omega'}(u(g)-u^\delta), & \text{ in $Q$,}\\
         \partial_{\nu_a}z =  0,          & \text{ on $\Sigma$,}\\
   J_0z = J_1z          =  0,          & \text{ in $\Omega$},
\end{array}
\right.
\end{equation}
 where $\chi_{\Omega'}$ is the characteristic function of $\Omega'$. Then, the first term of the equation (\ref{eqn:frechet}) is equal to
\begin{eqnarray*}
\int_0^T\int_{\Omega'}(u(g)-u^\delta)u(h)dxdt &=& \int_Q \chi_{\Omega'}(u(g)-u^\delta) u(h) dxdt\\
& = & \int_Q (\partial_t^{\alpha*}+\mathcal{A}_q)z(t,x)u(h) dxdt\\
& = & \int_Q h\sigma z(g) dxdt,
\end{eqnarray*}
which implies
$$\Phi'(g)h = 2\int_{\Omega}\left(\int_0^T \sigma z(g)dt+\rho g\right)h dx.$$
Therefore,  the minimizer $g$ of \eqref{eqn:ls} verifies
\begin{equation}\label{eqn:optimal}
\rho g = -\int_0^T \sigma z(g) dt.
\end{equation}
Thus, by adding $Mg$ on both sides of (\ref{eqn:optimal}),  we obtain the following iterative reconstruction algorithm, which is borrowed from \cite{daubechies2004iterative,JLLY}:
\begin{eqnarray*}
g_{k+1}& =& \frac{M}{M+\rho}g_k-\frac{1}{M+\rho}\int_0^T\sigma z(g_k)dt,\quad k=0,1,2,\ldots
\end{eqnarray*}
The iteration stops when the condition
$$\|g_{k+1}-g_{k}\|_{L^2(\Omega)}<\varepsilon \|g_k\|_{L^2(\Omega)}$$
 is fulfilled, where $\varepsilon$ is the  {\it a priori} fixed precision parameter.

In order to avoid dealing with the non-local final condition in the third line of (\ref{eqn:adjoint}) and with the backward Riemann-Liouville fractional derivative $\partial_t^{\alpha *}$ appearing in the master equation of (\ref{eqn:adjoint}), whose numerical treatment is a delicate matter, we rather solve the following problem with a backward Caputo
fractional derivative:
\begin{equation}\label{eqn:adjoint_Caputo}
\left\{
\begin{array}{ll}
(\partial_{t\_}^{\alpha}+\mathcal{A}_q)z  =  \chi_{\Omega'}(u(g)-u^\delta), & \text{ in $Q$,}\\
         \partial_{\nu_a}z  =  0,          & \text{ on $\Sigma$,}\\
   z(T,\cdot)=\partial_t z(T,\cdot) =  0,          & \text{ in $\Omega$}.
\end{array}
\right.
\end{equation}
From a theoretical viewpoint, this choice is justified by the:
\begin{lem}
\label{lm-equivalence}
The system  \eqref{eqn:adjoint_Caputo} is equivalent to (\ref{eqn:adjoint}).
\end{lem}
\begin{proof}
Let us first prove that the weak solution to \eqref{eqn:adjoint_Caputo} solves \eqref{eqn:adjoint}.
We start by showing that the homogeneous final condition $z(T,\cdot)=\partial_t z(T,\cdot) =  0$ of (\ref{eqn:adjoint_Caputo}),
implies $J_0 z = J_1 z = 0$ in $\Omega$. To this purpose, we preliminarily establish that the weak solution $z$ to \eqref{eqn:adjoint_Caputo}, lies in $C^1([0,T];L^2(\Omega))$. Indeed, putting $w(t,\cdot)=z(T-t,\cdot)$, we get that
$$\partial_{t\_}^{\alpha}z(T-t,\cdot)=\partial_t^\alpha w(t,\cdot),\quad t\in(0,T), $$
and hence
$$
\left\{
\begin{array}{ll}
(\partial_{t}^{\alpha}+\mathcal{A}_q)w(t,\cdot)  =  \chi_{\Omega'}(u(g)-u^\delta)(T-t,\cdot), & \text{ in $Q$,}\\
         \partial_{\nu_a}w(t,x)  =  0,          & \text{ on $\Sigma$,}\\
   w(0,x)=\partial_t w(0,x) =  0,          & \text{ in $\Omega$}.
\end{array}
\right.
$$
Notice that $\chi_{\Omega'}(u(g)-u^\delta)\in W^{1,1}(0,T;L^2(\Omega))$ by \cite[Theorem 1.1]{KY}, whence $w \in W^{2,r}(0,T;L^2(\Omega))$ for all $r\in(1,(2-\alpha)^{-1})$, from \cite[Theorem 2.7.]{KY3}. Therefore, we have
$z \in W^{2,r}(0,T;L^2(\Omega))$ and hence $z\in \mathcal C^1([0,T];L^2(\Omega))$ by the Sobolev embedding theorem. As a consequence we have $z\in L^\infty(0,T;L^2(\Omega))$, so we get for all $t\in(0,T)$ that
\begin{eqnarray*}
\left\|\frac{1}{\Gamma(2-\alpha)}\int_t^T(s-t)^{1-\alpha}z(s,\cdot)ds\right\|_{L^2(\Omega)}&\leq &\frac{\|z\|_{L^\infty(0,T;L^2(\Omega))}}{\Gamma(2-\alpha)}\int_t^T(s-t)^{1-\alpha}ds\\
&\leq &\frac{\|z\|_{L^\infty(0,T;L^2(\Omega))}}{\Gamma(2-\alpha)}\int_0^{T-t}s^{1-\alpha}ds.
\end{eqnarray*}
From this it then follows that
\begin{eqnarray}
|J_0z\|_{L^2(\Omega)} & \leq & \limsup_{t\rightarrow T}\left\|\frac{1}{\Gamma(2-\alpha)}\int_t^T(s-t)^{1-\alpha}z(s,\cdot)ds\right\|_{L^2(\Omega)} \nonumber \\
&\leq &\frac{\|z\|_{L^\infty(0,T;L^2(\Omega))}}{\Gamma(2-\alpha)} \lim_{t\rightarrow T}\int_0^{T-t}s^{1-\alpha}ds=0. \label{jzz}
\end{eqnarray}
Similarly, $z\in \mathcal C^1([0,T];L^2(\Omega))$ and $z(T,\cdot)=0$ yield
\begin{eqnarray}
\partial_t\int_t^T (s-t)^{1-\alpha}z(s,\cdot)ds&=& \partial_t\left(\int_0^{T-t} s^{1-\alpha}z(s+t,\cdot)ds\right) \nonumber \\
&=&\int_0^{T-t}  s^{1-\alpha}\partial_sz(t+s,\cdot)ds+z(T,\cdot)(T-t)^{1-\alpha} \nonumber \\
&=& \int_t^T (s-t)^{1-\alpha}\partial_sz(s,\cdot)ds, \label{eqn:equiv11}
\end{eqnarray}
which entails
$$J_1z=  \frac{1}{\Gamma(2-\alpha)}\lim_{t\rightarrow T}\int_t^T (s-t)^{1-\alpha}\partial_sz(s,\cdot)ds,$$
so we get $J_1z=0$ by arguing as in the derivation of \eqref{jzz}.

It remains to prove that $z$ satisfies the first line of \eqref{eqn:adjoint}. Since $z$ is a solution to \eqref{eqn:adjoint_Caputo}, this amounts to showing that $\partial_t^{\alpha*}z=\partial_{t\_}^{\alpha}z$.
We proceed in the same way as in the derivation of \eqref{eqn:equiv11}, that is we take into account that $z(T,\cdot)=\partial_tz(T,\cdot)=0$ and $z\in W^{2,r}(0,T;L^2(\Omega))$ for all $r\in(1,(2-\alpha)^{-1})$, to write
\begin{eqnarray*}
\partial_t^{\alpha*}z&=&\partial_t^2\int_t^T (s-t)^{1-\alpha}z(s,\cdot)ds\\
&=&\partial_t^2\left(\int_0^{T-t} s^{1-\alpha}z(s+t,\cdot)ds\right)\\
&=&\partial_t\left(\int_0^{T-t} s^{1-\alpha}\partial_sz(s+t,\cdot)ds+z(T,\cdot)(T-t)^{1-\alpha}\right)\\
&=&\int_0^{T-t}  s^{1-\alpha}\partial_s^2z(t+s,\cdot)ds+\partial_tz(T,\cdot)(T-t)^{1-\alpha}\\
&=&\int_t^T (s-t)^{1-\alpha}\partial_s^2z(s,\cdot)ds=\partial_{t\_}^{\alpha}z.
\end{eqnarray*}

Summing up, we have established that weak solution to \eqref{eqn:adjoint_Caputo} is a solution to \eqref{eqn:adjoint}. But, since the solution to \eqref{eqn:adjoint} is unique, this means that the two systems \eqref{eqn:adjoint} and \eqref{eqn:adjoint_Caputo} are equivalent. 
\end{proof}

\begin{rmk} It is easy to see that the statement Lemma \ref{lm-equivalence} remains valid for $\alpha\in(0,1)$ and $\sigma\in W^{1,1}(0,T)$. 
\end{rmk}

%\begin{rmk}
%{\color{red} In addition to the extension of the analysis of \cite{JLLY} to the super-diffusive case $\alpha\in(1,2)$, we %obtain several results that allows to distinguish our work with the one of \cite{JLLY}. For instance, we prove rigorously  %the equivalence of the coming systems \eqref{eqn:adjoint_Caputo} and (\ref{eqn:adjoint}) in the case $\alpha\in(1,2)$ %and for $\sigma\in L^1(0,T)$. The proof of this result, which simplifies considerably the numerical implementation for
%(IP1), seems to be missing so far, even for $\alpha\in(0,1)$ and smooth parameters $\sigma$. In addition to this result, in %contrast to \cite{JLLY},  we study numerically the impact of observation regions and extreme cases when 
%$\alpha\rightarrow 1$ and $\alpha\rightarrow 2$.
%According to our theory, we are able to choose $\sigma(t)$ compactly supported in $[0, T]$, which implies that
%we can wait long enough until the source propagates into the observation region and therefore we have full knowledge of %the source.
%This helps to improve the reconstruction, specially while $\alpha \rightarrow 2$ as demonstrated in Example
%\ref{exa:extreme}.}
%\end{rmk}

\subsection{Numerical computations}
This section provides several  numerical computations of  $g_{\text{true}}$ by the iterative scheme introduced in Section \ref{sec-IM}, in the particular case where
$$\Omega = (0,1)\times(0,1),\quad T=1,\quad \mathcal{A}_q u=-0.1\Delta u+u,$$
and
\begin{eqnarray*}
\sigma(t)&=&\frac{1}{\sqrt{2\pi}s}e^{-\frac{(t-0.4)^2}{2s^2}},\ s = 0.12,\quad g_0(x_1,x_2) = 2.
\end{eqnarray*}
Notice that $\sigma \approx 0$ near the boundary. From numerical point of view, 
$\sigma$ is compactly supported.
To obtain the noisy observation $u^\delta$, we solve the forward problem numerically and 
add uniformly distributed random noise to the solution, i.e.,
$$u^\delta(t,x) = \big(1+\delta \text{rand}(-1,1)\big)u(g_{\text{true}})(t,x).$$
Here $\text{rand}(-1,1)$ is a uniformly distributed number in $[-1,1]$ and $\delta$ is the noise level.
For parameters in the iterative method, we fix $\rho=10^{-5}$ and $M=4$.
To evaluate the performance of the reconstruction, we compute the relative error $Res:= {\|g_k-g_{\text{true}}\|_{L^2(\Omega)}}/{\|g_{\text{true}}\|_{L^2(\Omega)}}.$

\begin{example}
In this example we compare reconstructed results for $g(x_1,x_2)$ with different $\alpha$.
We choose the noise level $\delta = 2\%$, the stopping criterion $\varepsilon=\delta/50$ 
and the observation subregion $\Omega' = \overbar{\Omega}\setminus(0.1,0.9)^2$.
We choose two pairs of fractional orders $\alpha=1.2$ and $\alpha = 1.8$, and two true source terms
$$g_{\text{true}}(x_1,x_2) = x_1+x_2+1\ \text{ and }\ g_{\text{true}}(x_1,x_2) = \cos(\pi x_1)\cos(\pi x_2)+2.$$

Figure \ref{fig:alpha} demonstrates results with iteration steps $K$ and relative errors $Res$.
Reconstruction with $\alpha=1.2$ takes fewer steps and the result is more accurate than $\alpha=1.8$.
\begin{figure}[htb]
\centering
\begin{subfigure}[p]{.32\textwidth}
\includegraphics[width=.2\textheight]{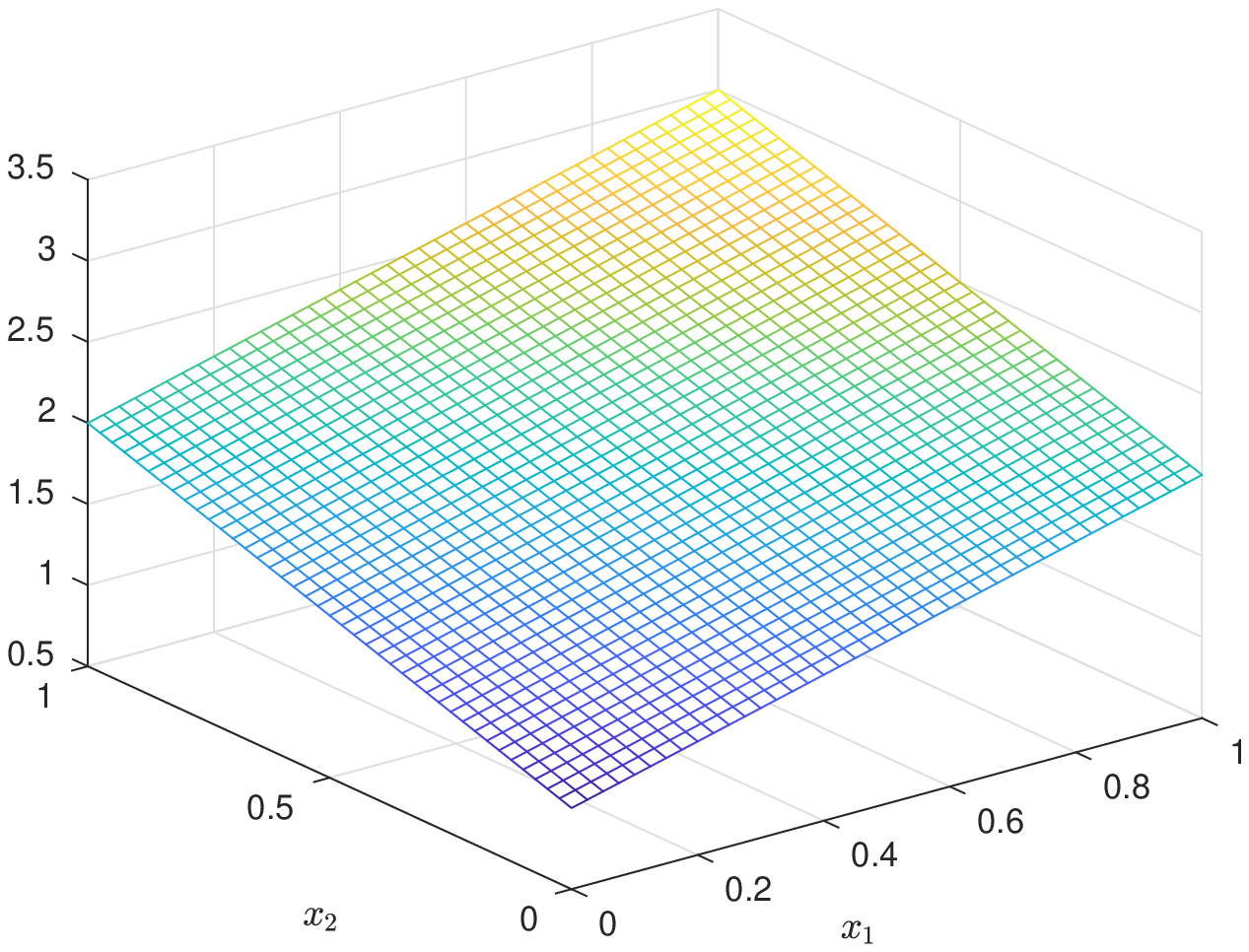}
\caption{$g_{\text{true}}=x_1+x_2+1$}
\end{subfigure}\quad
\begin{subfigure}[p]{.30\textwidth}
\includegraphics[width=.2\textheight]{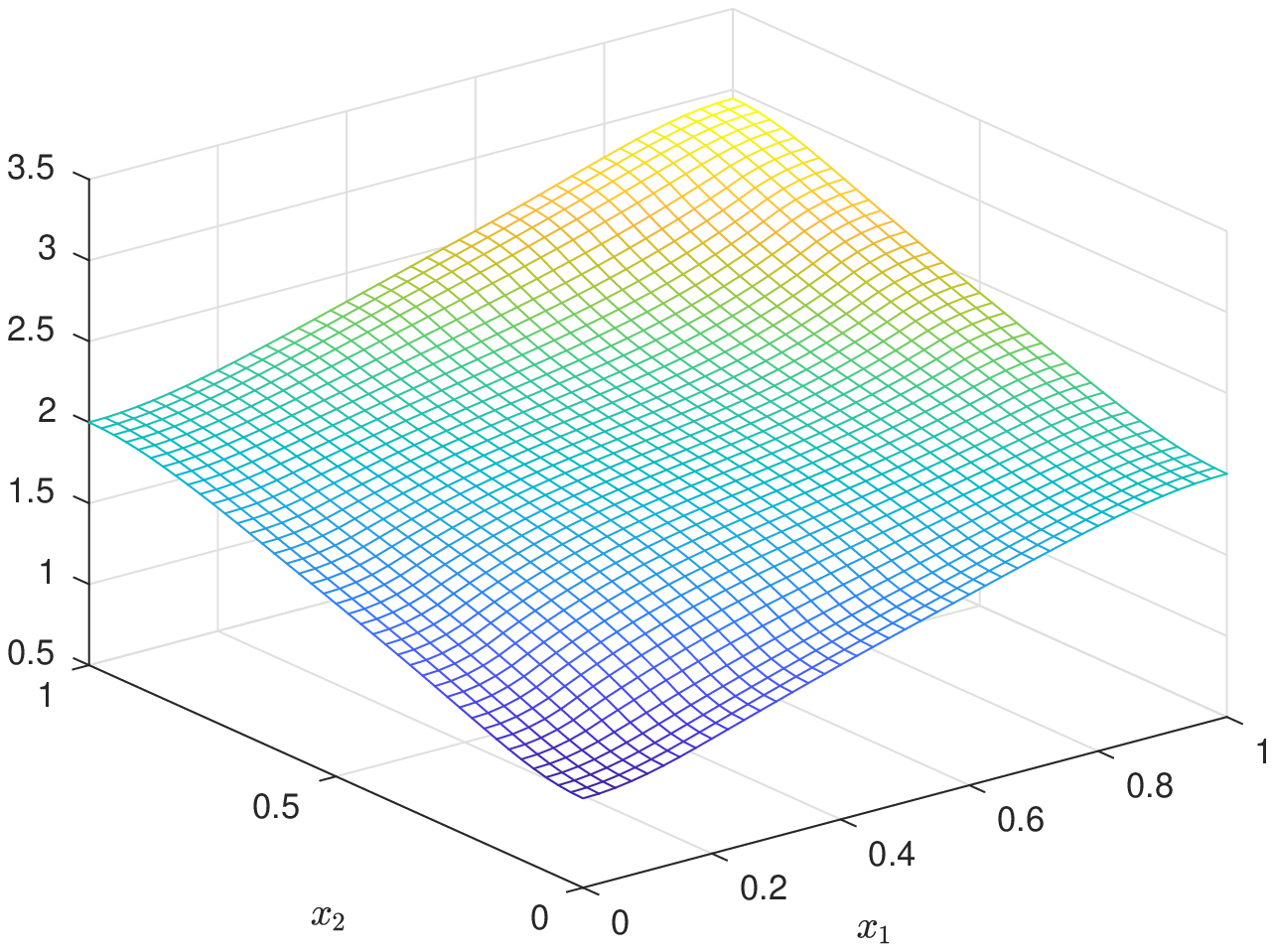}
\caption{$\alpha=1.2$}
\label{fig:inv_plane_1.2}
\end{subfigure}\quad
\begin{subfigure}[p]{.30\textwidth}
\includegraphics[width=.2\textheight]{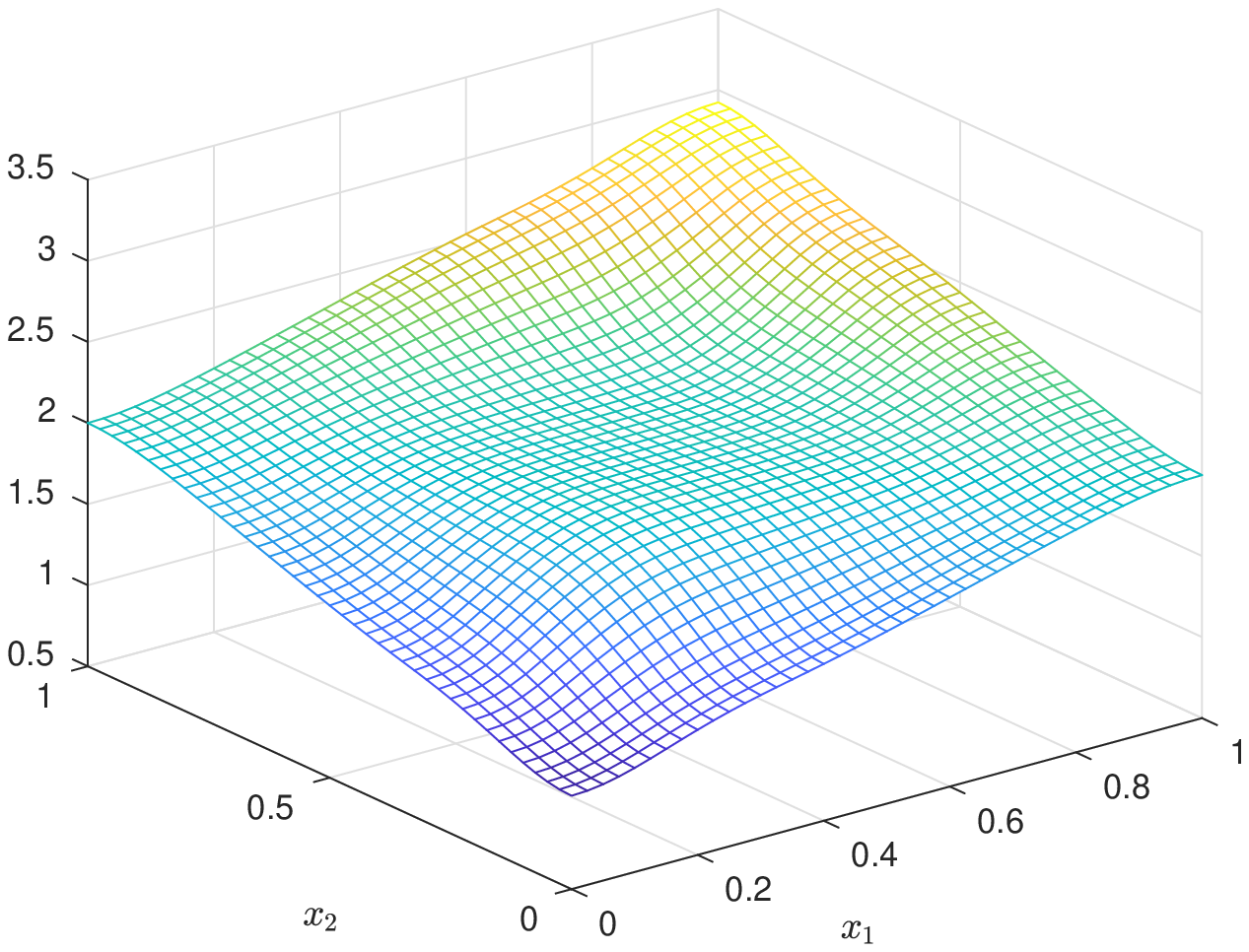}
\caption{$\alpha=1.8$}
\label{fig:inv_plane_1.8}
\end{subfigure}

\begin{subfigure}[p]{.32\textwidth}
\includegraphics[width=.2\textheight]{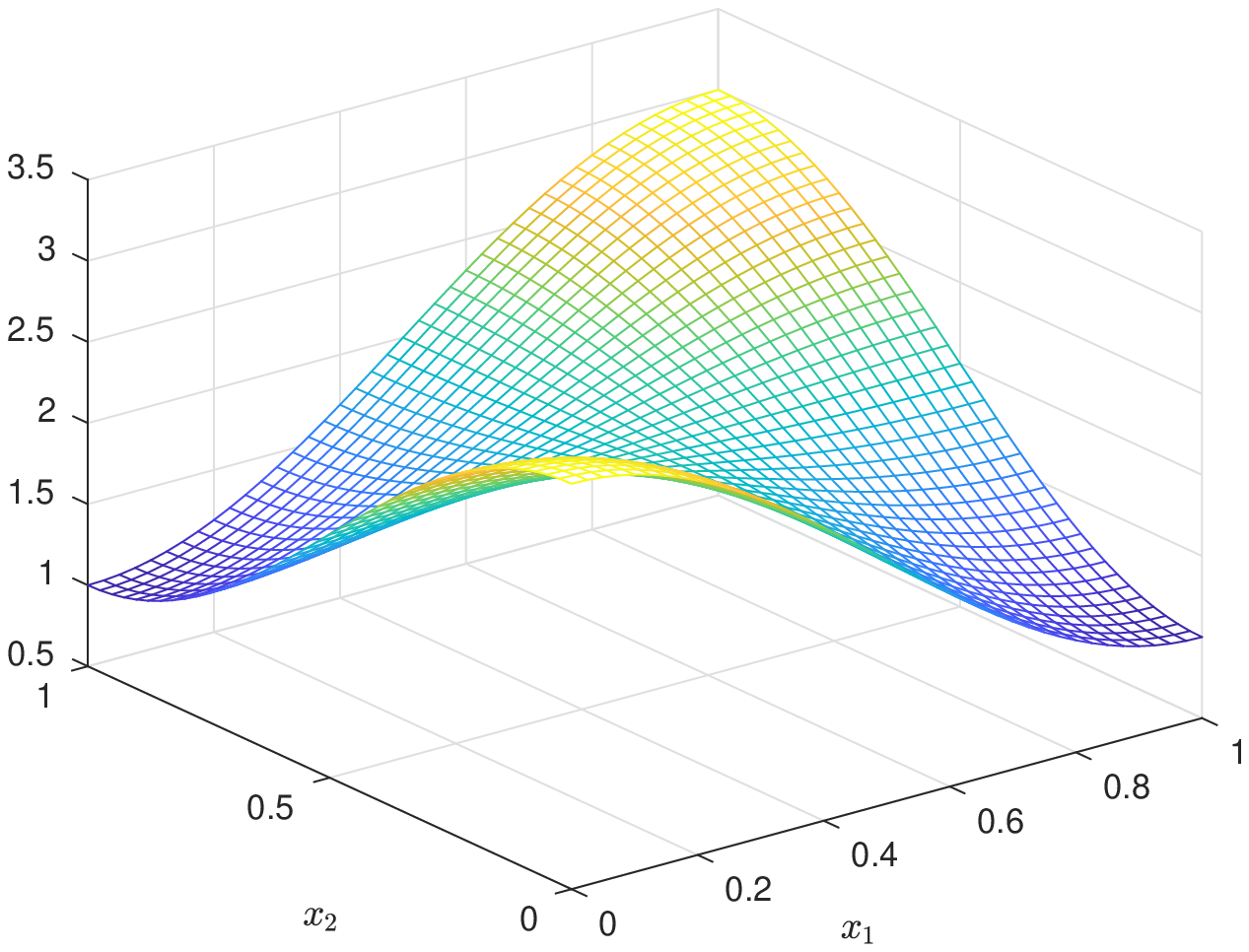}
\caption{$g_{\text{true}}=\cos(\pi x_1)\cos(\pi x_2)+2$}
\end{subfigure}\quad
\begin{subfigure}[p]{.30\textwidth}
\includegraphics[width=.2\textheight]{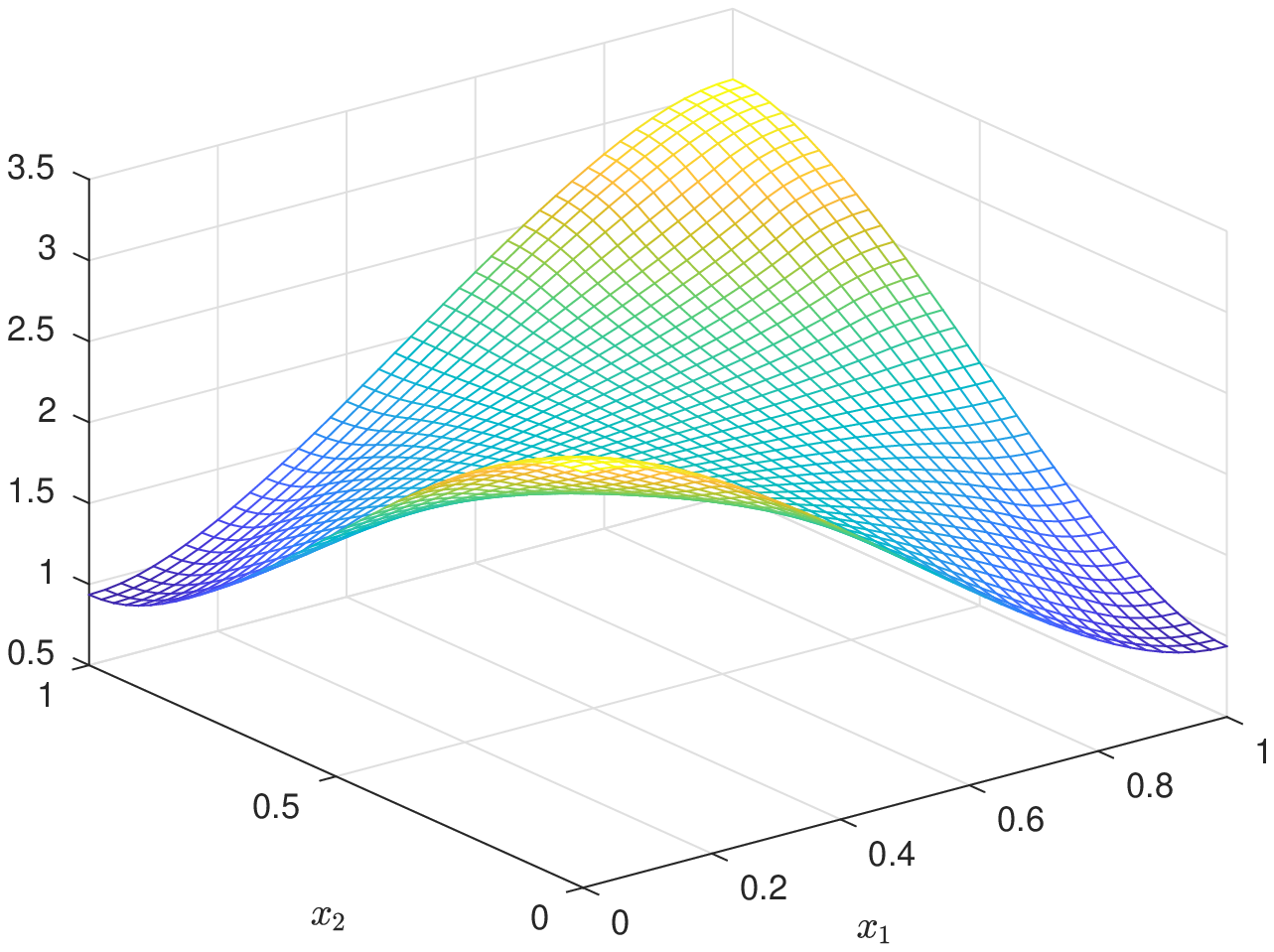}
\caption{$\alpha=1.2$}
\label{fig:inv_saddle_1.2}
\end{subfigure}\quad
\begin{subfigure}[p]{.30\textwidth}
\includegraphics[width=.2\textheight]{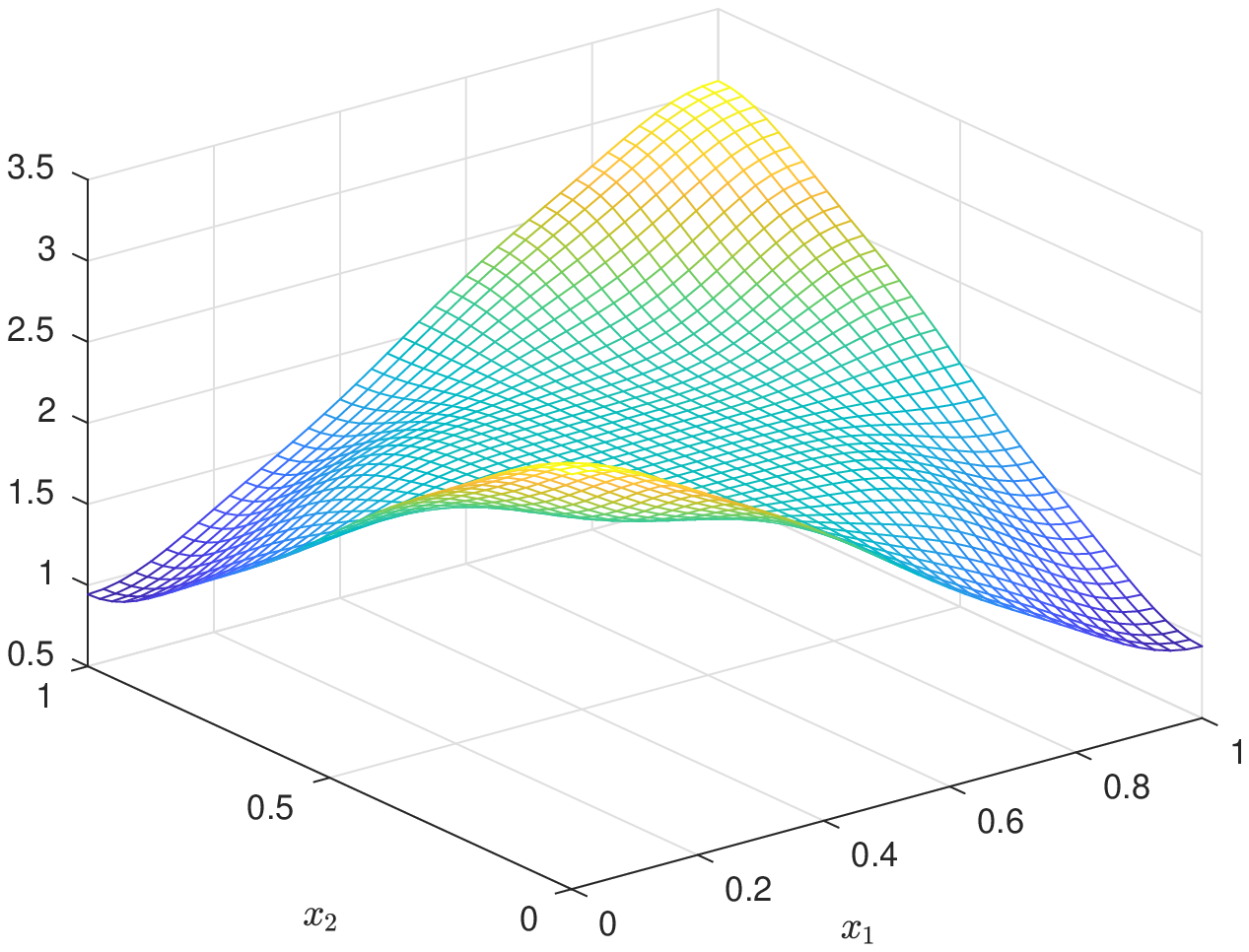}
\caption{$\alpha=1.8$}
\label{fig:inv_saddle_1.8}
\end{subfigure}
\caption{True solutions (left), reconstructions for $\alpha=1.2$ (middle) and reconstructions for $\alpha=1.8$ (right).
Figure (\ref{fig:inv_plane_1.2})  corresponds to $K=91,  Res=2.71\%$; 
Figure (\ref{fig:inv_plane_1.8})  corresponds to $K=139, Res=5.77\%$;
Figure (\ref{fig:inv_saddle_1.2}) corresponds to $K=113, Res=3.65\%$;
Figure (\ref{fig:inv_saddle_1.8}) corresponds to $K=166, Res=7.07\%$.}
\label{fig:alpha}
\end{figure}
\end{example}

\begin{example}
In this example we fix $\alpha=1.5$, $\delta = 10\%$, $\varepsilon=10^{-3}$ and
$$g_{\text{true}} = \cos(\pi x_1)\cos(\pi x_2)+2.$$
We study the effect of observation regions to reconstructed $g(x_1,x_2)$ by
choosing six different observation regions:
$$\begin{array}{lll}
\Omega' = \overbar{\Omega}\setminus(0.2,0.8)^2,&\quad& \Omega' = \overbar{\Omega}\setminus(0.05,0.95)^2,\\
\Omega' = \overbar{\Omega}\setminus[0,0.8)^2,&\quad&   \Omega' = \overbar{\Omega}\setminus[0,0.95)^2,\\
\Omega' = \overbar{\Omega}\setminus[0,1]\times[0,0.8),&\quad& \Omega' = \overbar{\Omega}\setminus[0,1]\times[0,0.95).
\end{array}$$

Figure (\ref{fig:inv_region}) shows reconstructed results with different observation regions and
Table (\ref{tab:inv_region}) lists the number of steps and relative errors.
With the increasing of the observation region, the reconstructed result becomes more accurate.
If we are lack of observation near some boundaries, it is hard to obtain a good reconstruction near those boundaries.
\begin{figure}[htb]
\centering
\begin{subfigure}[p]{.30\textwidth}
\includegraphics[width=.2\textheight]{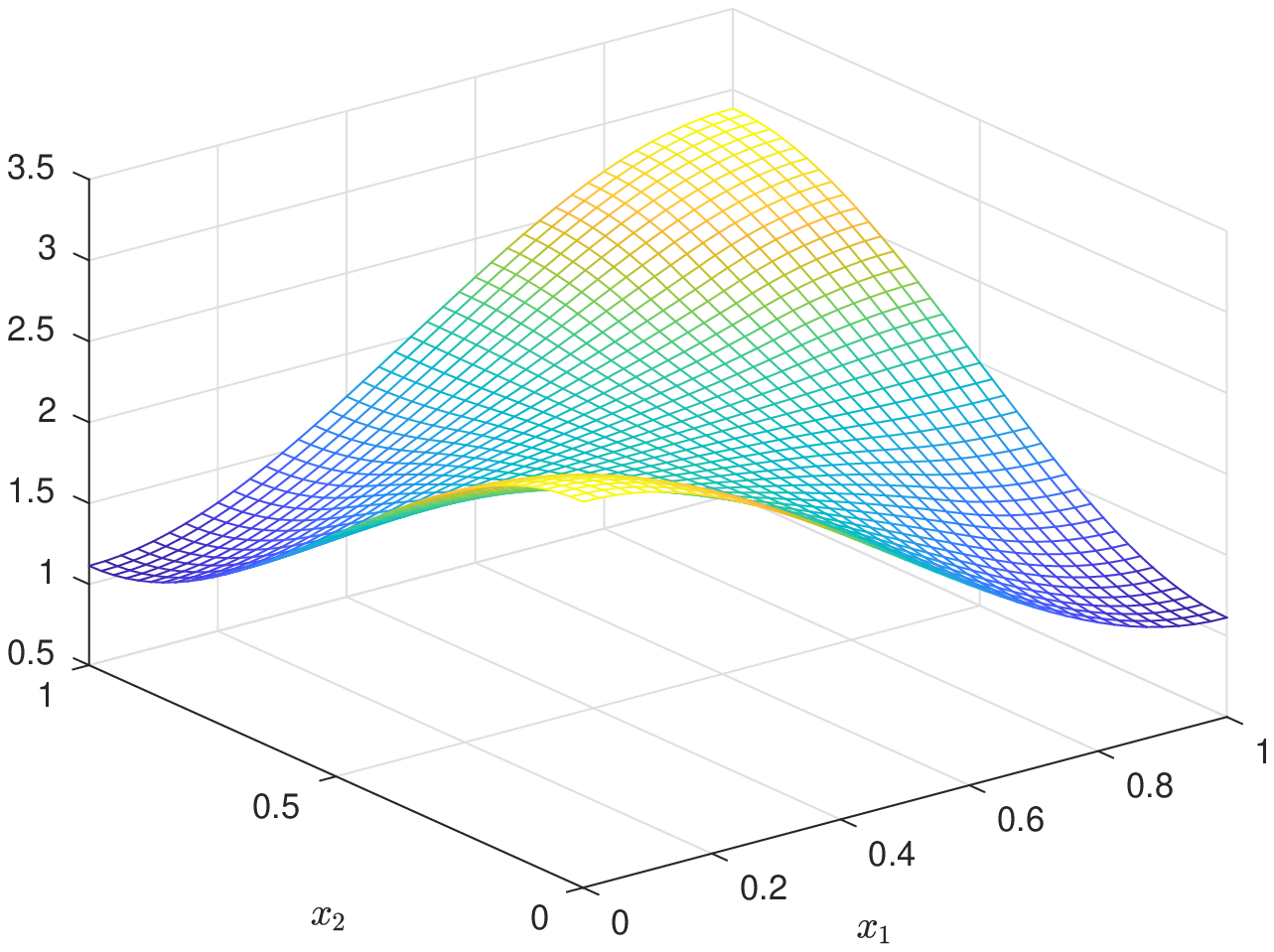}
\caption{$\Omega' = \overbar{\Omega}\setminus(0.2,0.8)^2$}
\end{subfigure}\quad
\begin{subfigure}[p]{.30\textwidth}
\includegraphics[width=.2\textheight]{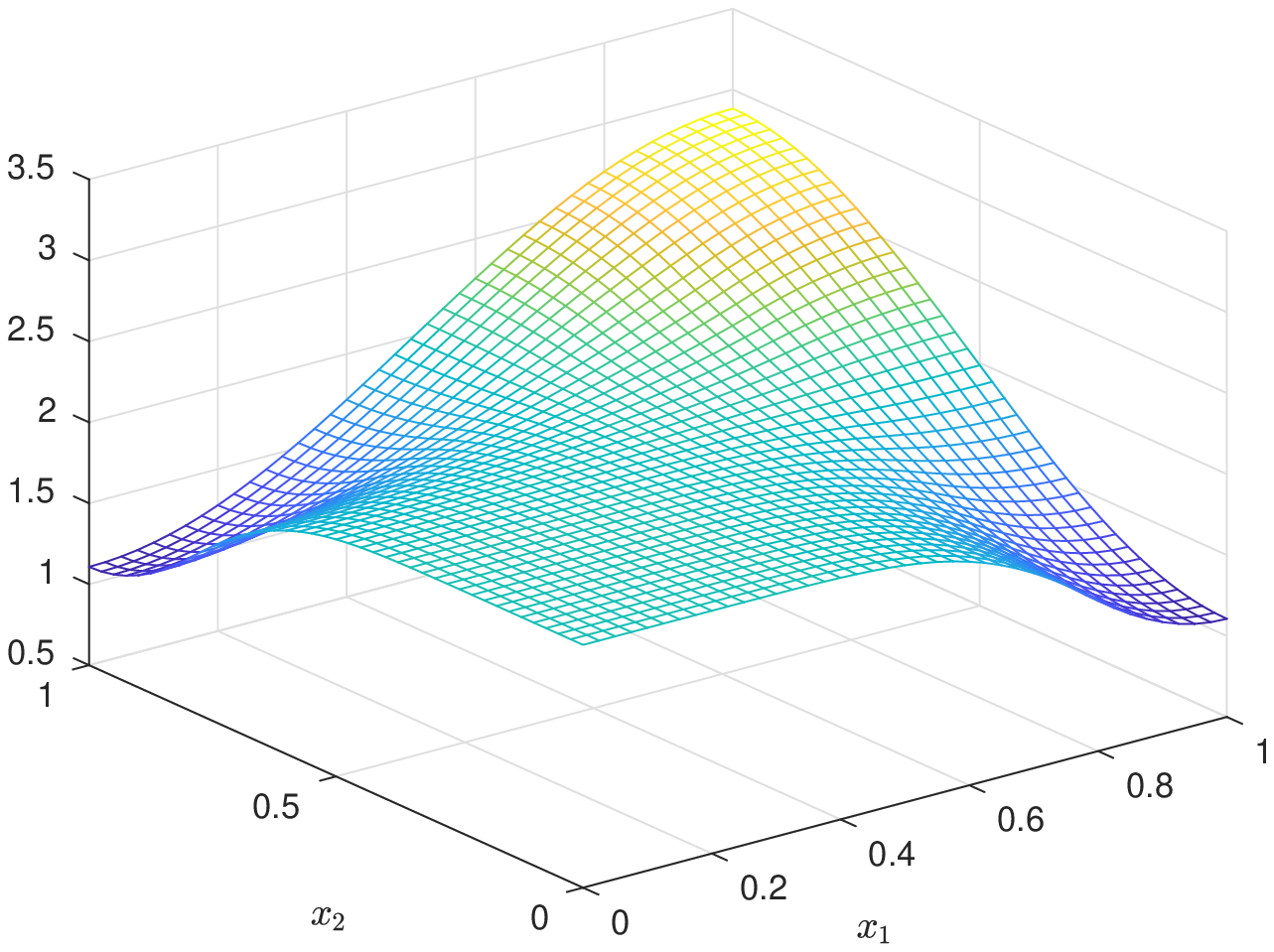}
\caption{$\Omega' = \overbar{\Omega}\setminus[0,0.8)^2$}
\end{subfigure}\quad
\begin{subfigure}[p]{.30\textwidth}
\includegraphics[width=.2\textheight]{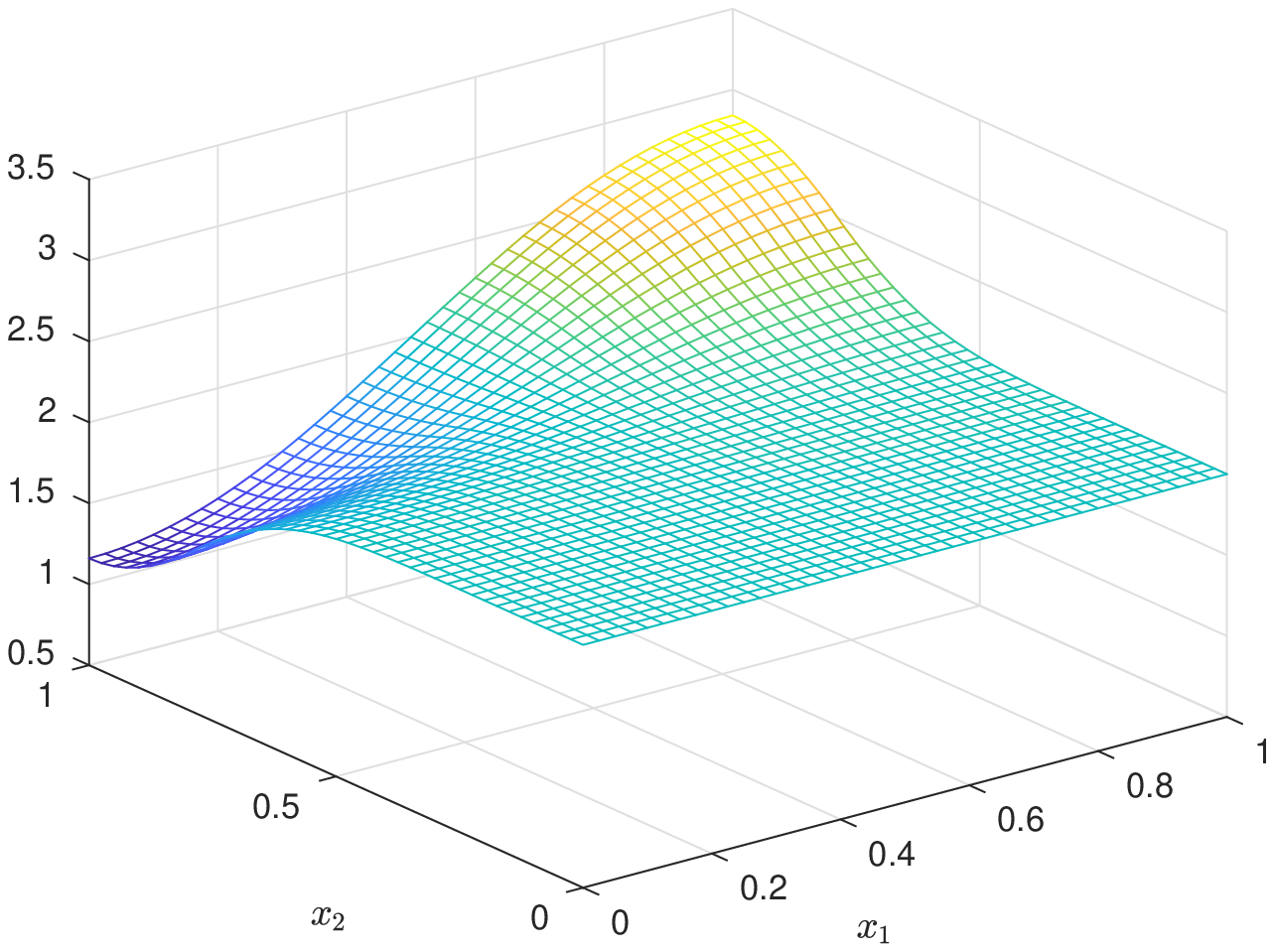}
\caption{$\Omega' = \overbar{\Omega}\setminus[0,1]\times[0,0.8)$}
\end{subfigure}

\begin{subfigure}[p]{.30\textwidth}
\includegraphics[width=.2\textheight]{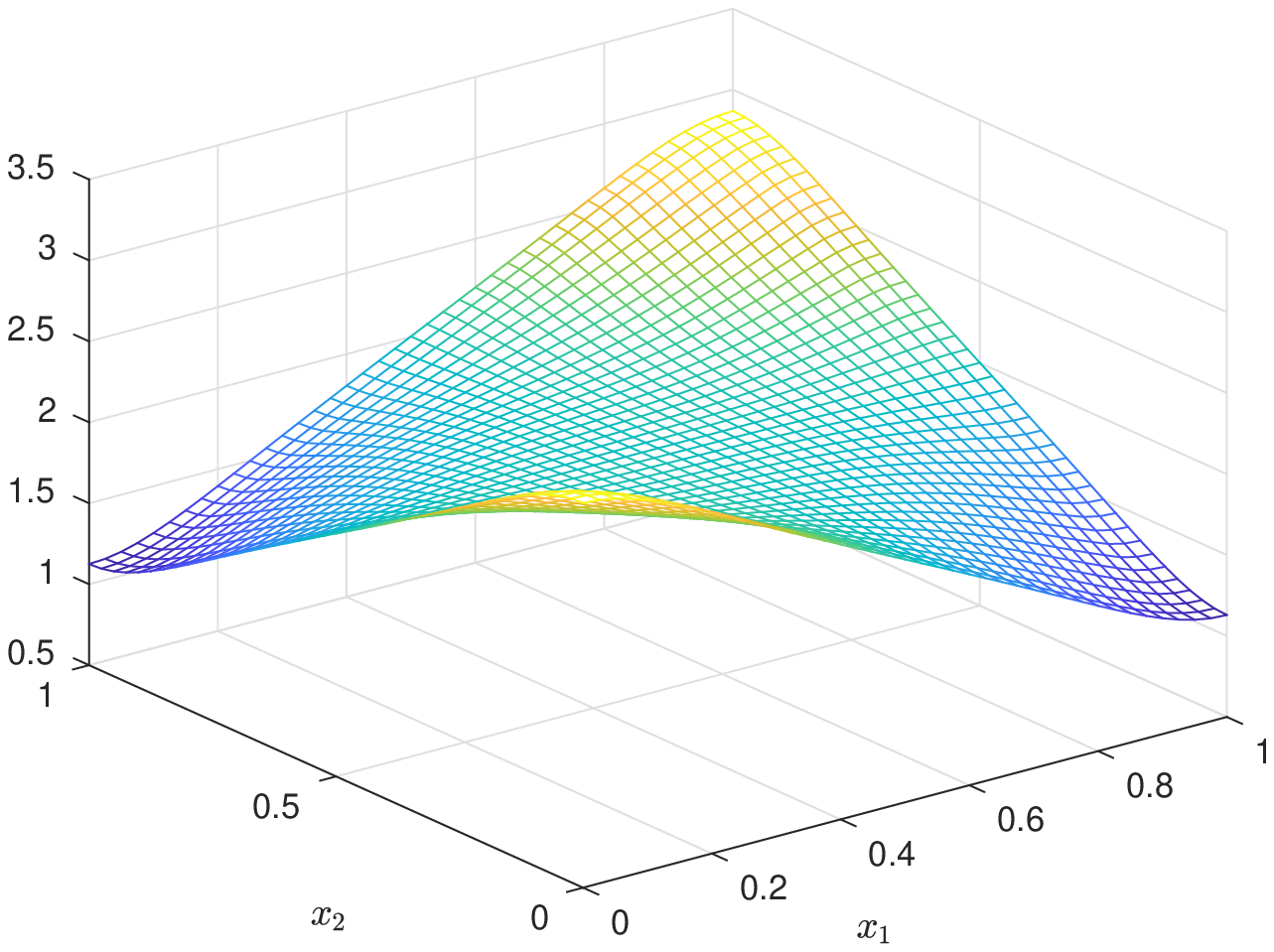}
\caption{$\Omega' = \overbar{\Omega}\setminus(0.05,0.95)^2$}
\end{subfigure}\quad
\begin{subfigure}[p]{.30\textwidth}
\includegraphics[width=.2\textheight]{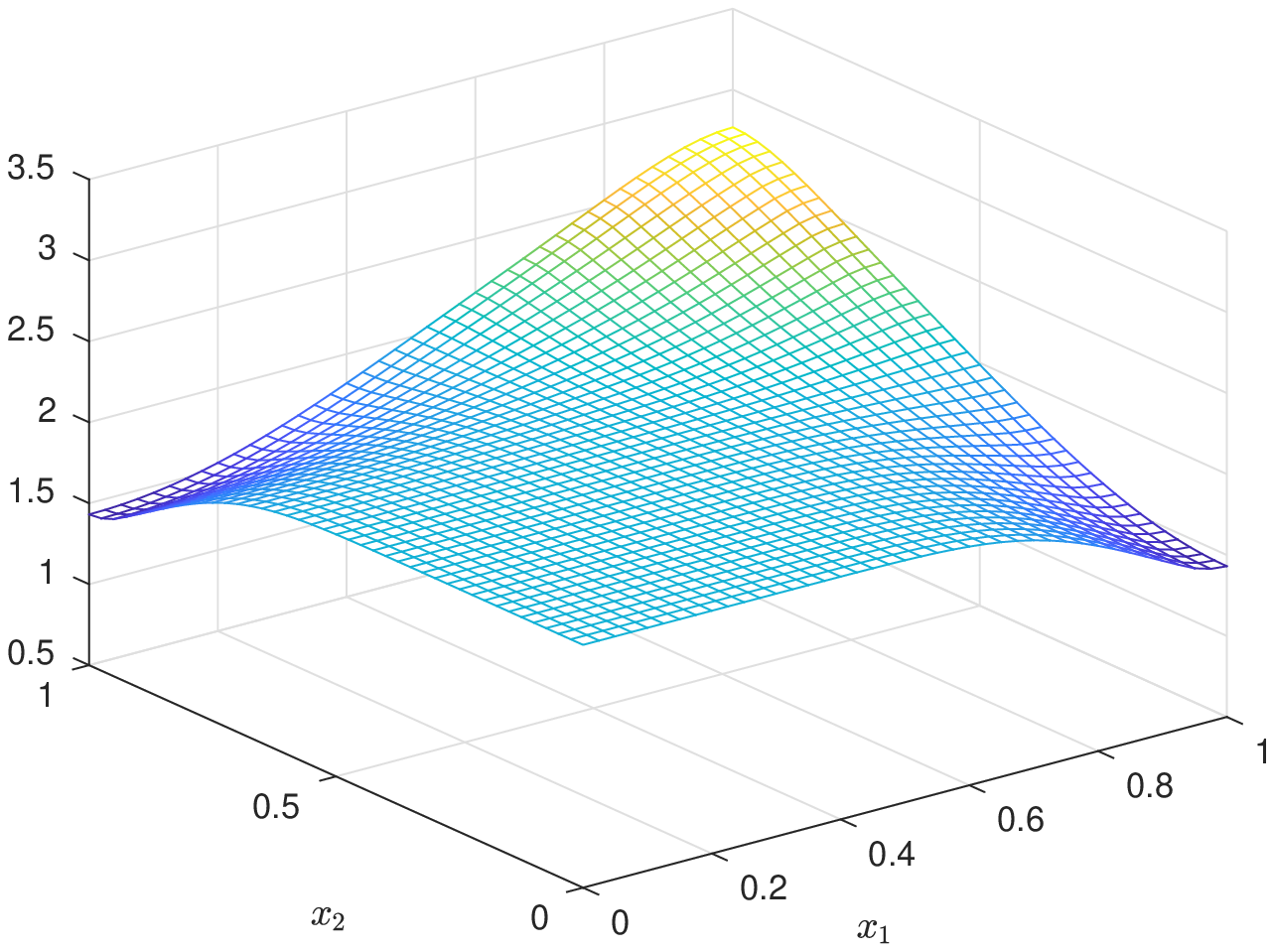}
\caption{$\Omega' = \overbar{\Omega}\setminus[0,0.95)^2$}
\end{subfigure}\quad
\begin{subfigure}[p]{.30\textwidth}
\includegraphics[width=.2\textheight]{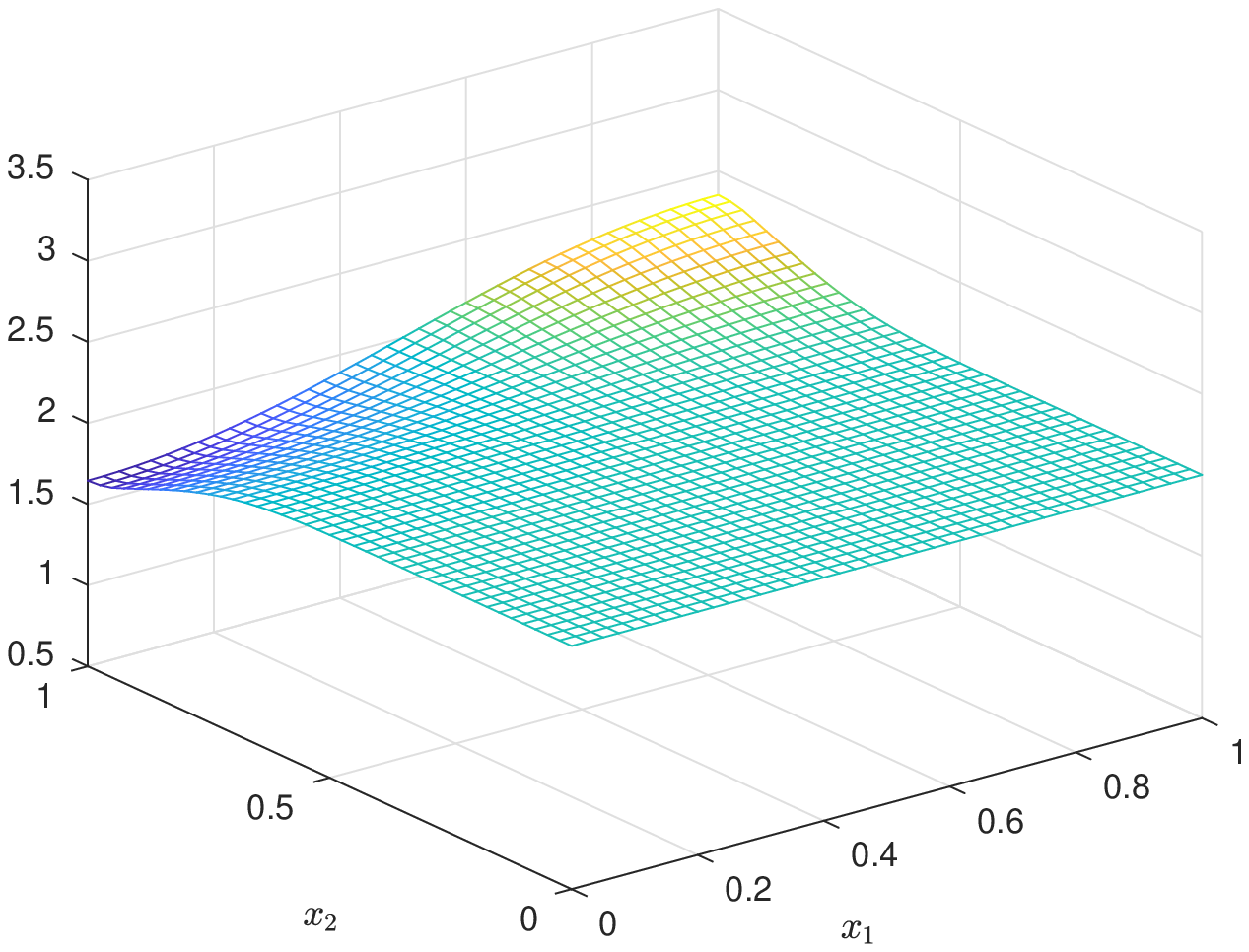}
\caption{$\Omega' = \overbar{\Omega}\setminus[0,1]\times[0,0.95)$}
\end{subfigure}
\caption{Effect of observation regions to reconstructed results. }
\label{fig:inv_region}
\end{figure}
\begin{table}
\centering
\begin{tabular}{lcc} 
	\toprule
	\hline 
	$\Omega'$ & $K$ & $Res$ \\ \hline
	$\overbar{\Omega}\setminus(0.2,0.8)^2$         & $73$ & $3.95\%$ \\
	$\overbar{\Omega}\setminus(0.05,0.95)^2$       & $92$ & $9.09\%$ \\ \midrule
	$\overbar{\Omega}\setminus[0,0.8)^2$           & $72$ & $13.54\%$\\
	$\overbar{\Omega}\setminus[0,0.95)^2$          & $73$ & $17.49\%$\\ \midrule
	$\overbar{\Omega}\setminus[0,1]\times[0,0.8)$  & $63$ & $18.42\%$\\
	$\overbar{\Omega}\setminus[0,1]\times[0,0.95)$ & $40$ & $22.09\%$\\
	\hline
	\bottomrule
\end{tabular}
\caption{Number of steps and relative errors for different regions of observation.}
\label{tab:inv_region}
\end{table}
\end{example}

\begin{example}
\label{exa:extreme}

In this example, we study two extreme cases, $\alpha=1.01$ and $\alpha=1.99$, 
and compare them with $\alpha=0.99$ \cite{JLLY} and $\alpha=2$ \cite{JLY}, respectively.
We choose
$$\delta = 4\%,\quad g_{\text{true}} = \cos(\pi x_1)\cos(\pi x_2)+2,$$
$\varepsilon=\tfrac{\delta}{200}$ for $\alpha\approx 1$ and 
$\varepsilon=\tfrac{\delta}{1500}$ for $\alpha\approx 2$.

We obtain $Res=4.14\%, K=162$ for $\alpha=0.99$ and $Res=3.50\%, K=143$ for $\alpha=1.01$.
Although we choose a smaller threshold $\varepsilon$, 
we still have $Res=5.93\%, K=601$ for $\alpha=1.99$ and $Res=6.81\%, K=531$ for $\alpha=2$.
We can see that in these two extreme cases, $\alpha = 1.01$ and $\alpha = 1.99$, our results are compatible
with cases $\alpha=0.99$ and $\alpha=2$, respectively.

We have less accurate results for $\alpha\approx 2$
than for $\alpha\approx 1$ even if we use a much smaller stopping threshold, which can also
been seen from Figure \ref{fig:inv_alpha}. This is because of the small stopping time and the finite wave speed,
i.e., our observation stops before the full source propagating onto the boundary.
By augmenting $T$ from 1 to 4, the result can be improved dramatically as demonstrated in the last line of Figure \ref{fig:inv_alpha}
with $Res=0.10\%, K=73$ for $\alpha=1.99$ and $Res=0.20\%, K=70$ for $\alpha=2$.
It seems that the wave propagation seems to filter out the random noise, 
which results in a surprisingly small error compared to the noise level.

\begin{figure}[htb]
\centering
\begin{subfigure}[p]{.30\textwidth}
\includegraphics[width=.2\textheight]{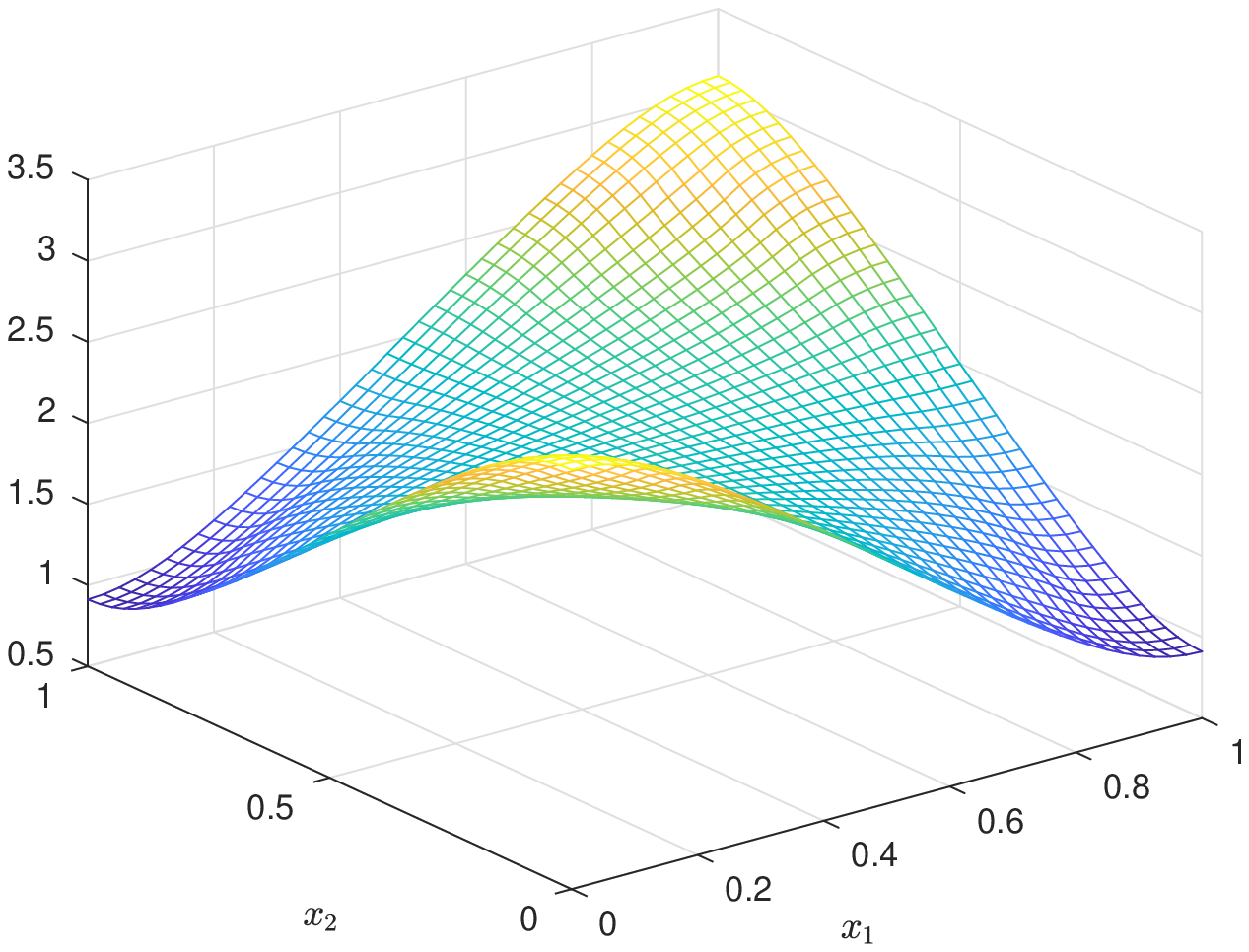}
\caption{$\alpha=0.99, T=1$}
\end{subfigure}\quad
\begin{subfigure}[p]{.30\textwidth}
\includegraphics[width=.2\textheight]{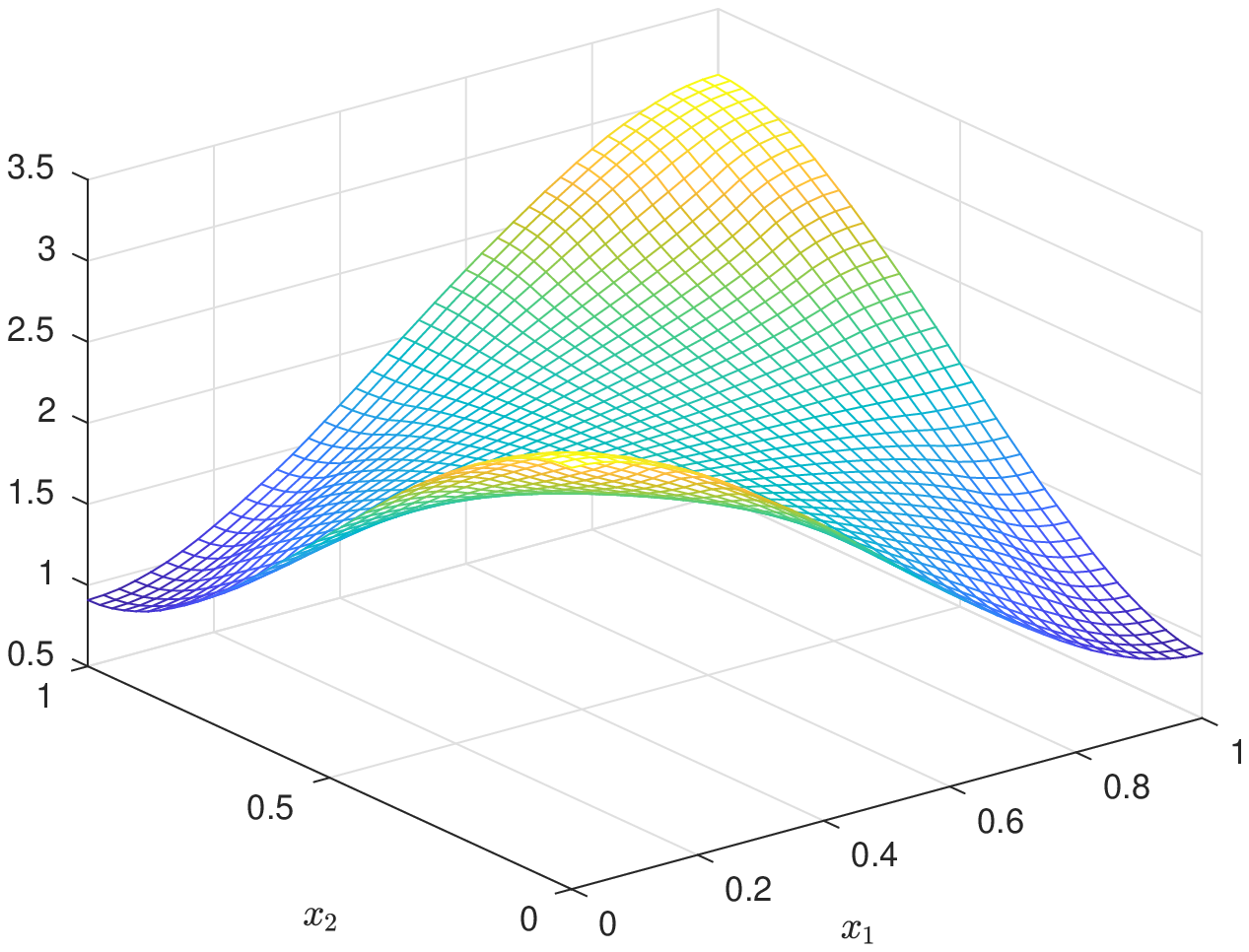}
\caption{$\alpha=1.01, T=1$}
\end{subfigure}

\begin{subfigure}[p]{.30\textwidth}
\includegraphics[width=.2\textheight]{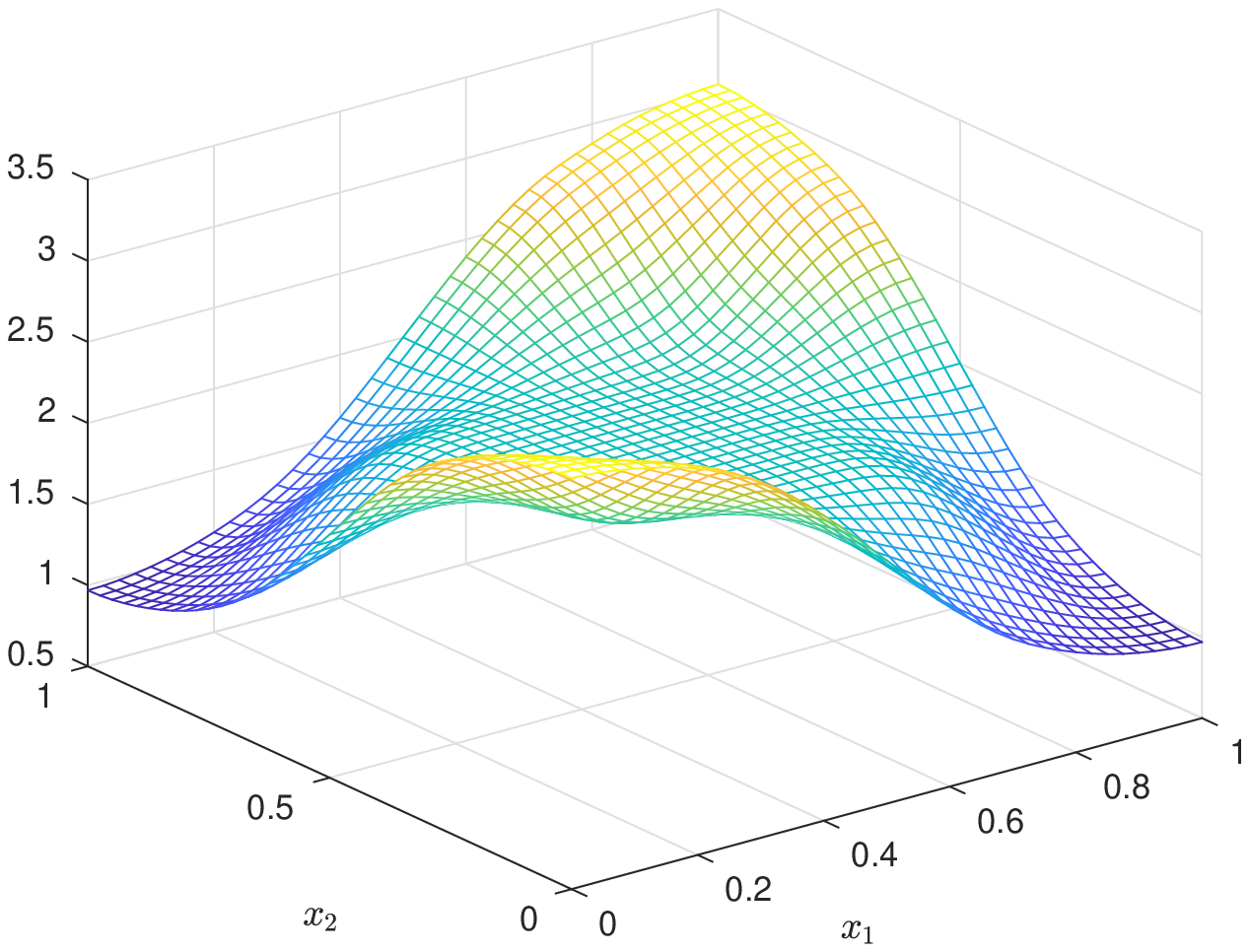}
\caption{$\alpha=1.99, T=1$}
\end{subfigure}\quad
\begin{subfigure}[p]{.30\textwidth}
\includegraphics[width=.2\textheight]{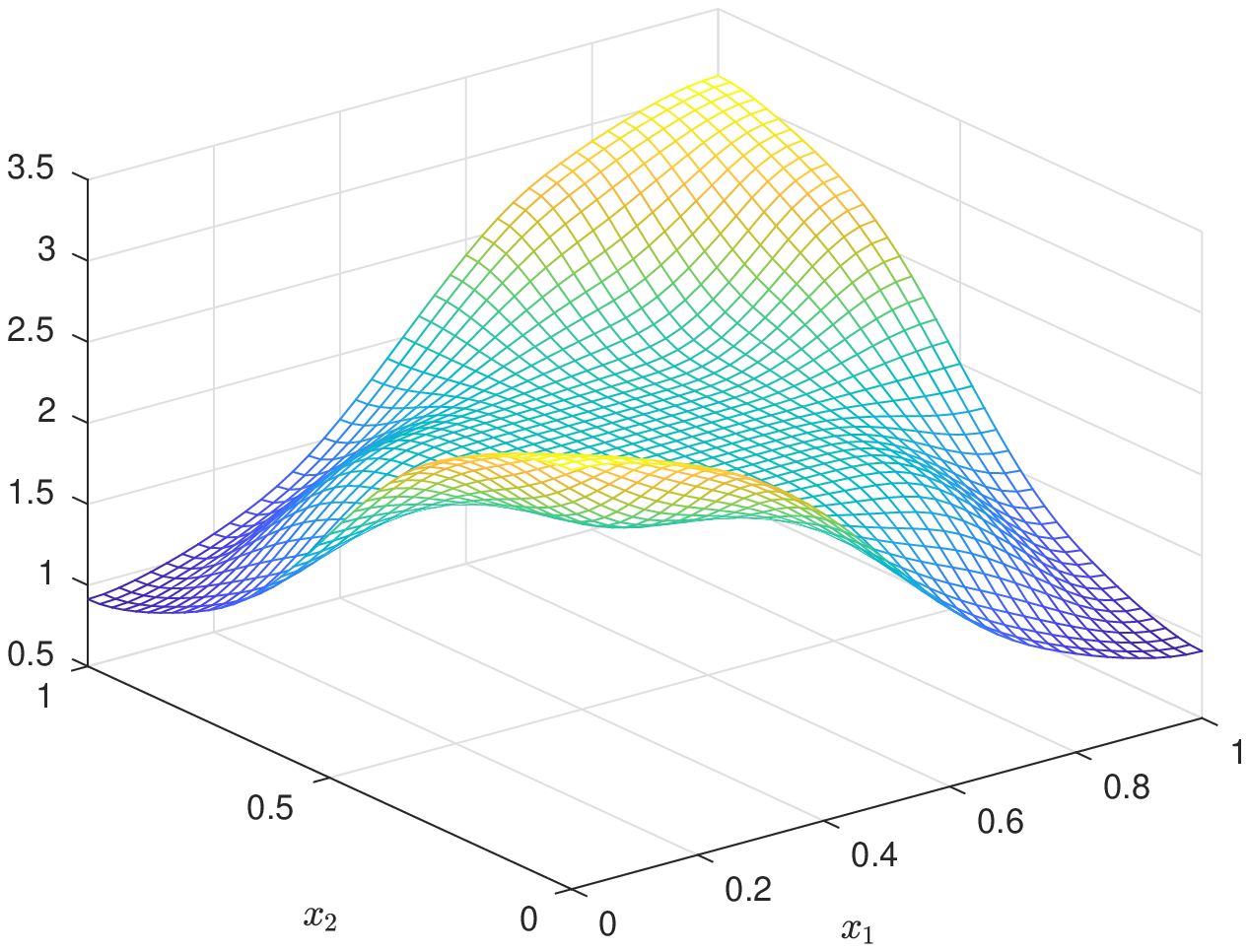}
\caption{$\alpha=2, T=1$}
\end{subfigure}

\begin{subfigure}[p]{.30\textwidth}
\includegraphics[width=.2\textheight]{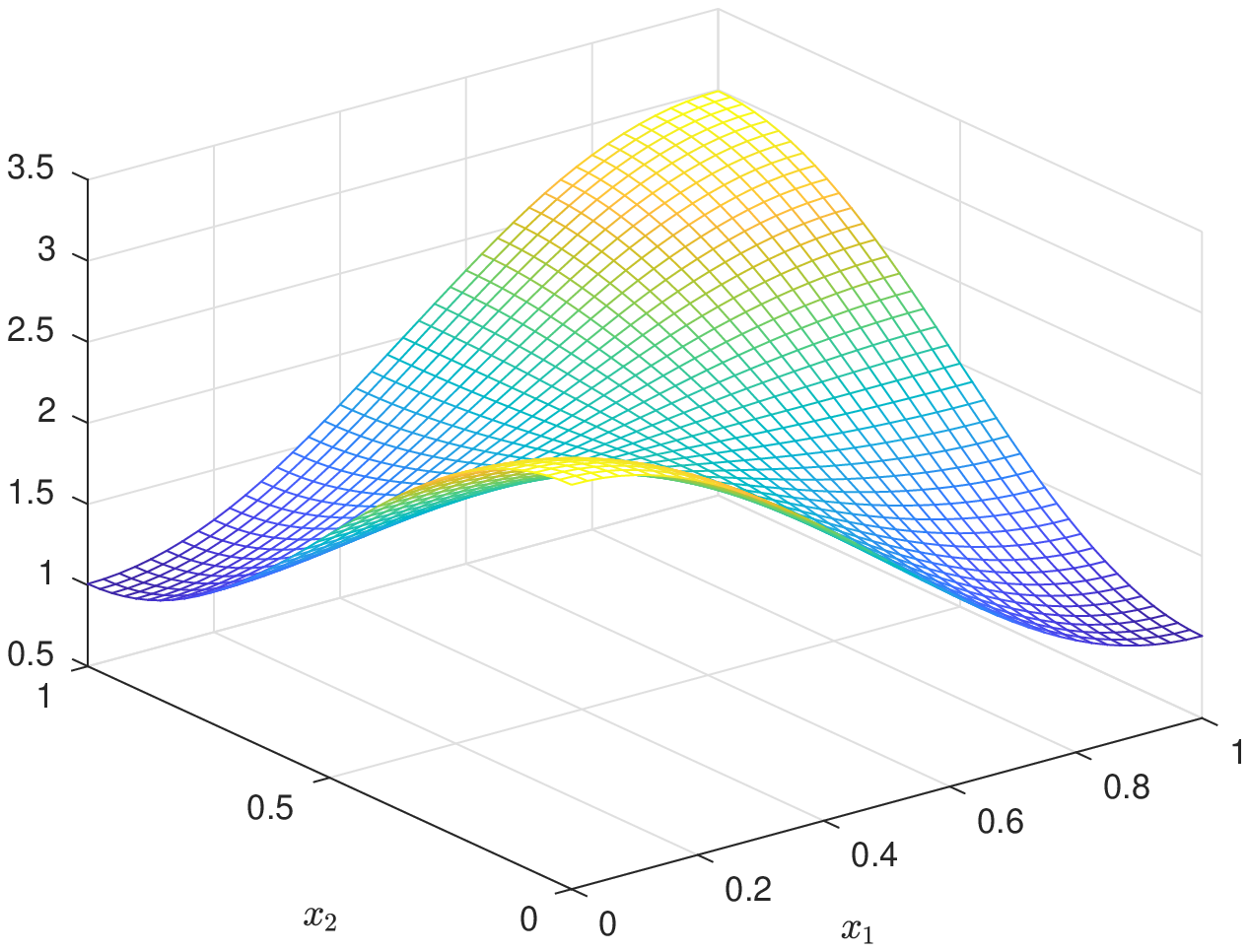}
\caption{$\alpha=1.99, T=4$}
\end{subfigure}\quad
\begin{subfigure}[p]{.30\textwidth}
\includegraphics[width=.2\textheight]{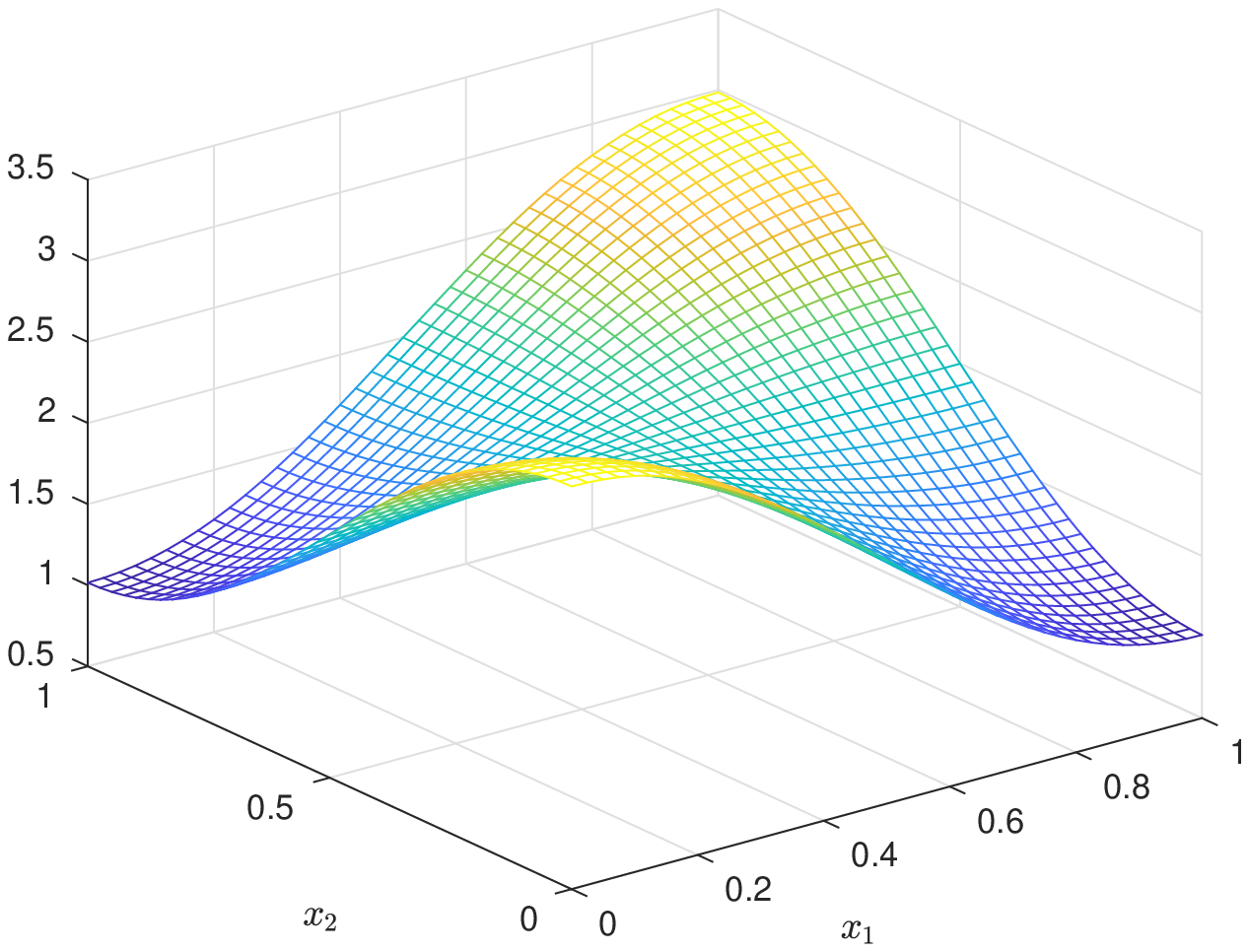}
\caption{$\alpha=2, T=4$}
\end{subfigure}

\caption{Comparison of reconstructed results for $\alpha\approx 1$ and for $\alpha\approx 2$.}
\label{fig:inv_alpha}
\end{figure}
\end{example}

%%%%%%%%%%%%%%%%%%%%%%%%%%%%%%%%%%%%%%%%%%%%%%%%%%%%%%%%%%%%%
%%%%%%%%%%%%%%%%%%%%%%%%%%%%%%%%%%%%%%%%%%%%%%%%%%%%%%%%%%%%%
%%%%%%%%%%%%%%                                                                                                   %%%%%%%%%%%%%%%%
%%%%%%%%%%%%%%                                        (IP1)'                                           %%%%%%%%%%%%%%%%
%%%%%%%%%%%%%%                                                                                                   %%%%%%%%%%%%%%%%
%%%%%%%%%%%%%%%%%%%%%%%%%%%%%%%%%%%%%%%%%%%%%%%%%%%%%%%%%%%%%
%%%%%%%%%%%%%%%%%%%%%%%%%%%%%%%%%%%%%%%%%%%%%%%%%%%%%%%%%%%%%

\section{Analysis of inverse problem (IP1')}
\label{sec-IP1prime} 
 The study of inverse source problems carried out in the preceding sections was concerned with unknown source terms in either of the two forms \eqref{source1} or \eqref{source2}. However, this analysis does not apply to source terms in the form of
\eqref{source1p}
 and we shall see in the case $\alpha=1$ that $f$ is not uniquely determined by partial observation of the solution to \eqref{eq1}. This is not surprising given the obstruction to identification of general time-dependent source terms by partial data, exhibited in the Appendix, but it turns out that it can be further described for source terms in the form of \eqref{source1p}.

\subsection{Statement of the results}

 Set $\alpha=1$. Given a suitable internal boundary observation of the solution to \eqref{eq1}, we aim to characterize all source terms
 $f(x,t)=\sigma(t) g(x) + \beta(t) h(x)$ in the form of \eqref{source1p} generating exactly the same data. To this purpose, assuming that the function $\beta$ does not change signs and that it is not-identically zero in $(0,T)$, we  introduce the operator $\int_0^T \beta (t) e^{A_{q,\rho} t} dt$,  where we recall that $A_{q,\rho}$ is defined in Section \ref{sec-Duhamel}. It is boundedly invertible in $L^2_\rho(\Omega)$ and we denote its inverse by $\left(\int_0^T \beta (t) e^{A_{q,\rho} t} dt \right)^{-1}$. Then, by the operatorial calculus, the following operator
%bel{H}
$$
H_{q,\rho} :=-\left(\int_0^T \beta (t) e^{A_{q,\rho} t} dt \right)^{-1} \left(\int_0^T \sigma (t) e^{A_{q,\rho} t}dt \right),
$$
%\ee
is self-adjoint in $L_\rho^2(\Omega)$.  

\begin{thm}
\label{t7} 
Let $\sigma \in L^2(0,T)$ and $\beta \in L^1(0,T)$ be supported in $[0,T)$. Assume further that $\beta$ is not-identically zero and does not change sign in $(0,T)$. Given $g$ and $h$ in $L^2(\Omega)$,  we denote by $u$ the solution to \eqref{eq1} associated with $\alpha=1$ and source term $f$ expressed by \eqref{source1p}. Then, for  an arbitrarily chosen non-empty open subset $\Omega' \subset \Omega$, we have the implication:
\bel{t7a}
 u_{|(0,T) \times \Omega'}=0  \Longrightarrow  h = \rho H_{q,\rho} \rho^{-1} g\ \mbox{in}\ \Omega.
\ee
\end{thm}

%Notice from the identity in the right-hand-side of \eqref{t7a} that $g \in \rho D(H_{q,\rho})$, where $D(H_{q,\rho})$ is the domain of the %operator $H_{q,\rho}$. 

Although Theorem \ref{t7} is interesting in its own right, the main  benefit of the above statement is the following characterization of the set of source terms expressed by \eqref{source1}.  We recall for $\ell \in \N$, that $H^\ell_0(0,T)$ denotes the closure of $\cC^\infty_0(0,T)$ in the $H^\ell(0,T)$-norm topology.

\begin{cor}
\label{c5} 
 For $\ell \in \N$ fixed, assume that $\beta \in H_0^\ell(0,T) \setminus \{ 0 \}$ is supported in $[0,T)$ and does not change sign in $(0,T)$. Suppose moreover that $\rho(x)=1$ for a.e. $x \in \Omega$. For $g$ and $h$ in $L^2(\Omega)$, let
$u_g$ denote the weak-solution to \eqref{eq1} associated with $f(t,x)=\frac{d^\ell \beta}{dt^\ell}(t)g(x)$, and let $u_h$ be the weak-solution to \eqref{eq1} with source term $f(t,x)=-\beta(t)h(x)$.
Then, for any non-empty open subset $\Omega' \subset \Omega$, we have: 
\bel{c1b}
 u_g =u_h\ \mbox{in}\ (0,T) \times \Omega' \Longrightarrow  h=(-1)^{\ell+1}A_q^\ell\ g\ \mbox{in}\ \Omega . 
\ee
Moreover,  if $\beta \in H_0^1(0,T)$, then we have
\bel{c1c}
g = h=0\ \mbox{in}\ \Omega'.
\ee
\end{cor}

\subsection{Proofs of Theorem \ref{t7} and Corollary \ref{c5}}

\paragraph{\bf Proof of Theorem \ref{t7}.} We argue as in the derivation of \eqref{t1dbis}  to obtain that the Laplace transform $U(p)$,  $p \in (0,+\infty)$, of the solution $u$ to
to \eqref{eq1} with source term $f$, given by \eqref{source1p}, solves
\bel{pc1}
\left\{ \begin{array}{rcll} 
U(p) & = &  \widehat{\sigma}(p) (A_{q,\rho}+ p)^{-1} \rho^{-1}  g  + \widehat{\beta}(p) (A_{q,\rho}+ p)^{-1} \rho^{-1}  h & \mbox{in}\ L_\rho^2(\Omega),\\
U(p) & = & 0 & \mbox{in}\ L_\rho^2(\Omega').
\end{array}
\right.
\ee 
Here we  use the notations of the proof of Theorem \ref{t1} and  we set $\widehat{\sigma}(p):=\int_0^Te^{-pt}\sigma(t)dt$ and $\widehat{\beta}(p):=\int_0^T e^{-pt} \beta(t)dt$ for $p \in \C$. From the spectral representation of the operator $A_{q,\rho}$, introduced in Section \ref{sec-Duhamel}, we infer from \eqref{pc1} that the identity
\bel{pc2}
\widehat{\sigma}(p) \sum_{n=1}^{+\infty} \frac{\sum_{k=1}^{m_n}  g_{n,k} \varphi_{n,k}}{\lambda_n+p} +
\widehat{\beta}(p) \sum_{n=1}^{+\infty}\frac{\sum_{k=1}^{m_n}  h_{n,k}\varphi_{n,k}}{\lambda_n+p} =0,
\ee
holds in $L_\rho^2(\Omega')$ for every $p \in (0,+\infty)$,  where $g_{n,k}:=\langle \rho^{-1}g,\varphi_{n,k}\rangle_{L_\rho^2(\Omega)}$ and
 $h_{n,k}:=\langle \rho^{-1}h,\varphi_{n,k}\rangle_{L_\rho^2(\Omega)}$. Moreover,  since $p \mapsto \widehat{\sigma}(p) \sum_{n=1}^{+\infty} \frac{\sum_{k=1}^{m_n}  g_{n,k} \varphi_{n,k}}{\lambda_n+p}$ and  $p \mapsto \widehat{\beta}(p) \sum_{n=1}^{+\infty} \frac{\sum_{k=1}^{m_n}  h_{n,k} \varphi_{n,k}}{\lambda_n+p}$ can be meromorphically continued to $\C \setminus \{-\lambda_n:\ n \in\N\}$,  we extend \eqref{pc2} meromorphically in $\C \setminus \{-\lambda_n:\ n \in\N\}$. Therefore, for each $N \in \N$,  multiplying \eqref{pc2} by $\lambda_N+p$ and sending 
 $p$ to $-\lambda_N$, we obtain 
$$ \sum_{k=1}^{m_N}  \left( \widehat{\sigma}(-\lambda_N) g_{N,k}+ \widehat{\beta}(-\lambda_N)  h_{N,k} \right) \varphi_{N,k}=0 $$
in $L_\rho^2(\Omega')$. 
Since the function $\beta$ is not identically zero and does not change sign in $(0,T)$, we have $\widehat{\beta}(-\lambda_N)\neq 0$, so the above line can be reformulated as
$$ \sum_{k=1}^{m_N} \left( h_{N,k}+ \frac{\widehat{\sigma}(-\lambda_N)}{\widehat{\beta}(-\lambda_N)}  g_{N,k} \right) \varphi_{N,k}=0,$$
the equality being understood in the $L_\rho^2(\Omega')$-sense.
Next, since the family $\{ \varphi_{N,k} :\ k=1,\ldots,m_N \}$ is linearly independent in $L_\rho^2(\Omega')$, by virtue of Step 4  in Section \ref{sec-t1pr}, we necessarily have
\bel{pc3}
\frac{\widehat{\sigma}(-\lambda_N)}{\widehat{\beta}(-\lambda_N)}  g_{N,k}=-h_{N,k},\ k=1,\ldots,m_N.
\ee
Now,  since \eqref{pc3} is valid for all $N \in \N$, it follows from the Parseval identity $\sum_{n=1}^{+\infty} \sum_{k=1}^{m_n} \abs{h_{n,k}}^2 = \norm{\rho^{-1} h}_{L_\rho^2(\Omega)}^2$ for $h \in L^2(\Omega)$, that
$$ \abs{\frac{\widehat{\sigma}(-\lambda_n)}{\widehat{\beta}(-\lambda_n)}}^2  \abs{g_{n,k}}^2 < \infty. $$
Therefore,  the operatorial calculus yields that $g$ lies in the domain of the operator of $ H_{q,\rho}$ and fulfills \eqref{t7a}.\qed

\paragraph{\bf Proof of Corollary \ref{c5}.}
 We set $\sigma:=\frac{d^\ell \beta}{dt^\ell}$.
Since $u:=u_g-u_h$ is a solution to the IBVP \eqref{eq1} with source term $f$ in the form of \eqref{source1p}, then we have $h=H_{q,1} g$ by Theorem \ref{t7}.  Hence
\bel{pc4}
-\widehat{\beta}(-\lambda_n)^{-1} \widehat{\frac{d^\ell \beta}{dt^\ell}}(-\lambda_n) \langle g,\varphi_{n,k} \rangle_{L^2(\Omega)}=\langle h,\varphi_{n,k} \rangle_{L^2(\Omega)},\ n\in \N,\ k=1,\ldots,m_n.
\ee
Moreover, we have
$\widehat{\frac{d^\ell \beta}{dt^\ell}}(p)=p^\ell  \widehat{\beta}(p)$ for each $p \in \R$,  by $\beta \in  H_0^\ell(0,T)$, 
 whence \eqref{pc4}  implies
$$
(-1)^{\ell+1}\lambda_n^\ell \langle g,\varphi_{n,k} \rangle_{L^2(\Omega)}= \langle h,\varphi_{n,k} \rangle_{L^2(\Omega)},\ n \in\N,\ k=1,\ldots,m_n.$$
This entails that $g\in D(A_q^\ell)$ verifies \eqref{c1b}. 

In the particular case where $\ell=1$, we have $h=-A_q g$, hence $u$ is a solution to the IBVP \eqref{eq1} with $\alpha=1$ and $f(t,x)=(\partial_t-\cA_q)\beta(t)g(x)$ for a.e. $(t,x) \in Q$.  Since $(t,x)\mapsto \beta(t)g(x)$ is a weak-solution to the exact same problem,  then $u(t,x) = \beta(t) g(x)$ in $Q$, by the uniqueness of the solution to \eqref{eq1}.  Thus \eqref{c1c} follows directly from this.\qed

%In this paper we have considered the recovery of source terms in the form of \eqref{source1} and \eqref{source2} in a %general framework. Our results were restricted by the obstruction to unique determination of time-dependent source %terms in Section \ref{sec-settings}. Nevertheless, in the case $\alpha=1$, we have been able to prove some extension of %our results to the determination of source terms in the form of \eqref{source1p}. The results that we obtain for such source %terms are restricted to $\alpha=1$ for some technical reason. A natural question would be "what hapen for this problem %when $\alpha\neq1$". We leave this question for some further investigation.

%%%%%%%%%%%%%%%%%%%%%%%%%%%%%%%%%%%%%%%%%%%%%%%%%%%%%%%%%%%%%
%%%%%%%%%%%%%%%%%%%%%%%%%%%%%%%%%%%%%%%%%%%%%%%%%%%%%%%%%%%%%
%%%%%%%%%%%%%%                                                                                                   %%%%%%%%%%%%%%%%
%%%%%%%%%%%%%%                                        Appendix                                            %%%%%%%%%%%%%%%%
%%%%%%%%%%%%%%                                                                                                   %%%%%%%%%%%%%%%%
%%%%%%%%%%%%%%%%%%%%%%%%%%%%%%%%%%%%%%%%%%%%%%%%%%%%%%%%%%%%%
%%%%%%%%%%%%%%%%%%%%%%%%%%%%%%%%%%%%%%%%%%%%%%%%%%%%%%%%%%%%%

\section{Appendix: A natural obstruction to identifiability}
\label{sec-app}
In this appendix we characterize the obstruction  to  the unique determination of time-dependent source terms $f$ in \eqref{eq1}, from either internal or lateral measurement of the weak solution $u$ to \eqref{eq1}. 

 Let $\Omega'$ satisfy $\overline{\Omega'} \subset \Omega$. Pick $u_0 \in \cC^\infty_0((0,T)\times (\Omega \setminus \overline{\Omega'})) \setminus \{ 0 \}$ and set
$$ \tilde{u}_0(t,x) := \left\{ \begin{array}{cl} u_0(t,x) & \mbox{if}\ (t,x) \in (0,T) \times (\Omega \setminus \overline{\Omega'}), \\ 0 & \mbox{if}\ (t,x) \in (0,T) \times \overline{\Omega'}. \end{array} \right. $$
We consider the IBVP \eqref{eq1} with source term $f_0:=\rho\partial_t^\alpha \tilde{u}_0 - \mathcal A_q \tilde{u}_0$.
Evidently, $\tilde{u}_0$ is a weak solution to \eqref{eq1}, hence $u=\tilde{u}_0$ from the uniqueness of the solution to \eqref{eq1} with $f=f_0$.
Moreover, since $\tilde{u}_0$ is not identically zero in $Q$, then the same is true for $f_0$ (otherwise $\tilde{u}_0$ would be zero everywhere by uniqueness of the solution to \eqref{eq1},  which in contradiction with the definition of $u_0$).
Thus, we have $u_{|(0,T) \times \Omega'}=0$, despite of the fact that $f_0$ is not identically zero in $Q$. 

This establishes that the recovery of the unknown source term $f$ by partial knowledge of $u$, is completely hopeless, or, otherwise stated, that full knowledge of the solution $u$ to \eqref{eq1} (i.e. measurement of $u$ performed on the entire time-space cylinder $Q$) is needed for determining $f \in L^1(0,T;L^2(\Omega))$.

\section*{ Acknowledgments}
The work of the first, second and third authors is partially supported by the Agence Nationale de la Recherche (project MultiOnde) under  grant ANR-17-CE40-0029.
The fourth author is supported by Grant-in-Aid for Scientific Research (S) 
15H05740 of Japan Society for the Promotion of Science and
by The National Natural Science Foundation of China 
(no. 11771270, 91730303), and the ``RUDN University Program 5-100".

%%%%%%%%%%%%%%%%%%%%%%%%%%%%%%%%%%%%%%%%%%%%%%%%%%%%%%%%%%%%%%%%%
%%%%%%%%%%%%%%%%%%%%%%%%%%%%%%%%%%%%%%%%%%%%%%%%%%%%%%%%%%%%%%%%%
%%%%%%%%%%%%%%%%%                                  Bibliography                             %%%%%%%%%%%%%%%%%%%%%%
%%%%%%%%%%%%%%%%%%%%%%%%%%%%%%%%%%%%%%%%%%%%%%%%%%%%%%%%%%%%%%%%%
%%%%%%%%%%%%%%%%%%%%%%%%%%%%%%%%%%%%%%%%%%%%%%%%%%%%%%%%%%%%%%%%%
 
\end{document}